\documentclass{amsart}

\usepackage{graphicx} %
\usepackage{hyperref}
\usepackage{amsmath}
\usepackage{amssymb}
\usepackage{mathtools}
\usepackage{xcolor}
\usepackage{cleveref}
\usepackage{amsthm}
\usepackage{url}
\usepackage{hyperref}

\definecolor{linkred}{rgb}{0.5,0,0}
\hypersetup{colorlinks=true, linkcolor=linkred, citecolor=linkred, urlcolor=linkred}
\usepackage{xspace}
\usepackage{bm}
\usepackage{bbm}
\usepackage{standalone}
\usepackage{tikz} %
\usetikzlibrary{arrows.meta}
\usepackage{subcaption}

\newcommand{\includeexternalstandalonewithpath}[2][]{%
    \includegraphics[#1]{figures/#2.pdf}%
}

\newtheorem{theorem}{Theorem}[section]
\newtheorem{lemma}[theorem]{Lemma}
\newtheorem{corollary}[theorem]{Corollary}

\theoremstyle{definition}

\newtheorem{assumption}[theorem]{Assumption}

\theoremstyle{remark}
\newtheorem{remark}[theorem]{Remark}

\numberwithin{equation}{section}

\newcommand{\CN}{Crank--Nicolson\xspace}

\def\R{\mathbb{R}}

\def\calM{\mathcal{M}}

\def\calZ{\mathcal{Z}}
\def\eps{\varepsilon}

\def\ubold{\boldsymbol{u}}
\def\vbold{\boldsymbol{v}}
\def\wbold{\boldsymbol{w}}
\def\sigmabold{\boldsymbol{\sigma}}
\def\lambdabold{\boldsymbol{\lambda}}

\def\vboldh{\boldsymbol{v}_h}
\def\wboldh{\boldsymbol{w}_h}
\def\sigmaboldh{\boldsymbol{\sigma}_h}
\def\lambdaboldh{\boldsymbol{\lambda}_h}

\newcommand{\fullspace}{\calZ}

\newcommand{\wavevar}{x}
\newcommand{\wavevarOne}{\wavevar_1}
\newcommand{\wavevarTwo}{\wavevar_2}
\newcommand{\wavetestvar}{y}

\newcommand{\wavevarOnepar}{\wavevar_{1,\parallel}}
\newcommand{\wavevarOneparE}{E_\parallel}

\newcommand{\wavevarn}[1]{\wavevar^{#1}}
\newcommand{\wavevarnOne}[1]{\wavevarOne^{#1}}
\newcommand{\wavevarnTwo}[1]{\wavevarTwo^{#1}}

\newcommand{\wavevarh}{\wavevar_h}
\newcommand{\wavevarhOne}{\wavevar_{h,1}}
\newcommand{\wavevarhTwo}{\wavevar_{h,2}}
\newcommand{\wavetestvarh}{\wavetestvar_h}

\newcommand{\wavevarhn}[1]{\wavevarh^{#1}}
\newcommand{\wavevarhnOne}[1]{\wavevar_{h,1}^{#1}}
\newcommand{\wavevarhnTwo}[1]{\wavevar_{h,2}^{#1}}
\newcommand{\wavevarhOneDotStar}[1]{\dot{\wavevar}_{1,h}^{\star,#1}}

\newcommand{\waveerrh}{e_h}
\newcommand{\waveerrhOne}{e_{h,1}}
\newcommand{\waveerrhTwo}{e_{h,2}}
\newcommand{\waveerrhn}[1]{\waveerrh^{#1}}
\newcommand{\waveerrhnOne}[1]{\waveerrhOne^{#1}}
\newcommand{\waveerrhnTwo}[1]{\waveerrhTwo^{#1}}

\newcommand{\waveop}{S}
\newcommand{\waveopOne}{\waveop_1}
\newcommand{\waveopTwo}{\waveop_2}

\newcommand{\waveoph}{\waveop_h}
\newcommand{\waveophskew}{\waveop_h^s}
\newcommand{\waveophOne}{\waveop_{h,1}}
\newcommand{\waveophTwo}{\waveop_{h,2}}

\newcommand{\wavespace}{X}
\newcommand{\wavespaceOne}{X_1}
\newcommand{\wavespaceTwo}{X_2}
\newcommand{\wavespaceh}{\wavespace_h}
\newcommand{\wavespacehOne}{X_{h,1}}
\newcommand{\wavespacehTwo}{X_{h,2}}

\newcommand{\wavespaceV}{Y}

\newcommand{\wavespaceVTwo}{Y_2}
\newcommand{\wavespaceVh}{\wavespaceV_h}

\newcommand{\ipwave}{p}
\newcommand{\ipwaveOne}{p_1}
\newcommand{\ipwaveTwo}{p_2}
\newcommand{\normwave}[1]{\|#1\|_{\ipwave}}
\newcommand{\normwaveOne}[1]{\|#1\|_{\ipwaveOne}}
\newcommand{\normwaveTwo}[1]{\|#1\|_{\ipwaveTwo}}

\newcommand{\ipwaveh}{\ipwave_h}

\newcommand{\ipwavehTwo}{p_{2,h}}

\newcommand{\sbilwave}{s}
\newcommand{\sbilwaveOne}{\sbilwave_1}
\newcommand{\sbilwaveTwo}{\sbilwave_2}

\newcommand{\sbilwaveext}{\widehat{s}}

\newcommand{\sbilwaveh}{\sbilwave_h}
\newcommand{\sbilwavehskew}{\sbilwave^{s}_{h}}
\newcommand{\sbilwavehpos}{\sbilwave^{p}_h}
\newcommand{\sbilwavehposOne}{\sbilwave^{p,1}_h}
\newcommand{\sbilwavehposTwo}{\sbilwave^{p,2}_h}

\newcommand{\waveophpos}{\waveop_h^p}
\newcommand{\waveophposOne}{\waveop_{h,1}^p}
\newcommand{\waveophposTwo}{\waveop_{h,2}^p}

\newcommand{\sbilwavehOne}{\sbilwave_{h,1}}
\newcommand{\sbilwavehTwo}{\sbilwave_{h,2}}

\newcommand{\heatvar}{\theta}
\newcommand{\heatop}{L}
\newcommand{\heatoph}{\heatop_h}

\newcommand{\heatvarn}[1]{\heatvar^{#1}}

\newcommand{\heatvarh}{\heatvar_h}
\newcommand{\heatvarhn}[1]{\heatvar_h^{#1}}

\newcommand{\heaterrh}{\epsilon_h}
\newcommand{\heaterrhn}[1]{\heaterrh^{#1}}

\newcommand{\heatspace}{H}
\newcommand{\heatspaceV}{V}
\newcommand{\heatspaceh}{\heatspace_h}
\newcommand{\heatspaceVh}{\heatspaceV_h}

\newcommand{\ipheat}{\widetilde{p}}
\newcommand{\normheat}[1]{\|#1\|_{\ipheat}}
\newcommand{\sbilheat}{\widetilde{s}}

\newcommand{\ipheath}{\ipheat_h}
\newcommand{\sbilheath}{\sbilheat_h}

\newcommand{\coupopfull}{\widetilde{B}}
\newcommand{\coupop}{B}

\newcommand{\coupopform}{b}
\newcommand{\coupopfullform}{\widetilde{\coupopform}}

\newcommand{\coupopformh}{\coupopform_h}
\newcommand{\coupopfullformh}{\coupopfullform_h}
\newcommand{\coupoph}{\coupop_h}

\newcommand{\coupopfullformext}{\widehat{\coupopform}}

\newcommand{\ipcol}[2]{\langle #1, #2\rangle_{\ipwave\times\ipheat}}

\newcommand{\normcol}[1]{\|#1\|_{\ipwave \times \ipheat}}
\newcommand{\stabnormwave}[1]{|#1|_{\sbilwavehpos}}

\newcommand{\tauhalf}{\frac{\tau}{2}}
\newcommand{\tausqf}{\frac{\tau^2}{4}}
\newcommand{\nhalf}{n + 1/2}
\newcommand{\tauCFL}{\tau_{\operatorname{CFL}}}

\newcommand{\makeoverline}[1]{%
  \expandafter\makeoverlineA#1\relax
}
\def\makeoverlineA#1#2\relax{%
  \overline{#1}#2%
}
\newcommand{\tauavg}[1]{\makeoverline{#1}}

\newcommand{\dtau}{\partial_\tau}
\newcommand{\ptau}{\dtau}

\newcommand{\colvar}{\mathcal{U}}
\newcommand{\colvarh}{\colvar_h}
\newcommand{\colvarhn}[1]{\colvar_h^{#1}}
\newcommand{\colvarn}[1]{\colvar^{#1}}

\newcommand{\coltestvar}{\mathcal{V}}
\newcommand{\coltestvarh}{\coltestvar_h}
\newcommand{\coltestvarhn}[1]{\coltestvar_h^{#1}}
\newcommand{\coltestvarhVTwo}{\mathcal{W}_h}

\newcommand{\colop}{\bm{\mathcal{A}}}
\newcommand{\coloph}{\colop_h}

\newcommand{\colstaboph}{\bm{\mathcal{S}}_h}

\newcommand{\colpertoph}{\bm{\mathcal{D}}_h}
\newcommand{\colpertrhsoph}{\bm{\mathcal{G}}_h}

\newcommand{\colrhs}{\mathcal{F}}

\newcommand{\colrhshn}[1]{\colrhs_h^{#1}}
\newcommand{\colrhshnavg}[1]{\overline{\widehat{\colrhs}}_h^{#1}}

\newcommand{\colerrhn}[1]{\mathcal{E}_h^{#1}}
\newcommand{\colerrhnavg}[1]{\overline{\mathcal{E}}_h^{#1}}

\DeclarePairedDelimiterX\setc[2]{\{}{\}}{\,#1 \;\delimsize\vert\; #2\,}
\DeclarePairedDelimiter\monosetc{\{}{\}}

\renewcommand{\div}{\operatorname{div}}
\newcommand{\curl}{\operatorname{curl}}
\newcommand{\imp}{{\operatorname{imp}}}
\newcommand{\pcurl}{\widetilde{\curl}}

 \newcommand{\Sym}{\operatorname{Sym}}

\newcommand{\pt}{\partial_t}
\newcommand{\ip}[1]{\langle #1 \rangle}

\newcommand{\interpolation}{\mathcal{I}_h}

\newcommand{\projwave}{P_{h}}
\newcommand{\projheat}{\widetilde{P}_{h}}
\newcommand{\projwaveRef}{J_{h}}
\newcommand{\projheatRef}{\widetilde{J}_{h}}

\newcommand{\projfull}{\mathbf{P}_{h}}
\newcommand{\projReffull}{\mathbf{J}_{h}}

\newcommand{\projRitz}{\textup{R}_h}
\newcommand{\projLtwo}{\pi_h}

\newcommand{\avg}[1]{\{\hspace{-3.8pt}\{#1\}\hspace{-3.8pt}\}}
\newcommand{\jmp}[1]{[\![#1]\!]}
\newcommand{\Fint}{{F_{\operatorname{int}}}}
\newcommand{\nint}{n_{\operatorname{int}}}

\newcommand{\dGMesh}{\mathcal{T}_h}
\newcommand{\dGnF}{n_F}
\newcommand{\dGnKF}{n_{K,F}}

\newcommand{\dGtF}{\tau_F}
\newcommand{\dGtKF}{\tau_{K,F}}

\newcommand{\dGFaceK}{\mathcal{F}_K}
\newcommand{\dGFaceKOne}{\mathcal{F}_{K_1}}
\newcommand{\dGFaceKTwo}{\mathcal{F}_{K_2}}
\newcommand{\dGFace}{\mathcal{F}_h}
\newcommand{\dGFaceInt}{\mathcal{F}_h^{\operatorname{int}}}
\newcommand{\dGFaceBnd}{\mathcal{F}_h^{\operatorname{bnd}}}
\newcommand{\dGJmp}[1]{[\![#1]\!]}
\newcommand{\dGAvg}[1]{\{\!\{#1\}\!\}}
\newcommand{\LTwoK}{L^2(K)}
\newcommand{\LTwoF}{L^2(F)}

\newcommand\restr[2]{{          %
\left.\kern-\nulldelimiterspace %
#1                              %
\vphantom{\big|}                %
\right|_{#2}                    %
}}
\newcommand{\bigrestr}[2]{#1\big|_{#2}}

\newcommand{\pdiv}{\widetilde{\div}}
\newcommand{\energyspace}{\mathbb{E}}
\newcommand{\maxop}{\calM}

\newcommand{\divspaceE}{\mathcal{N}_{\!\eps}}
\newcommand{\divspaceH}{\mathcal{N}_{\!\mu}}
\newcommand{\projE}{\widetilde{P}_{\eps}}
\newcommand{\projH}{P_{\mu,0}}
\newcommand{\sigmaeff}{\sigma_{\operatorname{eff}}}
\newcommand{\tang}{\mathrm{t}}

\usepackage{comment}

\begin{document}

\title{Leapfrog Crank--Nicolson decoupling of wave--heat-type problems}

\author{Benjamin D\"orich}
\address{Department of Mathematics, Saarland University, 66123 Saarbr\"ucken, Germany}

\email{benjamin.doerich@uni-saarland.de}

\author{Julian D\"orner}
\address{Institute for Applied and Numerical Mathematics, Karlsruhe Institute of Technology, Englerstraße 2, 76149 Karlsruhe, Germany}

\email{julian.doerner@kit.edu}



\begin{abstract}
    This paper proposes and analyzes a numerical method for coupled wave-heat systems that arise, for example, in visco-elasticity with memory, thermo-elastic wave propagation with finite thermal speed, and Maxwell's equations coupled to dispersive material laws. Motivated by the widespread use of explicit leapfrog schemes for wave problems and the parabolic time-step restriction for heat equations, we design a second-order scheme that combines leapfrog time integration for the wave part with a Crank--Nicolson step for the heat part. The coupling is arranged in a Strang-splitting--type fashion so that the overall scheme remains explicit in the coupling and is subject only to a CFL condition of hyperbolic type, rather than the more restrictive parabolic constraint.

    We introduce an abstract space discretization that covers a broad class of wave--heat systems and accommodates both conforming and stabilized non-conforming discretizations. Using a general error decomposition, we derive error estimates for the fully discrete scheme under the desired CFL condition. Numerical experiments for our model applications confirm the theoretical results.
\end{abstract}

\maketitle

\section{Introduction}

In this paper we study wave-type systems which are coupled to heat-type problems and can formally be written as
\begin{subequations} \label{eq:full_system_formal}
    \begin{align}
    \pt \wavevar + \waveop \wavevar &= - \coupopfull
    \heatvar + f ,
     \label{eq:full_system_formal_wave}
    \\
    \pt \heatvar + \heatop \heatvar &= \coupopfull^* \wavevar + g  , 
    \label{eq:full_system_formal_heat}
\end{align}
\end{subequations}
where the wave part \eqref{eq:full_system_formal_wave} is given as a two-component system with  $\wavevar = (\wavevarOne,\wavevarTwo) \in \wavespace = \wavespaceOne \times \wavespaceTwo$ 
and a skew-adjoint operator 
$\waveop$,
the heat part by
 some accretive operator $\heatop$ in \eqref{eq:full_system_formal_heat},
 and the coupling is encoded in $\coupopfull$.

In this work we have four main classes of examples in mind. 
As a particularly simple and illustrative starting point, we first consider a one-dimensional wave equation coupled to a heat equation, modeling a thermo-elastic string.

A second example is given by nonlocal-in-time linear visco-elasticity systems, which describe materials with memory effects subject to pressure and shear waves.

As a third example, we consider thermo-elastic wave equations, which model the coupled propagation of mechanical and thermal disturbances in a solid and, in contrast to classical Fourier heat conduction, account for a finite propagation speed of thermal signals.
Compare \cite{KupGBB79} and \cite{IgnOS10} for a general introduction.

Finally, we consider Maxwell’s equations in a bulk medium coupled to a system of differential equations posed on an interface, modeling the nonlocal-in-time response of an artificial two-dimensional metamaterial such as graphene.

To be more precise, all our examples fit in the more specific structure of
\eqref{eq:full_system_formal}, i.e. 
\begin{subequations} \label{eq:full_system_formal_v2}
    \begin{align}
    \pt \wavevarOne - \waveopTwo \wavevarTwo &= - \coupop
    \heatvar + f_1
    \label{eq:full_system_formal_v2:WaveOne}
    \\
    \pt \wavevarTwo + \waveopOne \wavevarOne &= 0
    \label{eq:full_system_formal_v2:WaveTwo}
    \\
    \pt \heatvar + \heatop \heatvar &= \coupop^* \wavevarOne + g.
    \label{eq:full_system_formal_v2:Heat}
\end{align}
\end{subequations}

In the applications, the leapfrog scheme is the method of choice to discretize the (spatially discrete) wave problem. 
It is explicit and preserves important geometric properties of the solution. 
However, this scheme comes with a Courant--Friedrichs--Lewy (CFL) condition which roughly speaking only allows for time steps $\tau$ of the order of the mesh size $h$.
While this is widely accepted to be still competitive, the situation changes if an explicit scheme is applied to a heat-type problem as well, since here a scaling as $\tau \lesssim h^2$ is usually required. 
Thus, we propose to combine the leapfrog method for the wave part with the implicit \CN method for the heat part in order to solve the problem only under the weaker CFL condition.
Formal application of the scheme to a finite dimensional version of \eqref{eq:full_system_formal_v2} reads
\begin{subequations}\label{eq:LFCNFormally}
    \begin{align}
                \wavevarnOne{\nhalf}
            &= \wavevarnOne{n}
                + \tauhalf \waveopTwo \wavevarnTwo{n}
                -\tauhalf \coupop \heatvarn{n}
                + \tauhalf f_{1}^n,
            \label{eq:LFCNFormally:One}\\
        \wavevarnTwo{n+1}
            &= \wavevarnTwo{n} 
                - \tau \waveopOne \wavevarnOne{\nhalf},
            \label{eq:LFCNFormally:Two}\\
        \heatvarn{n+1}
            &= \heatvarn{n} 
            - \tauhalf \heatop
                \bigl(\heatvarn{n+1} + \heatvarn{n}\bigr)
            + \tau\coupop^\ast\wavevarnOne{\nhalf}
            + \tauhalf \bigl(
                g^{n+1} + g^{n}
                \bigr)
            \label{eq:LFCNFormally:Three},\\
        \wavevarnOne{n+1}
            &= \wavevarnOne{\nhalf}
                + \tauhalf \waveopTwo \wavevarnTwo{n+1}
                -\tauhalf \coupop \heatvarn{n+1}
                + \tauhalf f_{1}^{n+1},
            \label{eq:LFCNFormally:Four}
    \end{align}
\end{subequations}
at times \(t_n = n\tau + t_0\), for \(n\geq 0\) subject to suitable initial values \((\wavevarn{0},\heatvarn{0})\).
Equations \eqref{eq:LFCNFormally:One}, \eqref{eq:LFCNFormally:Two}, and \eqref{eq:LFCNFormally:Four} constitute the standard leapfrog scheme applied to the wave equation \eqref{eq:full_system_formal_v2:WaveOne}, \eqref{eq:full_system_formal_v2:WaveTwo}, with the heat variable entering as an additional source term. Equation \eqref{eq:LFCNFormally:Three} is a Crank–Nicolson step for the heat variable, where the source term arising from the wave coupling is evaluated at the temporal midpoint in order to keep the overall coupling of both time-integration schemes explicit.
Thus, our scheme has the flavor of a Strang splitting scheme, solving both subproblems in an interleaved fashion, each with a suitable scheme.

Since the application of an explicit scheme is not suitable on the operator level, we discretize in space first.
We present our discretization in an abstract manner in order to cover the four examples mentioned above in a general framework. In particular, we allow for conforming and non-conforming discretizations in the wave part. For the spatially discrete error analysis, we follow the ideas of \cite{Hip.Hoc.Sto.2019} and collect the different error contributions in a decomposition which allows for direct application of results from the literature to derive the optimal error bounds.
We emphasize that we allow for stabilized non-conforming discretizations, and, thus, extend the unified framework in \cite{Hip.Hoc.Sto.2019} about this important class of space discretization schemes.
The fully discrete scheme can then be analyzed as a perturbation of the \CN method. The perturbation can be given a definite sign, which yields a stable scheme under the desired CFL condition \(\tau\lesssim h\).
Further, let us note that we employ summation-by-parts   in order to avoid unphysical constraints on the right-hand sides.

To the best of our knowledge, there is hardly any work on the decoupling strategies for wave–heat-type problems, apart from \cite{LayT12}.
There, the stability of two second-order two-step methods is analyzed. However, the analysis is restricted to ordinary differential equations and does not provide any error bounds.

Nevertheless, there is a large literature available for the individual examples, and we provide a non-exhaustive short overview of results in the following:

In the context of visco-elastic waves, an upwind-flux stabilized discontinuous Galerkin method is a well-established scheme in geophysics, cf. \cite{KaeD06a,KaeD06b}.
For time-integration, implicit Runge--Kutta schemes are discussed in \cite{HocPSTW15}, as well as space-time methods in \cite{DoeFW16,DoeWZ19}.

In the case of linear thermo-elasticity, various numerical and analytical approaches have been proposed in the literature. 
Boundary element methods for coupled thermo-elastic problems in a scattering setting are investigated in \cite{HsiSTS19}, 
while in \cite{JiaR00} the problem is formulated and studied in the framework of abstract evolution equations.
Furthermore, qualitative stability properties for thermo-elastic systems are analyzed in \cite{BodGN14}.

For the Maxwell problem with discontinuous Galerkin discretization in space there are the seminal works by Hesthaven \& Warburton \cite{Hes.W.2002} and Fezoui et al. \cite{FezLLP05}, which lay the foundation for the time-domain Maxwell problem. A generalization to dispersive media, which has structural similarities to our model, was provided by Lanteri \& Scheid \cite{LanS13},
and more recently the authors laid the foundation for the treatment of the interface problem in \cite{DoeDH26}.

\medskip

The rest of the paper is organized as follows:
In Section~\ref{sec:ana_frame}, we introduce the analytical framework as well as the above-mentioned examples that fit into this framework. 
We then introduce the abstract spatial discretization in Section~\ref{sec:space}, and provide a semi-discrete error bound. 
In Section~\ref{sec:full}, we derive our fully discrete method, provide stability estimates, and show convergence of optimal order. 
Section~\ref{sec:application} is devoted to our examples. We discuss the spatial discretization in the specific settings and apply our main results.
Our numerical experiments in Section~\ref{sec:num_exp}
confirm the theoretical findings.
Some computations are postponed to the Appendix~\ref{sec:app_1}.

\medskip

\subsection*{Notation}

Throughout this work, we denote by $L^2(\Omega)$ the standard Lebesgue space of square integrable functions on a domain \(\Omega\).  The standard Sobolev spaces of order $k \geq 0$ are denoted with $H^k(\Omega)$, as well as its closed subspace $H_0^1(\Omega)$ with vanishing trace.
Further, we use the standard spaces for the weak divergence $H(\div)$, 
together with its closed subspaces $H_0 (\div)$, i.e. $ \vbold \cdot\nu = 0$ on $\partial \Omega$
for all $\vbold \in H_0 (\div)$.
For a Hilbert space $X$, we denote its topological dual by $X^*$ together with the induced norm.
In addition, we
write $\alpha \lesssim \beta$ if there is a constant $C$ independent of the spatial mesh parameter $h$ and the time step size $\tau$ such that $\alpha \leq C \beta$.

\section{Analytical framework and main examples}
\label{sec:ana_frame}

We first discuss the analytical framework for the rest of the paper. For a rigorous treatment, we move to a weak formulation of \eqref{eq:full_system_formal}.
Before we present the equation, we introduce the two Gelfand triples
\begin{equation*}
\wavespaceV \hookrightarrow \wavespace \hookrightarrow \wavespaceV^*,
\qquad
\heatspaceV \hookrightarrow \heatspace \hookrightarrow \heatspaceV^*
\end{equation*}
with dense embeddings, belonging to the wave and heat part, respectively.
On these spaces, we consider the bilinear forms
\begin{align*}
    &\ipwave \colon \wavespace \times \wavespace \to \R, 
    \quad
    \sbilwave \colon \wavespaceV \times \wavespaceV \to \R, 
    \\
    &\ipheat \colon \heatspace \times \heatspace \to \R, 
    \quad
    \sbilheat \colon \heatspaceV \times \heatspaceV \to \R ,
\end{align*}
where $\ipwave$ and $\ipheat$ are simply the inner products of $\wavespace$ and $\heatspace$, respectively, and $\sbilwave$ and $\sbilheat$ encode the differential operators.  Further, we introduce the coupling form
\begin{equation*}
    \coupopfullform \colon \heatspaceV \times \wavespaceV \to \R \,.
\end{equation*}
In order to resemble the two-field structure of the wave part, cf. \eqref{eq:full_system_formal_v2}, we assume that the bilinear form $\ipwave$ and $\sbilwave$ can be decomposed as
\begin{equation*} 
    \ipwave(x,y) = \ipwaveOne(x_1,y_1) + \ipwaveTwo(x_2,y_2),
\qquad
    \sbilwave(x,y) = \sbilwaveOne(x_1,y_2) - \sbilwaveTwo(x_2,y_1),
\end{equation*}
and the coupling as well as the forcing only appears in the first component of the wave system as 
\begin{equation} \label{eq:coup_f_decomp} 
   \coupopfullform(\heatvar, y)  =    \coupopform(\heatvar, y_1) ,
   \qquad
\ipwave(f,y) = \ipwaveOne(f_1,y_1),
\end{equation}
With this, we can formulate problem \eqref{eq:full_system_formal} as follows:

Seek $\wavevar \in C^1(\R,\wavespace) \cap C(\R,\wavespaceV)$
and
$\heatvar \in C^1(\R,\heatspace) \cap C(\R,\heatspaceV)$
which satisfy 
\begin{subequations} \label{eq:full_system}
\begin{align}
 \ipwave(\pt \wavevar, y)
+
\sbilwave(\wavevar, y)
&= - \coupopfullform(\heatvar, y) +
\ipwave(f,y)
 \label{eq:full_system_wave}
    \\
\ipheat(\pt \heatvar, \varphi)
+
\sbilheat(\heatvar, \varphi)
&= 
\coupopfullform(\varphi, \wavevar) +
\ipheat(g,\varphi) , 
\label{eq:full_system_heat}
\end{align}
\end{subequations}
for all $y \in \wavespaceV$ and $\varphi \in \heatspaceV$,
as well as the initial conditions $\wavevar(0) = \wavevar_0$ and $\heatvar(0) = \heatvar_0$.
For the analysis, we make the following assumptions, which are verified later on for our individual examples.

\begin{assumption} \label{ass:cont_forms}
(a) The bilinear forms $\ipwave$ and $\ipheat$ are inner products and induce the norms
\begin{equation*}
    \| y \|_{\wavespace}^2 = \ipwave(y,y)
    \quad
    \text{and}
    \quad
    \| \varphi \|_{\heatspace}^2 = \ipheat(\varphi,\varphi),
 \end{equation*}
for any $y \in \wavespace$ and $\varphi \in \heatspace$.
 
\noindent
(b)   It holds that %
\begin{align*}
\sbilwave(y, y) = 0 
    \quad
    \text{and}
    \quad
     \sbilheat(\varphi, \varphi) \geq 0
\end{align*}
for $y \in \wavespaceV$ and $\varphi \in \heatspaceV$.

\noindent
(c) There exist spaces $\fullspace \subset \fullspace_1  \times \fullspace_2  \subset \wavespaceV \times \heatspaceV$ 
and bilinear forms
 $\sbilwaveext \colon \fullspace_1  \times \wavespace$
 and
 $\coupopfullformext \colon \fullspace_2 \times \wavespace$
 such that it holds
\begin{align*}
\sbilwave(x, y) + \coupopfullform(\heatvar, y) = \sbilwaveext(x,y)
+
\coupopfullformext(\heatvar, y)
\end{align*}
for all $(x,\heatvar) \in \fullspace$ and $y \in \wavespaceV$.
\end{assumption}
\begin{remark} \label{rem:Ass_ibp}
        The rather technical assumption in part (c) allows for coupling operators that do not act in the bulk but via an interface or boundary.
        Such a coupling can be expressed indirectly via partial integration, which motivates our assumption.
        A concrete examples is given later in Section~\ref{sec:ana_Maxwell}.
 \end{remark}
Let us also emphasize that part (c) implies that a solution $(\wavevar,\heatvar)$ of \eqref{eq:full_system}
which is continuous in $\fullspace$ also satisfies
\begin{align} \label{eq:first_eq_tested_in_X}
 \ipwave(\pt \wavevar, y)
+
\sbilwaveext(\wavevar, y)
&= - \coupopfullformext(\heatvar, y) +
\ipwave(f,y)
\end{align}
for all $y \in \wavespace$ using a density argument.
We also note that with some additional notation, one could also treat a dissipative wave-type problem, i.e. $\sbilwave(y, y) \geq 0 $.

This framework immediately allows for the following a-priori stability result. 
\begin{lemma}
Let $\wavevar \in C^1([0,T],\wavespace) \cap C([0,T],\wavespaceV)$
and
$\heatvar \in C^1([0,T],\heatspace) \cap C([0,T],\heatspaceV)$ be a solution of
 \eqref{eq:full_system}, then
\begin{equation*}
    \| \wavevar(t) \|_{\wavespace}^2 +  \| \heatvar(t) \|_{\heatspace}^2
    \leq
    e 
    \bigl(
    \| \wavevar(0) \|_{\wavespace}^2 +  \| \heatvar(0) \|_{\heatspace}^2
     \bigr)
    +
    e T 
    \bigl(
    \int_0^t
    \| f(s) \|_{\wavespace}^2 +  \| g(s) \|_{\heatspace}^2 \,ds
    \bigr) .
    \end{equation*}
\end{lemma}

\begin{proof}
To obtain this result, we test with the solution in  \eqref{eq:full_system}, apply Cauchy-Schwarz as well as a (time-)weighted Young inequality, and conclude via the Gronwall lemma in \cite[Proposition~2.1]{Emm99}.
\end{proof}

\begin{remark}
    We could also allow for nonlinear contributions $f$ or $g$ in \eqref{eq:full_system}, which are locally Lipschitz in $\wavespace$ or $\heatspace$, respectively. However, since this does not change the analysis in an essential manner, we omit this here for the sake of presentation.
\end{remark}

In the following sections, we discuss the examples from the introduction and show how they fit into the framework from above.

\subsection{Thermo-elastic string}
\label{sec:ana_thermoelastic_string}

Our first example concerns the modelling of a string on an interval $\Omega =(a,b)$ with displacement $u$ subject to a heating field $\heatvar$, see for example Section~3.3.1 in \cite{BodGN14}. The governing equations read
\begin{subequations} \label{eq:acous_wave_heat_1d}
\begin{align}
    \pt^2 u - \partial_x^2 u &= - \partial_x \heatvar + f ,
    \\
    \pt \heatvar - \partial_x^2 \heatvar &= - \partial_x \pt u + g .
\end{align}
\end{subequations}
In order to fit this problem into our general framework, we have to make the (unusual) definition 
$x = (\pt u, u)$. If we want to change the positions, we simply have to change the roles of the first and second component of $x$ in the rest of this work.
Consequently, we define the spaces
 \begin{align*}
        \wavespace &= L^2(\Omega) \times H_0^1(\Omega) , 
        \quad
        \wavespaceV =  H_0^1(\Omega) \times H_0^1(\Omega), 
        \qquad
        \heatspace =  L^2(\Omega),
        \quad
        \heatspaceV =  H_0^1(\Omega) ,
      \end{align*}
    as well as the forms
\begin{equation*}
\begin{alignedat}{4}
\ipwave(\cdot,\cdot)
&{}= \langle \cdot, \cdot \rangle_{L^2 \times H_0^1},
&\quad
\sbilwaveOne(x_1,y_2) 
&{}=
 - \langle  \partial_x x_1 ,  \partial_x y_2 \rangle_{L^2},
 &\quad
\sbilwaveTwo(x_2,y_1)
&{}=
  \sbilwaveOne(x_2,y_1),
\\
\ipheat(\cdot,\cdot)
&{}=
 \langle \cdot, \cdot \rangle_{L^2},
&\quad
\sbilheat(\cdot,\cdot)
&{}= 
\langle \cdot, \cdot \rangle_{H_0^1},
\end{alignedat}
\end{equation*}
such that part (a) and (b) of Assumption~\ref{ass:cont_forms} are satisfied.
Further, we define
the  coupling form
\begin{align*}
    \coupopfullform(\varphi,x)
     =
    - \langle
     \partial_x \varphi ,  \pt u
     \rangle_{L^2},
\end{align*}
and 
observe that the extended bilinear forms in Assumption~\ref{ass:cont_forms}~(c) can be chosen to be the original ones.

\subsection{Non-local in time visco-elasticity and acoustics}
\label{sec:ana_ell_wave}
In our second example we treat wave propagation in visco-elastic solids and its acoustic analogue. 
We begin with the visco-elastic wave equations, i.e., elastic waves in solids subject to non-local in time material laws.

Let $\Omega \subset \R^d$, $d\in\{2,3\}$, denote the spatial domain. 
We denote by $\vbold$ the velocity field and by $\sigmabold_i$ the symmetric stress tensors corresponding to different 
relaxation mechanisms, $i=0,\ldots,G$. 
The strain-rate tensor is given by $\epsilon(\vbold) = \Sym (D \vbold)$.

They satisfy the system
\begin{subequations} \label{eq:visco_elastic}
\begin{align}        
  \rho \partial_t \vbold - \div \sigmabold_0 
    &= 
    \sum_{i=1}^G \div \sigmabold_i  +  f,
    \\
  C_0^{-1} \partial_t \sigmabold_0 - \epsilon(\vbold) &= 0,
    \\
  C_i^{-1} \partial_t \sigmabold_i
    + (C_i \tau_i )^{-1} \sigmabold_i
    &= \epsilon(\vbold) , \qquad  i =1,\ldots,G .
\end{align}
\end{subequations} 
To fit this problem into our abstract framework, we split the unknowns in
\[
  \wavevar = (\vbold, \sigmabold_0),
  \qquad
  \heatvar = (\sigmabold_1,\ldots,\sigmabold_G),
\]
and choose the spaces
\begin{equation*}
\begin{alignedat}{4}       
  \wavespace &{}=  L^2(\Omega)^{(d+1)d}, 
  &{}\quad
  \wavespaceV &{}= H^1_0(\Omega,\R^d)
                 \times H(\div, \Omega,\R^{d\times d}_{\textup{sym}}),
        \\
  \heatspace &{}=  L^2(\Omega)^{G d^2}, 
  &{}\quad
  \heatspaceV &{}=  H(\div, \Omega,\R^{d\times d}_{\textup{sym}})^G .
\end{alignedat} 
\end{equation*}
For $x=(\vbold, \sigmabold_0)$, $y = (\wbold,\widetilde{ \sigmabold_0})$ and 
$\heatvar = (\sigmabold_1,\ldots,\sigmabold_G)$ we define
\begin{align*}
  \ipwave(x,y) 
    &= \langle \rho \vbold, \wbold \rangle_{L^2}
       + \langle C_0^{-1} \sigmabold_0, \widetilde{\sigmabold_0} \rangle_{L^2},
  \\
  \sbilwave(x,y) 
    &=
    - \langle \div \sigmabold_0 , \wbold \rangle_{ L^2}
    - \langle \epsilon( \vbold ) ,\widetilde{ \sigmabold_0} \rangle_{L^2},
  \\
  \ipheat(\heatvar,\widehat{\heatvar}) 
    &= \sum_{i=1}^G
       \langle C_i^{-1} \sigmabold_i, \widehat{\sigmabold}_i \rangle_{L^2},
  \\
  \sbilheat(\heatvar,\widehat{\heatvar}) 
    &= \sum_{i=1}^G
       \langle (C_i \tau_i)^{-1} \sigmabold_i, \widehat{\sigmabold}_i \rangle_{L^2},
\end{align*}
and the coupling form
\begin{align*}
  \coupopfullform(\heatvar, \wavevar) 
    &=
    - \sum_{i=1}^G   \langle
        \div \sigmabold_i  , \vbold
      \rangle_{L^2}
    =
      \sum_{i=1}^G   \langle
        \sigmabold_i  , \epsilon( \vbold )
      \rangle_{L^2},
\end{align*}
where the second equality follows from integration by parts and the homogeneous boundary conditions.
Note that by using the first equality for $\coupopfullform$ the extended bilinear forms in Assumption~\ref{ass:cont_forms}~(c) can be chosen to be the original ones.

\medskip

A scalar visco-acoustic analogue of this model is obtained by replacing the tensor-valued stresses 
$\sigmabold_i$ by scalar pressures $p_i$, $i=0,\ldots,G$, and the shear modulus by the bulk modulus $\kappa$. 
This leads to the system
\begin{subequations} \label{eq:visco_acoustic}
\begin{align}
  \rho \partial_t \vbold - \nabla p_0 
    &= \sum_{i=1}^G \nabla p_i  +  f,
    \\
  \kappa^{-1} \partial_t p_0 -  \div \vbold &= 0,
    \\
  (\kappa \tau_P)^{-1} \partial_t p_i 
    + (\kappa \tau_P \tau_i)^{-1}  p_i
    &= \div \vbold ,
       \qquad  i =1,\ldots,G .
\end{align}
\end{subequations}
To fit \eqref{eq:visco_acoustic} into our framework, we again split
\[
  \wavevar = (\vbold, p_0),
  \qquad
  \heatvar = (p_1,\ldots,p_G),
\]
and choose
\begin{equation*}
\begin{alignedat}{4}       
  \wavespace &{}=  L^2(\Omega)^{d+1}, 
  &{}\quad
  \wavespaceV &{}= H_0(\div,\Omega,\R^d) \times H^1(\Omega), 
        \\
  \heatspace &{}=  L^2(\Omega)^G,
  &{}\quad
  \heatspaceV &{}=  H^1(\Omega)^G .
\end{alignedat} 
\end{equation*}
For $x=(\vbold,p_0)$, $y = (\wbold,q_0)$ we set
\begin{align*}
  \ipwave(x,y) 
    &= \langle \rho \vbold, \wbold \rangle_{(L^2)^d}
       + \langle \kappa^{-1} p_0, q_0 \rangle_{L^2},
  \\
  \sbilwave(x,y) 
    &=
    - \langle \nabla p_0 , \wbold \rangle_{ L^2}
    - \langle \div \vbold , q_0 \rangle_{L^2},
  \\
  \ipheat(\heatvar,\widehat{\heatvar}) 
    &= \langle (\kappa \tau_P)^{-1} \heatvar, \widehat{\heatvar} \rangle_{(L^2)^G},
  \\
  \sbilheat(\heatvar,\widehat{\heatvar}) 
    &= \sum_{i=1}^G
       \langle (\kappa \tau_P \tau_i)^{-1} p_i, \widehat{p}_i \rangle_{L^2},
\end{align*}
and the coupling form
\begin{align*}
  \coupopfullform(\heatvar, \wavevar) 
    &=
    - \sum_{i=1}^G   \langle
        \nabla p_i  , \vbold
      \rangle_{L^2} .
\end{align*}
Again, the extended bilinear forms in Assumption~\ref{ass:cont_forms}~(c) can be chosen to be the original ones.

\subsection{Thermo-elastic wave coupling}
\label{sec:ana_thermo_wave}

Our third example concerns elastic solids coupled with heat conduction.

Let $\Omega \subset \R^d$, \(d\in\{2,3\}\), be the spatial domain.
We consider the displacement field 
\(\ubold\) 
and the scalar temperature distribution 
\(\heatvar\),
which satisfy the linear thermo-elastic system
\begin{align*}
  \rho \partial_t^2 \ubold - \mu \Delta\ubold - (\lambda + \mu)\nabla(\div\ubold ) 
    &= \alpha \nabla \heatvar + \rho f,\\
  \partial_t \heatvar - \kappa\Delta \heatvar 
    &= \alpha \div (\partial_t\ubold) + h,
\end{align*}
with Lamé parameters \(\lambda \geq 0\), \(\mu > 0\), coupling parameter \(\alpha > 0\), 
and thermal diffusivity \(\kappa > 0\).
We rewrite the system in terms of the velocity 
\(\vbold = \partial_t \ubold\) 
and the symmetric stress \(\sigmabold\), to obtain
\begin{subequations} \label{eq:thermo_elastic}
\begin{align}    
  \rho \partial_t \vbold - \div \sigmabold 
    &= \alpha \nabla \heatvar + \rho f,
    \label{eq:thermo_elastic:evo_v}\\
  \partial_t \sigmabold - C\,\epsilon(\vbold) 
    &= 0,
    \label{eq:thermo_elastic:evo_sigma}\\
  \partial_t \heatvar - \kappa \Delta \heatvar
    &= \alpha \div \vbold + h,
    \label{eq:thermo_elastic:evo_theta}
\end{align}
\end{subequations} 
where \(C\wbold = 2\mu \wbold + \lambda\operatorname{tr}(\wbold)I\) is the isotropic elasticity tensor and
\(\epsilon( \vbold ) = \Sym(D \vbold)\) the symmetric gradient.

To fit \eqref{eq:thermo_elastic} into our abstract framework, we collect the mechanical unknowns in
\(
  \wavevar = (\vbold, \sigmabold)
\)
and treat the temperature as the “heat part” \(\heatvar\).
We choose the spaces
\begin{equation*}
\begin{alignedat}{4}       
  \wavespace &{}=  L^2(\Omega)^{d(d+1)}, 
  &{}\quad
  \wavespaceV &{}= 
      H^1_0(\Omega,\R^d)
      \times
      H(\div, \Omega,\R^{d\times d}_{\textup{sym}}),
  \\
  \heatspace &{}=  L^2(\Omega), 
  &{}\quad
  \heatspaceV &{}=  H^1_0(\Omega) .
\end{alignedat} 
\end{equation*}
For $x = (\vbold,\sigmabold)$ and $y = (\widetilde{\vbold},\widetilde{\sigmabold})$ we define
\begin{equation*}
\begin{alignedat}{4}   
  \ipwave(x,y) 
    &= \langle \rho \vbold, \widetilde{\vbold} \rangle_{L^2}
       + \langle C^{-1} \sigmabold, \widetilde{\sigmabold} \rangle_{L^2},
  &{}\quad
  \sbilwave(x,y) 
    &=
    - \langle \div \sigmabold , \widetilde{\vbold} \rangle_{ L^2}
    - \langle \epsilon( \vbold ) , \widetilde{\sigmabold} \rangle_{L^2},
  \\
  \ipheat(\heatvar,\widehat{\heatvar}) 
    &= \langle \heatvar, \widehat{\heatvar} \rangle_{L^2},
  &{}\quad
  \sbilheat(\heatvar,\widehat{\heatvar}) 
    &= \langle \kappa \nabla \heatvar, \nabla \widehat{\heatvar} \rangle_{L^2},
\end{alignedat} 
\end{equation*}
and the coupling form
\begin{align*}
  \coupopfullform(\heatvar, \wavevar) 
    &=
    - \langle \alpha \nabla \heatvar , \vbold \rangle_{L^2}.
\end{align*}
As in the previous examples, the extended bilinear forms in Assumption~\ref{ass:cont_forms}~(c) can be chosen to be the original ones.

\subsection{Maxwell equation with dissipative 2D material}
\label{sec:ana_Maxwell}
Our last example models the light-matter interaction of thin, two-dimensional me\-ta\-ma\-ter\i\-als, like gra\-ph\-en\-e.
Here, the graphene sheet is described by hypersurface with a non-local in time response, modelled via a set of ordinary extension.

Let \(\Omega = (-1,1)\times (0,1)^2\) denote a rectangular domain, partitioned by an interface \(\Fint\) such that
\begin{equation*}
    \Omega_1 = (-1,0) \times (0,1)^2,
    \quad 
    \Omega_2 = (0,1) \times (0,1)^2,
    \quad 
    F_{\operatorname{int}} = \monosetc{0}\times (0,1)^2.
\end{equation*}
Besides the electric and magnetic field $E$ and $H$,  we further have a set of auxiliary fields $(\heatvar_1,\ldots,\heatvar_\ell)$ on the interface \(\Fint\) with values parallel to the interface.
The governing equations read 
\begin{subequations} \label{eq:Maxwell_interface}
\begin{align}
   \pt E^{(i)} - \curl H^{(i)}  &= J^{(i)},  
    &\text{in } [0,T]\times \Omega_i,
    \\
    \pt H^{(i)} + \curl E^{(i)}  
    &= 0, &\text{in } [0,T]\times \Omega_i,
    \\
    \pt \heatvar + (\heatop \otimes I_2) \heatvar 
    &= (c_\ell \otimes I_2)\wavevarOneparE  + g, 
    &\text{on } [0,T]\times \Fint,
\end{align}
where \(E^{(i)} = E|_{\Omega_i}\), and $\wavevarOneparE$ is the tangential projection of $E$ onto \(\Fint\).
The matrix $\heatop \in \R^{\ell \times \ell}$ satisfies \(\varphi^T \heatop \varphi \geq 0\) for all \(\varphi \in \R^\ell\), i.e., its symmetric part is positive semi-definite, and \(c_\ell \in \monosetc{0,1}^\ell\), \(c_\ell\neq 0\),  is the coupling vector.

In  order to close the system described by \eqref{eq:Maxwell_interface}, a coupling of the auxiliary variables into the Maxwell system needs to be described. 
This is formally given by the interface condition
\begin{equation}\label{eq:MaxwellCouplingCondition}
    \jmp{ H\times \nint}_\Fint = - (c_\ell^T\otimes I_2)\heatvar,
\end{equation}
\end{subequations}
where \(I_2 = \operatorname{diag}(1,1)\), the vector \(\nint\) denotes the unit normal on \(\Fint\), and the double bracket the jump at the interface.
The latter condition has to be carefully included into the system via the domain of the operators since the coupling is not described by an operator defined on the bulk.
We will further elaborate on this at the end of the section.

To fit this problem into our abstract framework, we set $\wavevar = (\wavevarOne,\wavevarTwo) = (E,H)$
    and use the spaces
    \begin{align*}
        \wavespace &= L^2(\Omega)^6, 
        \quad
        \wavespaceV = H^{\imp}_0(\curl) \times H^\imp(\pcurl),
        \\
        \heatspace &= \heatspaceV = L^2(\Fint)^{2\ell},
        \quad
        \ell\in \mathbb{N}_+,
      \end{align*}
    with the regularized \(\curl\)-spaces
    \begin{align*}
        H^\imp_0(\curl) &= \setc{E\in L^2(\Omega)^3}{\curl E\in L^2(\Omega)^3,\, E_\parallel \in L^2(\Fint)^2,\, E\times\nu = 0},\\
        H^\imp(\pcurl) &= \setc{H\in L^2(\Omega)^3}{\pcurl\, H \in L^2(\Omega)^3,\, \jmp{H\times n}_\Fint \in L^2(\Fint)^2}.
    \end{align*}
    Here, \(\pcurl\) denotes a piecewise \(\curl\), i.e. for \(H\in L^2(\Omega)^3\) we define
    \begin{equation*}
        \pcurl\, H\in L^2(\Omega)^3 \quad\Leftrightarrow\quad
            \restr{\curl H}{\Omega_i} \in L^2(\Omega_i)^3,\,\, i\in\monosetc{1,2}. 
    \end{equation*}
    For any details on those spaces and the resulting traces we refer to \cite[Sec.~2.2.2]{Ass.C.L.2018.MFCE}.
    Finally, we define the forms
    \begin{align*}
        \ipwave(\cdot,\cdot) &= \langle \cdot, \cdot \rangle_{L^2(\Omega)^6},
        \\
        \ipheat(\cdot,\cdot) &= \langle \cdot,\cdot \rangle_{L^2(\Fint)^{2\ell}},
        \\
        \sbilwave(x,y) &= \ipwave(\curl x_1, y_2) - \ipwave(x_2, \curl y_1),
        \\
        \sbilheat(\heatvar,\varphi) &= \ipheat( (\heatop \otimes I_2) \heatvar, \varphi),
    \end{align*}
    with coupling form
    \begin{align*}
        \coupopfullform( \varphi, \wavevar) 
        =
        \ipheat( \varphi , (c_\ell \otimes I_2)\wavevarOnepar).
    \end{align*}
    Standard results yield that part (a) and (b) of Assumption~\ref{ass:cont_forms} are satisfied.
    
    As mentioned before, we include condition \eqref{eq:MaxwellCouplingCondition} in the operator domain and define
    \begin{equation*}
        \fullspace 
        =
        \wavespaceV \times \heatspaceV \cap 
        \monosetc{ 
         \jmp{\wavevarTwo \times \nint}_\Fint = - (c_\ell^T\otimes I_2)\heatvar
        },
    \end{equation*}
    together with the extended forms
    \begin{equation*}
        \coupopfullformext \equiv 0, 
        \qquad
        \sbilwaveext (x,y) = p(\curl \wavevarOne , y_2) - p(\pcurl\, \wavevarTwo, y_1) .
    \end{equation*}
    We verify Assumption~\ref{ass:cont_forms}~(c) and prove well-posedness of the example in Appendix~\ref{appendix:Maxwell} below.

\section{Spatial discretization}
\label{sec:space}

We now turn to the discretization of \eqref{eq:full_system} in space. 
We therefore consider the finite dimensional spaces
$\wavespaceVh$ and $\heatspaceVh$ for the discrete wave function $\wavevarh$ and discrete heat function $\heatvarh$. Since we do not only consider conforming methods for the wave problem, we do not assume the discrete spaces $\wavespaceVh$ to be subspaces of $\wavespaceV$, but only
\begin{equation} \label{eq:emb_discr_spaces}
    \wavespaceVh \subset \wavespace
    \quad
    \text{and}
    \quad
    \heatspaceVh \subset \heatspaceV.
\end{equation} 
In addition, we also use the subspaces with the weaker norms induced by 
$\ipwave$ and $\ipheat$ and denote them by $\wavespaceh$ and $\heatspaceh$, respectively.
Thus, we can keep the forms $\ipwave$, $\ipheat$ and $\sbilheat$ for the spatial discretization and only focus on the discrete counterparts of
$\sbilwave$ and $\coupopfullform$.
We consider
\begin{align*}
    \sbilwaveh \colon (\wavespaceV + \wavespaceVh) \times ( \wavespaceV  + \wavespaceVh ) \to \R, 
\qquad
   \coupopfullformh \colon \heatspaceh \times ( \wavespaceV  + \wavespaceVh ) \to \R \,,
\end{align*}
and can thus formulate the semi discrete problem as follows:

Seek $\wavevarh \in C^1([0,T],\wavespaceVh)$
and
$\heatvarh \in C^1([0,T],\heatspaceVh)$
such that
\begin{subequations} \label{eq:full_system_space}
\begin{align}
 \ipwave(\pt \wavevarh, y_h)
+
\sbilwaveh(\wavevarh, y_h)
&= - \coupopfullformh(\heatvarh, y_h) +
\ipwave(f_h,y_h) ,
 \label{eq:full_system_wave_space}
    \\
\ipheat(\pt \heatvarh, \varphi_h)
+
\sbilheat(\heatvarh, \varphi_h)
&= 
\coupopfullformh(\varphi_h, \wavevarh) +
\ipheat(g_h,\varphi_h) ,
\label{eq:full_system_heat_space}
\end{align}
\end{subequations}
holds for all $y_h \in \wavespaceVh$ and $\varphi_h \in \heatspaceVh$
with initial conditions
 $\wavevarh(0) = \wavevar_{h,0}$ and $\heatvarh(0) = \heatvar_{h,0}$.

\begin{remark}
    We could also treat the more general case of a non-conforming method for the heat part, where the second part of \eqref{eq:emb_discr_spaces} is not satisfied, for example using mass lumping or symmetric interior penalty dG for the heat part. However, this only leads to a repetition of the arguments from the wave part, and we thus omit this for the sake of readability.  
\end{remark}

In the following, we make two structural assumptions on the discrete form for the wave part. The first one guarantees that up to a stabilization term, the skew-adjointness is preserved.
Furthermore, we assume that the stabilization only acts on discrete functions.
\begin{assumption} \label{ass:discr_forms_skew_stab}
(a) The discrete form $\sbilwaveh$ can be decomposed into a skew-symmetric and positive (stabilization) part as
$
\sbilwaveh 
=
\sbilwavehskew
+
\sbilwavehpos
$
which satisfy
\begin{align*}
\sbilwavehskew(y_h, y_h) = 0 ,
\qquad
\sbilwavehpos(y_h, y_h) 
\geq 0 
\end{align*}
for all $y_h \in \wavespaceVh$.
We additionally require symmetry:
\(\sbilwavehpos(x_h,y_h)=\sbilwavehpos(y_h,x_h)\) for all
\(x_h,y_h\in\wavespaceVh\).

(b) For all $y \in \wavespaceV$, the stabilization vanishes, i.e.,
\begin{align*}
\sbilwavehpos(y, y_h) = 0 
\end{align*}
for all $y_h \in \wavespaceVh$.   
\end{assumption}

The next assumption is a consistency condition which allows us to incorporate coupling information of the wave part from the domain into the error analysis via integration by parts, see Remark~\ref{rem:Ass_ibp}.
\begin{assumption} \label{ass:consistency}
    For $(\wavevar,\heatvar) \in \fullspace$ and $y_h \in \wavespaceh$ it holds 
      the generalized consistency condition
\begin{align*}
  \sbilwaveh( \wavevar, y_h)
    +
    \coupopfullformh(  \heatvar, y_h)
    =
    \sbilwaveext( \wavevar, y_h)
    + 
    \coupopfullformext(\heatvar, y_h),
\end{align*}
where $\sbilwaveext$ and $\coupopfullformext$ are defined in Assumption~\ref{ass:cont_forms}.
\end{assumption}

The positive part in our discretization stems from numerical stabilization such as upwinding, commonly used with discontinuous Galerkin methods.
We associate with it a semi-norm
$    \stabnormwave{y_h}^2 = \sbilwavehpos(y_h, y_h) $
and note, due to the skew-adjointness of $\sbilwavehskew$, the relation
\begin{equation} \label{eq:rel_sbilwaveh_pos}
    \sbilwaveh(y_h, y_h)  = | y_h |_{\sbilwavehpos}^2 .
\end{equation}
For the best approximation result, we further require several projections into the discrete spaces. We consider the  orthogonal projections
\begin{equation*}
\projwave \colon \wavespace \to \wavespaceVh ,
\quad
\projheat \colon \heatspace \to \heatspaceVh ,
\end{equation*}
with respect to the inner products $\ipwave$ and $\ipheat$, respectively,
and also the reference operators
\begin{equation*}
\projwaveRef \colon \wavespaceV \to \wavespaceVh
\quad
\projheatRef \colon \heatspaceV \to \heatspaceVh .
\end{equation*}
In the spirit of \cite{Hip.Hoc.Sto.2019},
we define the following errors terms
in order to keep track of the different contributions and to estimate them for the specific examples:
\begin{subequations} \label{eq:def_space_errors}
    \begin{align}
    \Delta p ( \wavevar,y_h ) 
    &=
    \ipwave( \projwaveRef \wavevar - \wavevar , y_h) ,
    \label{eq:def_space_errors_Delta_p}
    \\
    \Delta_{s,b} ( \wavevar, \heatvar,y_h )  
    &=
    \sbilwaveh( \projwaveRef \wavevar, y_h) 
    +
    \coupopfullformh( \projheatRef \heatvar, y_h) 
    -
    \sbilwaveext( \wavevar, y_h) 
    -
    \coupopfullformext(\heatvar, y_h) ,
    \label{eq:def_space_errors_Delta_s_b}
    \\
    \Delta f (y_h)
    &=
    \ipwave(f_h - f,y_h) ,
    \label{eq:def_space_errors_f}
    \\
    \Delta \ipheat( \heatvar,\varphi_h ) 
    &=
    \ipheat( \projheatRef \heatvar - \heatvar, \varphi_h) ,
    \label{eq:def_space_errors_Delta_t_p}
    \\
    \widetilde{\Delta}_{s,b} ( \wavevar, \heatvar,\varphi_h ) 
    &=
    \sbilheat( \projheatRef \heatvar - \heatvar, \varphi_h)
    -
    \coupopfullformh( \varphi_h , \projwaveRef \wavevar)
    + 
    \coupopfullform( \varphi_h ,  \wavevar ) , 
    \label{eq:def_space_errors_Delta_t_s_b}
    \\
    \Delta g  (\varphi_h)
    &=
     \ipheat(g_h - g,\varphi_h) .
    \label{eq:def_space_errors_g}
    \end{align}
\end{subequations}
Here, the errors in \eqref{eq:def_space_errors_Delta_p}
and \eqref{eq:def_space_errors_Delta_t_p} constitute best-approximation errors,
the terms \eqref{eq:def_space_errors_Delta_s_b} and \eqref{eq:def_space_errors_Delta_t_s_b} contain the approximation of the differential and coupling operators, and 
\eqref{eq:def_space_errors_f} and \eqref{eq:def_space_errors_g}
describe the approximation of the right-hand sides. This allows us to formulate the following abstract error result for the discrete errors 
\begin{equation} \label{eq:def_space_error}
\waveerrh(t) = \projwaveRef \wavevar(t) - \wavevarh(t) ,
\qquad
\heaterrh(t) = \projheatRef \heatvar(t) - \heatvarh(t).
\end{equation}
The full error is then decomposed in the discrete error and some best-approximation error. 

\begin{theorem} \label{thm:space}
Let $\wavevar \in C^1([0,T],\wavespaceV)$
and
$\heatvar \in C^1([0,T],\heatspaceV)$ be the solution of
 \eqref{eq:full_system},
 which satisfies
 $(\wavevar, \heatvar) \in C^1([0,T],\fullspace)$,
and
$(\wavevarh,\heatvarh)$ the solution of \eqref{eq:full_system_space}.
Further, let Assumptions~\ref{ass:cont_forms}, \ref{ass:discr_forms_skew_stab},
and \ref{ass:consistency} hold true.
Then, the errors defined in \eqref{eq:def_space_error} satisfy
\begin{equation*} 
    \begin{aligned}
   & \| \waveerrh(t) \|_{\wavespaceh}^2
    +
    \| \heaterrh(t) \|_{\heatspaceh}^2
    +
      \int_0^t 
     |\waveerrh(t) |_{\sbilwavehpos}^2
     \,ds
    \\
    &
     \qquad
     \lesssim 
     \| \waveerrh(0) \|_{\wavespaceh}^2 
     +
     \| \heaterrh(0) \|_{\wavespaceh}^2 
     +
     \int_0^t 
     \|\Delta \ipwave (\pt \wavevar)\|_{\wavespaceh^*}^2
    +
    \|\Delta \ipheat (\pt \heatvar)\|_{\heatspaceh^*}^2
     \,ds
          \\
     &
     \qquad
     \quad
     +
     \bigl|
     \int_0^t 
      \Delta_{s,b} ( \wavevar, \heatvar, \waveerrh)  
      \,ds
      \bigr|
    +
     \int_0^t  \| \widetilde{\Delta}_{s,b} ( \wavevar, \heatvar)   \|_{\heatspaceh^*}^2
     \,ds
     +
     \int_0^t 
     \|\Delta f \|_{\wavespaceh^*}^2
    +
    \|\Delta g\|_{\heatspaceh^*}^2
     \,ds
\end{aligned}
\end{equation*}
with constants independent of $h$.
\end{theorem}

We note that in order to later treat upwinded schemes, we have to leave the 
term $\Delta_{s,b}$ on the right-hand side, since, depending on the chosen stabilization, the error $\waveerrh$ is estimated either in the standard norm or the stabilization semi-norm leading to different orders of convergence in the spatial parameter $h$. 
This will be discussed in detail in the corresponding examples below.

\begin{proof}[Proof of Theorem~\ref{thm:space}]
We insert the projected exact solutions into the discretized equations \eqref{eq:full_system_space} to obtain
\begin{align*}
&
\ipwave(\pt \projwaveRef \wavevar, y_h)
+
\sbilwaveh( \projwaveRef \wavevar, y_h)
+
\coupopfullformh( \projheatRef \heatvar, y_h)
-
\ipwave(f_h,y_h)
\\
&\qquad +
\ipheat(\pt \projheatRef \heatvar, \varphi_h)
+
\sbilheat( \projheatRef \heatvar,  \varphi_h)
-
\coupopfullformh( \varphi_h , \projwaveRef \wavevar)
-
\ipheat(g_h,\varphi_h)
\\
=\, 
& 
\ipwave( \pt \wavevar, y_h)
+
\Delta \ipwave (\pt \wavevar , y_h ) 
+
\sbilwaveext(\wavevar,y_h)
+
\coupopfullformext(\heatvar, y_h)
+
\Delta_{s,b} ( \wavevar, \heatvar,y_h ) 
-
\ipwave(f,y_h)
-
\Delta f (y_h)
\\
&\,
+ 
\ipheat (\pt \heatvar , \varphi_h ) 
+
\Delta \ipheat (\pt \heatvar , \varphi_h ) 
+
\sbilheat(  \heatvar, \varphi_h)
-
\coupopfullform(  \varphi_h, \wavevar)\\
&+
\widetilde{\Delta}_{s,b} ( \wavevar, \heatvar,\varphi_h  ) 
-
\ipheat(g,\varphi_h)
-
\Delta g(\varphi_h)
\\
=\,
& 
\Delta p (\pt \wavevar , y_h ) 
+
\Delta_{s,b} ( \wavevar, \heatvar,y_h ) 
-
\Delta f (y_h)
+
\widetilde{p} (\pt \heatvar , \varphi_h ) 
+
\widetilde{\Delta}_{s,b} ( \wavevar, \heatvar,\varphi_h  ) 
-
\Delta g(\varphi_h) ,
\end{align*}
where we used in the last step that \eqref{eq:full_system} also holds when tested with discrete functions $(y_h,\varphi_h) \in \wavespaceh \times \heatspaceh$, see \eqref{eq:first_eq_tested_in_X}.
Subtracting \eqref{eq:full_system_space}, we obtain the error equation
\begin{align*}
&
\ipwave(\pt  \waveerrh, y_h)
+
\sbilwaveh( \waveerrh, y_h)
+
\coupopfullformh( \heaterrh, y_h)
 +
\ipheat(\pt \heaterrh , \varphi_h)
+
\sbilheat( \heaterrh ,  \varphi_h)
-
\coupopfullformh( \varphi_h , \waveerrh)
\\
= \,
& 
\Delta p (\pt \wavevar , y_h ) 
+
\Delta_{s,b} ( \wavevar, \heatvar,y_h ) 
-
\Delta f (y_h)
+
\widetilde{p} (\pt \heatvar , \varphi_h ) 
+
\widetilde{\Delta}_{s,b} ( \wavevar, \heatvar,\varphi_h  ) 
-
\Delta g(\varphi_h) .
\end{align*}
Testing with $(y_h,\varphi_h) = (\waveerrh,\heaterrh)$, using \eqref{eq:rel_sbilwaveh_pos}, integrating in time, and finally a weighted Gronwall argument yields
the claim. 
\end{proof}

\section{Full discretization}
\label{sec:full}

Since we aim to apply a leapfrog-type method to the wave system, we need to exploit again the two-field structure of the problem and assume that also the discrete bilinear form $\sbilwaveh$ can be decomposed as 
\begin{equation*} 
    \sbilwavehskew(x_h,y_h) = \sbilwavehOne(x_{h,1},y_{h,2}) - \sbilwavehTwo(x_{h,2},y_{h,1}),
\end{equation*}
and we have the representation
\begin{equation*} 
    \sbilwavehpos(x_h,y_h) = \sbilwavehposOne(x_{h,1},y_{h,1}) + \sbilwavehposTwo(x_{h,2},y_{h,2})
\end{equation*}
for the stabilization.
Further, the coupling term is given by
\begin{equation*} 
   \coupopfullformh(\heatvarh, y_h)  =    \coupopformh(\heatvarh, y_{h,1}) ,
\end{equation*}
see \eqref{eq:coup_f_decomp}. With this, we can formulate our problem equivalently as
\begin{subequations}  
\begin{align}
 \ipwaveOne(\pt \wavevarhOne, y_{h,1})
-
\sbilwavehTwo(\wavevarhTwo, y_{h,1})
&=
 - \coupopformh(\heatvarh, y_{h,1}) +
\ipwaveOne(f_{h,1},y_{h,1})
\\
 \ipwaveTwo(\pt \wavevarhTwo, y_{h,2})
+
\sbilwavehOne(\wavevarhOne, y_{h,2})
&=
0
\\
\ipheat(\pt \heatvarh, \varphi_h)
+
\sbilheat(\heatvarh, \varphi_h)
&= 
\coupopformh(\varphi_h, \wavevarhOne) +
\ipheat(g_h,\varphi_h)
\end{align}
\end{subequations}
or in the strong form as
\begin{subequations}\label{eq:SpatialDiscreteStrongForm} 
\begin{align}
 \pt \wavevarhOne
-
\waveophTwo \wavevarhTwo + \waveophposOne \wavevarhOne
&=
 - \coupoph \heatvarh + f_{h,1} 
\label{eq:SpatialDiscreteStrongFormWaveOne}\\
 \pt \wavevarhTwo
+
\waveophOne \wavevarhOne + \waveophposTwo \wavevarhTwo
&=
0
\label{eq:SpatialDiscreteStrongFormWaveTwo}\\
\pt \heatvarh
+
\heatoph \heatvarh
&= 
\coupoph^*  \wavevarhOne
 +
g_h ,
\label{eq:SpatialDiscreteStrongFormHeat}
\end{align}
\end{subequations}
which is the starting point for our time integration scheme.

\subsection{Derivation of the method}

In the following, we denote the time step size as $\tau > 0$ and define the times 
$t_n = n \tau$, $n\geq 0$.
In order to see the structural similarity of our fully discrete method,
we formulate the system as a single first-order evolution equation
with variable 
$\colvar = \bigl( \wavevarOne, \wavevarTwo, \heatvar \bigr)$
in the form
\begin{equation}\label{eq:EvolutionEquation}
    \pt \colvar(t)  
        + \colop \colvar(t)
        =   \colrhs(t),
    \qquad
    \colop =
    \begin{pmatrix}
        0 & -\waveopTwo & \coupop\\
        \waveopOne & 0 & 0 \\
        -\coupop^\ast & 0 & \heatop
    \end{pmatrix},\,\,\,
    \qquad
    \colrhs = 
    \begin{pmatrix}
        f_1 \\ 0 \\ g
    \end{pmatrix} \,.
\end{equation}
Further, we introduce the fully discrete approximations and their collections 
\begin{equation*}
    \wavevarhnOne{n} \approx \wavevarhOne(t_n),
    \quad
    \wavevarhnTwo{n} \approx \wavevarhTwo(t_n),
    \quad
    \heatvarhn{n} \approx \heatvarh(t_n),
   \quad
    \colvarhn{n} = \bigl( \wavevarhnOne{n}, \wavevarhnTwo{n} ,  \heatvarhn{n} \bigr).
\end{equation*}
We derive our time stepping scheme in two steps.
Assume for a moment that the field \(\heatvarh\) is known at all time steps.
The well-known leapfrog scheme for the wave equation described by \eqref{eq:SpatialDiscreteStrongFormWaveOne} and \eqref{eq:SpatialDiscreteStrongFormWaveTwo} in half-step formulation (staggered) then reads
\begin{subequations}\label{eq:LeapfrogHalfStep}
    \begin{align}
        \wavevarhnOne{\nhalf}
            &= \wavevarhnOne{n} 
                + \tauhalf \waveophTwo \wavevarhnTwo{n}
                - \tauhalf \waveophposOne \wavevarhnOne{n}
                -\tauhalf \coupoph \heatvarhn{n}
                + \tauhalf f_{h,1}^n,
            \label{eq:LeapfrogHalfStep:One}\\
        \wavevarhnTwo{n+1}
            &= \wavevarhnTwo{n} 
                - \tau \waveophOne \wavevarhnOne{\nhalf}
                - \tau \waveophposTwo \wavevarhnTwo{n},
            \label{eq:LeapfrogHalfStep:Two}\\
        \wavevarhnOne{n+1}
            &= \wavevarhnOne{\nhalf}
                + \tauhalf \waveophTwo \wavevarhnTwo{n+1}
                - \tauhalf \waveophposOne \wavevarhnOne{n}
                -\tauhalf \coupoph \heatvarhn{n+1}
                + \tauhalf f_{h,1}^{n+1}.
            \label{eq:LeapfrogHalfStep:Three}
    \end{align}
\end{subequations}
Here, \(\wavevarhnOne{\nhalf} \approx \frac{1}{2}\bigl(\wavevarhnOne{n+1} + \wavevarhnOne{n}\bigr)\) is a second order in \(\tau\) approximation to the mean value.
Since the stabilization terms break the Hamiltonian structure of the discrete wave equation, it is common practice to apply them in such a way that the scheme stays explicit.
Note that \eqref{eq:LeapfrogHalfStep:One} and \eqref{eq:LeapfrogHalfStep:Two} are explicitly computable from known values at time \(t = t_n\).
Step \eqref{eq:LeapfrogHalfStep:Three}, on the other hand, is not explicitly computable since it demands the yet unknown value \(\heatvarhn{n+1}\).
In hope to omit any CFL condition from the heat part, we introduce a \CN step for \eqref{eq:SpatialDiscreteStrongFormHeat}, which reads
\begin{equation}\label{eq:CNStepHeat}
    \heatvarhn{n+1} 
        = \heatvarhn{n} 
        - \tauhalf \heatoph\bigl(\heatvarhn{n+1} + \heatvarhn{n}\bigr)
        + \tauhalf \coupoph^\ast\bigl(\wavevarhnOne{n+1} + \wavevarhnOne{n}\bigr)
        + \tauhalf \bigl(g_h^{n+1} + g_h^{n}\bigr).
\end{equation}
The average of \(\wavevarhOne\) in \eqref{eq:CNStepHeat} makes the coupling of \eqref{eq:CNStepHeat} and \eqref{eq:LeapfrogHalfStep:Three} implicit, which we want to avoid.
Circumventing this, we replace the average with the already calculated half-step \(\wavevarhnOne{\nhalf}\) and obtain the approximation
\begin{equation}\label{eq:CNStepHeatApprox}
    \heatvarhn{n+1} 
        \approx \heatvarhn{n} 
        - \tauhalf \heatoph\bigl(\heatvarhn{n+1} + \heatvarhn{n}\bigr)
        + \tau\coupoph^\ast\wavevarhnOne{\nhalf}
        + \tauhalf \bigl(g_h^{n+1} + g_h^{n}\bigr).
\end{equation}
Since the half-step is a second order in \(\tau\) approximation of the average, we expect the scheme described by the coupling of \eqref{eq:LeapfrogHalfStep} and \eqref{eq:CNStepHeatApprox} to keep its good properties and convergence rate.

Summarizing this, our proposed time stepping scheme reads
\begin{subequations}\label{eq:LFCN}
    \begin{align}
                \wavevarhnOne{\nhalf}
            &= \wavevarhnOne{n} 
                + \tauhalf \waveophTwo \wavevarhnTwo{n}
                - \tauhalf \waveophposOne \wavevarhnOne{n}
                -\tauhalf \coupoph \heatvarhn{n}
                + \tauhalf f_{h,1}^n,
            \label{eq:LFCN:One}\\
        \wavevarhnTwo{n+1}
            &= \wavevarhnTwo{n} 
                - \tau \waveophOne \wavevarhnOne{\nhalf}
                - \tau \waveophposTwo \wavevarhnTwo{n},
            \label{eq:LFCN:Two}\\
        \heatvarhn{n+1}
            &= \heatvarhn{n} 
            - \tauhalf \heatoph\bigl(\heatvarhn{n+1} + \heatvarhn{n}\bigr)
            + \tau\coupoph^\ast\wavevarhnOne{\nhalf}
            + \tauhalf \bigl(g_h^{n+1} + g_h^{n}\bigr)
            \label{eq:LFCN:Three},\\
        \wavevarhnOne{n+1}
            &= \wavevarhnOne{\nhalf}
                + \tauhalf \waveophTwo \wavevarhnTwo{n+1}
                - \tauhalf \waveophposOne \wavevarhnOne{n}
                -\tauhalf \coupoph \heatvarhn{n+1}
                + \tauhalf f_{h,1}^{n+1}.
            \label{eq:LFCN:Four}
    \end{align}
\end{subequations}
The rest of this manuscript is devoted to proving stability and convergence of the scheme, and to applying the results to our model examples.

\subsection{Stability}

We prove stability by rewriting \eqref{eq:LFCN} as a perturbed evolution equation, deploying energy estimates and closing with a Gronwall-type argument.
Let \( \{\alpha_n\}_{n\geq 0} \subset \mathbb{R}^\ell\) be a sequence for given \(\ell \geq 1\).
We define the average and discrete derivative of the series as
\begin{equation*}
    \tauavg{\alpha^{n+1}} = \frac{\alpha^{n+1} + \alpha^n}{2},\quad
        \dtau \alpha^{n+1} = \frac{\alpha^{n+1} - \alpha^{n}}{\tau},\qquad 
        n\geq 0.
\end{equation*}
We derive the perturbed evolution equation of \eqref{eq:LFCN}.
Subtraction of \eqref{eq:LFCN:Four} from \eqref{eq:LFCN:One} and rearranging gives
\begin{equation}\label{eq:HalfStepRepresentation}
    \wavevarhnOne{\nhalf} 
        = \tauavg{\wavevarhnOne{n+1}}
            - \tausqf \waveophTwo \dtau \wavevarhnTwo{n+1}
            + \tausqf \coupoph \dtau \heatvarhn{n+1}
            - \tausqf \dtau f_{h,1}^{n+1}.
\end{equation}
Addition of \eqref{eq:LFCN:One} and \eqref{eq:LFCN:Four} on the other hand shows
\begin{equation}\label{eq:LFCN:One:Representation}
    \dtau \wavevarhnOne{n+1} 
            -\waveophTwo \tauavg{\wavevarhnTwo{n+1}}
            + \waveophposOne \wavevarhnOne{n}
            + \coupoph \tauavg{\heatvarhn{n+1}}
            = \tauavg{f_{h,1}^{n+1}}.
\end{equation}
Insertion of \eqref{eq:HalfStepRepresentation} in \eqref{eq:LFCN:Two} gives
\begin{equation}\label{eq:LFCN:Two:Representation}
\begin{multlined}
    \dtau \wavevarhnTwo{n+1} 
        + \waveophOne \tauavg{\wavevarhnOne{n+1}} + \waveophposTwo \wavevarhnTwo{n}\\
            = \tausqf \waveophOne\waveophTwo \dtau \wavevarhnTwo{n+1}
            - \tausqf \waveophOne \coupoph \dtau \heatvarhn{n+1}
            + \tausqf  \waveophOne \dtau f_{h,1}^{n+1}.
\end{multlined}
\end{equation}
Again, insertion of \eqref{eq:HalfStepRepresentation} in \eqref{eq:LFCN:Three} shows
\begin{equation}\label{eq:LFCN:Four:Representation}
\begin{multlined}
    \dtau \heatvarhn{n+1} 
        + \heatoph \tauavg{\heatvarhn{n+1}}
            - \coupoph^\ast \tauavg{\wavevarhnOne{n+1}}\\
            = \tauavg{g_h^{n+1}} 
            -\tausqf \coupoph^\ast \waveophTwo \dtau \wavevarhnTwo{n+1}
            + \tausqf \coupoph^\ast \coupoph \dtau \heatvarhn{n+1}
            - \tausqf  \coupoph^\ast \dtau f_{h,1}^{n+1}.
\end{multlined}
\end{equation}
Gathering \eqref{eq:LFCN:One:Representation}, \eqref{eq:LFCN:Two:Representation} and \eqref{eq:LFCN:Four:Representation}, we obtain the perturbed discrete evolution equation for the collection \(\colvarhn{n} = \bigl(\wavevarhnOne{n},\wavevarhnTwo{n},\heatvarhn{n}\bigr)\)
\begin{equation}\label{eq:DiscreteEvolutionEquation}
    \bigl(1 - \tauhalf\colstaboph - \tausqf \colpertoph \bigr)\dtau \colvarhn{n+1} 
        + \bigl(\coloph + \colstaboph\bigr)\tauavg{\colvarhn{n+1}}
        =   \tauavg{\colrhshn{n+1}}
            - \tausqf \colpertrhsoph \dtau f_{h,1}^{n+1},\quad n\geq 0,
\end{equation}
where \(\colrhshn{n} = \bigl(f_{h,1}^{n}, 0, g_h^{n}\bigr)\) and
\begin{gather}
    \coloph =
    \begin{pmatrix}
        0 & -\waveophTwo & \coupoph\\
        \waveophOne  & 0 & 0 \\
        -\coupoph^\ast & 0 & \heatoph
    \end{pmatrix},\,\,\,
    \colstaboph = 
    \begin{pmatrix}
        \waveophposOne & 0 & 0\\
        0 & \waveophposTwo & 0 \\
        0 & 0 & 0
    \end{pmatrix},
    \label{eq:DiscreteOperators:EvolutionOph}\\
    \colpertoph = 
    \begin{pmatrix}
        0 & 0 & 0\\
        0 & \waveophOne\waveophTwo & -\waveophOne \coupoph\\
        0 & -\coupoph^\ast \waveophTwo & \coupoph^\ast \coupoph
    \end{pmatrix},\,\,\,
    \colpertrhsoph = 
    \begin{pmatrix}
        0 \\ -\waveophOne \\ \coupoph^\ast
    \end{pmatrix}  ,
    \label{eq:DiscreteOperators:Perturbations}
\end{gather}
and the application of $\colpertrhsoph$ is understood component wise. 

\begin{lemma}\label{lem:ColophAccretive}
    Let Assumptions~\ref{ass:cont_forms} and \ref{ass:discr_forms_skew_stab} hold.
    
    \noindent(a) 
            The operator \(\coloph\,:\, \wavespaceVh \times \heatspaceVh \to \wavespaceVh \times \heatspaceVh\) is accretive, i.e.,
            \[
                \ipcol{\coloph \coltestvarh}{\coltestvarh} \geq 0,\quad \forall \coltestvarh \in \wavespaceVh \times \heatspaceVh.
            \]
    
    \noindent(b)    
            Let $\coltestvarh , \coltestvarhVTwo\in \wavespaceVh \times \heatspaceVh$. If the last component of $\coltestvarh$ or $\coltestvarhVTwo$ is zero, then
\begin{equation} \label{eq:cond_skew_adj}
\ipcol{\coloph \coltestvarh}{\coltestvarhVTwo}
=
-
\ipcol{ \coltestvarh}{\coloph \coltestvarhVTwo} .
\end{equation} 
\end{lemma}
\begin{proof}(a)
    Let \(\coltestvarh = \bigl(y_{h,1},y_{h,2},\varphi_h\bigr) \in \wavespaceVh\times \heatspaceVh\).
    By definition, we obtain
    \begin{align*}
        \ipcol{\coloph \coltestvarh}{\coltestvarh}
            &= 
            \begin{multlined}[t]
                \ipwaveOne({-\waveophTwo y_{h,2}},{y_{h,1}})        
                    + \ipwaveOne({\coupoph \varphi_h},{y_{h,1}})
                    + \ipwaveTwo({\waveophOne y_{h,1}},{y_{h,2}})\\
                    + \ipheat({-\coupoph^\ast y_{h,1}},{\varphi_h}) + \ipheat({\heatoph \varphi_h},{\varphi_h})
            \end{multlined}\\
            &= 
            \begin{multlined}[t]
                \sbilwavehOne({y_{h,1}},{y_{h,2}})
                -\sbilwavehTwo({y_{h,2}},{y_{h,1}}) 
                + \sbilheat({\varphi_h},{\varphi_h})\\
                + \coupopformh(\varphi_h,y_{h,1}) - \coupopformh(\varphi_h, y_{h,1})
            \end{multlined}\\ 
            &= \sbilwavehskew(y_h,y_h)
                + \sbilheat(\varphi_h,\varphi_h)\\
            &= \sbilheat(\varphi_h,\varphi_h)\\
            &\geq 0,
    \end{align*}
    where \(y_h = (y_{h,1},y_{h,2})\) and Assumption~\ref{ass:cont_forms} and \ref{ass:discr_forms_skew_stab} were used in the last inequality.
    The second claim (b) is a direct consequence of Assumption~\ref{ass:discr_forms_skew_stab}.
\end{proof}

In the following, we frequently use the discrete version of differentiation of a square norm.
Let \(\{\alpha^{n}\}_{n\geq 0} \subset \mathbb{R}^{\ell}\), for \(\ell \geq 1\), and \(|\cdot|_{\ell}\) denote the Euclidean norm.
It holds that
\begin{equation}\label{eq:DiscreteFundamental}
    \dtau \frac{1}{2} |\alpha^n|_{\ell}^2
    = \frac{1}{2\tau}\bigl(
        |\alpha^{n+1}|_\ell^2 - |\alpha^{n}|_\ell^2
        \bigr) 
    = \dtau \alpha^n \cdot \tauavg{\alpha^{n}}.
\end{equation}

We obtain the following equality for the perturbation.
\begin{lemma}\label{lem:DiscretePertEstimate}
    Let \(\{\coltestvarhn{n}\}_{n\geq 0} \subset \wavespaceVh \times \heatspaceVh\). It holds for \(n\geq 0\)
    \[
        \frac{\tau^2}{4} \ipcol{\colpertoph \dtau \coltestvarhn{n+1}}{\tauavg{\coltestvarhn{n+1}}}
        = \frac{\tau}{8}\Bigl(
            \normwaveOne{\waveophTwo \wavevarhnTwo{n+1} - \coupoph \heatvarhn{n+1}}^2
            -\normwaveOne{\waveophTwo \wavevarhnTwo{n} - \coupoph \heatvarhn{n}}^2
        \Bigr) .
    \]
\end{lemma}
\begin{proof}
    We use the definition in \eqref{eq:DiscreteOperators:Perturbations}, and compute%
    \begin{align*}
        \ipcol{\colpertoph \dtau \coltestvarhn{n+1}}{\tauavg{\coltestvarhn{n+1}}}
        &=  
            \begin{multlined}[t]
                \ipwaveTwo(\waveophOne \waveophTwo \dtau \wavevarhnTwo{n+1}, \tauavg{\wavevarhnTwo{n+1}})
                - \ipwaveTwo(\waveophOne \coupoph \dtau \heatvarhn{n+1}, \tauavg{\heatvarhn{n+1}})\\
                - \ipheat(\coupoph^\ast \waveophTwo \dtau \wavevarhnTwo{n+1}, \tauavg{\heatvarhn{n+1}})
                + \ipheat(\coupoph^\ast \coupoph \dtau \heatvarhn{n+1}, \tauavg{\heatvarhn{n+1}})
            \end{multlined}\\
        &=
            \begin{multlined}[t]
                \ipwaveOne(\waveophTwo \dtau \wavevarhnTwo{n+1}, \waveophTwo \tauavg{\wavevarhnTwo{n+1}})
                - \ipwaveOne(\coupoph \dtau \heatvarhn{n+1}, \waveophTwo\tauavg{\wavevarhnTwo{n+1}})\\
                - \ipwaveOne(\waveophTwo \dtau \wavevarhnTwo{n+1}, \coupoph \tauavg{\heatvarhn{n+1}})
                + \ipwaveOne(\coupoph \dtau \heatvarhn{n+1}, \coupoph \tauavg{\heatvarhn{n+1}})
            \end{multlined}\\
        &=
            \begin{multlined}[t]
                \ipwaveOne(
                    \waveophTwo\dtau \wavevarhnTwo{n+1} - \coupoph \dtau \heatvarhn{n+1},
                    \waveophTwo \tauavg{\wavevarhnTwo{n+1}} - \coupoph \tauavg{\heatvarhn{n+1}}    
                )
            \end{multlined}\\
        &=
            \frac{1}{2\tau}\Bigl(
                \normwaveOne{\waveophTwo \wavevarhnTwo{n+1} - \coupoph \heatvarhn{n+1}}^2
                -\normwaveOne{\waveophTwo \wavevarhnTwo{n} - \coupoph \heatvarhn{n}}^2
            \Bigr),
    \end{align*}
    where we employed \eqref{eq:DiscreteFundamental} in the last step.
\end{proof}

This allows us to conclude the following energy inequality for some general right-hand side. Later on, this estimate will also be used for the error recursion and the defects as right-hand sides.
\begin{lemma} \label{lem:stab_1}
    Let \(K^{n+1}_h = (k_{h,1}^{n+1}, k_{h,2}^{n+1}, \kappa_h^{n+1}) \in \wavespaceh \times \heatspaceh\) for \(n\geq 0\).
    Then, the recursion
    \begin{equation}\label{eq:GenericRecursion}
        \bigl(1 -\tauhalf \colstaboph - \tausqf \colpertoph\bigr)\dtau \colvarhn{n+1} 
            + \bigl(\coloph + \colstaboph\bigr) \tauavg{\colvarhn{n+1}}
        =   
            {K_{h}^{n+1}}
            ,\quad n\geq 0,
    \end{equation}
    satisfies for \(n\geq 0\) the estimate
    \begin{equation}\label{eq:GenericRecursionEstimate}
    \begin{multlined}
        \normcol{\colvarhn{n+1}}^2
        - \tauhalf \stabnormwave{\wavevarhn{n+1}}^2
        - \tausqf \normwaveOne{\waveophTwo \wavevarhnTwo{n+1} - \coupoph \heatvarhn{n+1}}^2
        + 2\tau \stabnormwave{\tauavg{\wavevarhn{n+1}}}^2
        \\ 
        \leq
        \normcol{\colvarhn{n}}^2 
        - \tauhalf \stabnormwave{\wavevarhn{n}}^2
        - \tausqf \normwaveOne{\waveophTwo \wavevarhnTwo{n} - \coupoph \heatvarhn{n}}^2
        + 2\tau \ipcol{{K_h^{n+1}}}{\tauavg{\colvarhn{n+1}}}.
    \end{multlined}
    \end{equation}
\end{lemma}
\begin{proof}
    We test \eqref{eq:GenericRecursion} with \(\tauavg{\colvarhn{n+1}}\) and obtain for the first term
    and \eqref{eq:DiscreteFundamental} that
    \begin{equation*}
        \ipcol{\dtau \colvarhn{n+1}}{\tauavg{\colvarhn{n+1}}}
            = \frac{1}{2\tau} \Bigl(\normcol{\colvarhn{n+1}}^2 - \normcol{\colvarhn{n}}^2\Bigr).
    \end{equation*}
    For the second term we obtain again with \eqref{eq:DiscreteFundamental}
    \begin{equation*}
        \tauhalf\ipcol{\colstaboph \dtau \colvarhn{n+1}}{\tauavg{\colvarhn{n+1}}}
            = \frac{1}{4}\stabnormwave{\wavevarhn{n+1}}^2 - \frac{1}{4}\stabnormwave{\wavevarhn{n}}^2 .
    \end{equation*}
    The third term is already controlled by Lemma~\ref{lem:DiscretePertEstimate}, thus it remains to estimate the fourth term.
    With Lemma~\ref{lem:ColophAccretive}, we see that
    \begin{equation*}
        \stabnormwave{\tauavg{\wavevarhn{n+1}}}^2 
            = \ipcol{\colstaboph \tauavg{\colvarhn{n+1}}}{\tauavg{\colvarhn{n+1}}}
            \leq \ipcol{\bigl(\coloph + \colstaboph\bigr) \tauavg{\colvarhn{n+1}}}{\tauavg{\colvarhn{n+1}}}
    \end{equation*}
    Gathering all estimates and multiplication with \(2\tau\) proves the claim.
\end{proof}

We proceed to prove stability under the following CFL condition.

\begin{assumption}\label{ass:CFL}
    Given \(\eta \in (0,1)\) there exists \(\tauCFL>0\) such that for all \(\tau \leq\tauCFL\) and \(\colvarh = (\wavevarh, \heatvarh) \in \wavespaceVh \times \heatspaceVh\) it holds 
    \begin{equation}\label{eq:CFL}
        \frac{\tau^2}{2}\normwaveOne{\waveophTwo \wavevarhTwo}^2  
        + \frac{\tau^2}{2}\normwaveOne{\coupoph \heatvarh}^2
            +  \tau \stabnormwave{\wavevarh}^2 \leq \eta \normcol{\colvarh}^2.
    \end{equation}
\end{assumption}

We note that the assumption directly implies the bound
\begin{equation} \label{eq:CFL_for_stab}
    \tau \|\colstaboph \|  \leq \eta,
\end{equation}
which is used throughout the error analysis.

\begin{remark}
    Note that we expect the first two norms in \eqref{eq:CFL} to scale inversely with the minimal mesh size \(h_{\operatorname{min}}\), and the third with \(h_{\operatorname{min}}^{1/2}\).
    Thus, we consider a step-size restriction \(\tau \approx  h_{\operatorname{min}}\).
    Usually, the stabilization depends on a user-given parameter set.
    The second norm in \eqref{eq:CFL} is expected to scale inversely with these parameters.
    Hence, the CFL condition becomes slightly worse with stronger stabilization.
\end{remark}

\begin{theorem}\label{thm:StabilityGenericRecursion}
    Let Assumption~\ref{ass:CFL} hold.
    For the recursion \eqref{eq:GenericRecursion} it holds
    \begin{equation*}
        \begin{multlined}
            (1-\eta)\normcol{\colvarhn{n+1}}^2
                + \tauhalf \sum_{\ell = 0}^n \stabnormwave{{\wavevarhn{\ell+1}} + {\wavevarhn{\ell}}}^2\\
            \leq e^{3/2} \normcol{\colvarhn{0}}^2 
                + \frac{e^{3/2}T}{(1-\eta)} \tau \sum_{l=0}^{n+1} \normcol{K_h^{\ell +1}}^2.
        \end{multlined}
    \end{equation*}
\end{theorem}
\begin{proof}
    Summation in \eqref{eq:GenericRecursionEstimate} shows
    \begin{equation*}
    \begin{multlined}
        \normcol{\colvarhn{n+1}}^2
        + 2\tau \sum_{\ell = 0}^n \stabnormwave{\tauavg{\wavevarhn{\ell+1}}}^2
        \leq
        \normcol{\colvarhn{0}}^2
        + \tauhalf \stabnormwave{\wavevarhn{n+1}}^2\\
        + \tausqf \normwaveOne{\waveophTwo \wavevarhnTwo{n+1} - \coupoph \heatvarhn{n+1}}^2
        + 2\tau \sum_{\ell=0}^n \ipcol{{K_h^{\ell+1}}}{\tauavg{\colvarhn{\ell+1}}}.
    \end{multlined}
    \end{equation*}
    With \eqref{eq:CFL}, we conclude that
    \begin{equation*}
        (1-\eta) \normcol{\colvarhn{n+1}}^2
            + 2\tau \sum_{\ell = 0}^n \stabnormwave{\tauavg{\wavevarhn{\ell+1}}}^2
            \leq \normcol{\colvarhn{0}}^2
                + 2\tau \sum_{\ell=0}^n \ipcol{{K_h^{\ell+1}}}{\tauavg{\colvarhn{\ell+1}}}.
    \end{equation*}
    We estimate the remaining product with Young and obtain
    \begin{equation*}
        \begin{multlined}
            2\tau \sum_{\ell=0}^n \ipcol{{K_h^{\ell+1}}}{\tauavg{\colvarhn{\ell+1}}}
                \leq 
                 2 \frac{(1-\eta)}{2T} \tauhalf \sum_{\ell=0}^n
                 \Bigl(
                     \normcol{{\colvarhn{\ell+1}}}^2  
                    + \normcol{{\colvarhn{\ell}}}^2
                 \Bigr)
                \\
                    +\frac{T\tau}{(1-\eta)} \sum_{\ell=0}^n \normcol{K_h^{\ell+1}}^2.
        \end{multlined}
    \end{equation*}
    Therefore, we see that
    \begin{equation*}
        \begin{multlined}
            (1-\eta) \normcol{\colvarhn{n+1}}^2
            + \tauhalf \sum_{\ell = 0}^n \stabnormwave{{\wavevarhn{\ell+1}} + {\wavevarhn{\ell}}}^2\\
            \leq \normcol{\colvarhn{0}}^2
                + 2 \frac{(1-\eta)}{2T} \tauhalf \sum_{\ell=0}^n
                \Bigl(
                    \normcol{{\colvarhn{\ell+1}}}^2  
                    + \normcol{{\colvarhn{\ell}}}^2
                \Bigr)
                +\frac{T\tau}{(1-\eta)} \sum_{\ell=0}^n \normcol{K_h^{\ell+1}}^2.
        \end{multlined}
    \end{equation*}
    Finally, the application of a discrete Gronwall estimate 
    \cite[Proposition~4.1]{Emm99}
    proves the claim.
\end{proof}

Using the CFL condition \eqref{eq:CFL} and the special structure of the right-hand side, we may conclude the following stability estimate.

\begin{corollary}\label{cor:StabilityLFCN}
    Let Assumption~\ref{ass:CFL} hold. Then, the time integration scheme \eqref{eq:LFCN} is stable and
    \begin{equation*}
        \begin{multlined}
            (1-\eta)\normcol{\colvarhn{n+1}}^2
                + \tauhalf \sum_{\ell = 0}^n \stabnormwave{{\wavevarhn{\ell+1}} + {\wavevarhn{\ell}}}^2\\
            \leq e^{3/2} \normcol{\colvarhn{0}}^2 
                + \frac{2e^{3/2}T}{(1-\eta)} \tau \sum_{l=0}^{n+1} (1+\eta)\normwaveOne{f_{h,1}^\ell}^2 + \normheat{g_h^\ell}^2.
        \end{multlined}
    \end{equation*}
\end{corollary}
\begin{proof}
    First we note that the CFL condition \eqref{eq:CFL} on \(\waveophTwo\) implies the same condition on \(\waveophOne\) by adjointness.
    It holds that
    \begin{align*}
        \frac{\tau}{\sqrt{2}}\ipwaveTwo(\waveophOne \wavevarhOne, \wavevarhTwo)
        &= \frac{\tau}{\sqrt{2}} \ipwaveOne(\wavevarhOne, \waveophTwo \wavevarhTwo)\\
        &\leq \normwaveOne{\wavevarhOne} \frac{\tau}{\sqrt{2}} \normwaveOne{\waveophTwo\wavevarhTwo}\\
        &\leq \eta^{1/2} \normwaveOne{\wavevarhOne} \normcol{\colvarh},
    \end{align*}
    which implies that \(\tau^2\normwaveTwo{\waveophOne \wavevarhOne}^2 \leq 4\eta \normcol{\colvarh}^2\).
    Thus, we obtain 
    \begin{align*}
        \frac{\tau^4}{16} \normcol{\colpertrhsoph \dtau f_{h,1}^{n+1}}^2 
            &= \frac{\tau^4}{16} \normwaveTwo{\waveophOne \dtau f_{h,1}^{n+1}}^2
                + \frac{\tau^4}{16} \normheat{\coupoph^\ast \dtau f_{h,1}^{n+1}}^2\\
            &\leq  \eta \frac{\tau^2}{4} \normwaveOne{\dtau f_{h,1}^{n+1}}^2\\
            &\leq \eta\frac{1}{2} \bigl(\normwaveOne{f_{h,1}^{n+1}}^2 + \normwaveOne{f_{h,1}^{n}}^2\bigr).
    \end{align*}
    Furthermore, we see that
    \begin{equation*}
        \normcol{\tauavg{\colrhshn{n+1}}}^2 
        \leq \normwaveOne{f_{h,1}^{n+1}}^2 + \normwaveOne{f_{h,1}^{n}}^2
            + \normheat{g_{h}^{n+1}}^2 + \normheat{g_{h}^{n}}^2.
    \end{equation*}
    The claim now follows with
    \(K_h^{n+1} = \tauavg{\colrhshn{n+1}}
    - \tausqf \colpertrhsoph \dtau f_{h,1}^{n+1}\)
    and Theorem~\ref{thm:StabilityGenericRecursion}.
\end{proof}

\subsection{Error analysis}

In the following, we derive an error recursion in the form of \eqref{eq:DiscreteEvolutionEquation} in order to exploit the derived stability estimates, where the right-hand side will be given by the defects. 
Thus, we derive the defects in the first step. 
We begin with the defect of the \CN method and treat the rest as a perturbation. We introduce the full projections in the collection variable
\begin{equation*}
    \projfull = ( \projwave, \projheat ),
    \qquad
    \projReffull = ( \projwaveRef, \projheatRef ),
\end{equation*}
and abbreviate $\colvarn{n} = \colvar(t_n)$.
The defect $D_{\text{CN}}^{n+1}$ is given as
\begin{equation*} 
    \dtau \projReffull \colvarn{n+1} 
        + 
        \coloph \projReffull \tauavg{\colvarn{n+1}}
        +
         \colstaboph \projReffull \colvarn{n} 
        =  \tauavg{\colrhshn{n+1}} 
        +
         D_{\text{CN}}^{n+1} .
\end{equation*}
Inserting the exact solution \eqref{eq:EvolutionEquation},
the defect has the form
\begin{align*}
     D_{\text{CN}}^{n+1}
     &=
     \projfull D_{\text{tr}}^{n+1}
    +
    \colstaboph \projReffull \colvarn{n} 
     +
     \dtau (\projReffull -\projfull ) \colvarn{n+1} 
     \\
     &\quad +
     ( \coloph \projReffull -  \projfull \colop ) \tauavg{\colvarn{n+1}}
     +
     \colrhshnavg{n+1} - \tauavg{\colrhshn{n+1}},
\end{align*}
where $D_{\text{tr}}^{n+1}$ is given by
\begin{align*}
    D_{\text{tr}}^{n+1} = \tau^2 \int_{0}^{1}
    k_{\text{tr}}(\sigma) \, \pt^3 \colvar(t_n+ \tau \sigma)
    \,d \sigma
\end{align*}
for some bounded Peano kernel $k_{\text{tr}}$.
In the following Lemma, we provide estimates of this defect, noting that all the terms defined in \eqref{eq:def_space_errors} appear again.
For these estimates, we additionally define the stabilization approximation quantities
\begin{equation}\label{eq:def_stabilization_errors}
 \sigma_{J,h}(v)=\stabnormwave{\projwaveRef v-v},
 \qquad
 \sigma_{P,h}^{(1)}(w)=\stabnormwave{(\projwave^1w-w,0)}.
\end{equation}

\begin{lemma} \label{lem:CN_defect} 
Let \(\coltestvarh = \bigl(y_{h,1},y_{h,2},\varphi_h\bigr) \in \wavespaceVh\times \heatspaceVh\). 
The defect can be decomposed into 
    \begin{align*}
    \ipcol{ D_{\text{CN}}^{n+1} }{\coltestvarh}
    =
    \ipcol{ \widetilde{D}_{\text{CN}}^{n+1} }{\coltestvarh}
    +
\Delta_{s,b} ( \tauavg{\wavevar}^{n+1}, \tauavg{\heatvar}^{n+1}, y_h),
\end{align*}
and further into
\[
 L^{n+1}=-\tauhalf\colstaboph\projReffull\dtau\colvarn{n+1},
 \qquad
 \mathcal R^{n+1}=\widetilde D_{\mathrm{CN}}^{n+1}-L^{n+1}.
\]
The following bounds hold:
\begin{align*}
 \normcol{\mathcal R^{n+1}}
 &\leq \frac{\tau^2}{12}\max_{t\in[t_n,t_{n+1}]}\normcol{\pt^3\colvar(t)}\\
 &\quad+\|\Delta\ipwave(\dtau\wavevarn{n+1})\|_{\wavespaceh^*}
 +\|\Delta\ipheat(\dtau\heatvarn{n+1})\|_{\heatspaceh^*}\\
 &\quad+\|\widetilde\Delta_{s,b}(\tauavg{\wavevar}^{n+1},
                 \tauavg{\heatvar}^{n+1})\|_{\heatspaceh^*}\\
 &\quad+\|\Delta\tauavg{f^{n+1}}\|_{\wavespaceh^*}
 +\|\Delta\tauavg{g^{n+1}}\|_{\heatspaceh^*},\\
 \normcol{\widetilde D_{\mathrm{CN}}^{n+1}}
 &\leq\normcol{\mathcal R^{n+1}}
 +\frac{\sqrt{\eta\tau}}2\sigma_{J,h}(\dtau\wavevarn{n+1}).
\end{align*}
Moreover, for every discrete \(v_h=(y_h,\varphi_h)\) and
\(\varepsilon>0\), we have
\begin{equation}\label{eq:CN_stabilization_pairing}
 2\tau|\ipcol{L^{n+1}}{v_h}|
 \leq\varepsilon\tau\stabnormwave{y_h}^2
 +\frac{\tau^3}{4\varepsilon}\sigma_{J,h}(\dtau\wavevarn{n+1})^2.
\end{equation}
\end{lemma}

\begin{proof}
    The first part of the estimate follows directly from the boundedness of the Peano kernel. Further, 
    we note that for the discrete test function
$\coltestvarh$, we have
\begin{align*}
\ipcol{ \dtau (\projReffull -\projfull ) \colvarn{n+1} }{\coltestvarh}
    &=
\Delta \ipwave (\dtau \wavevarn{n+1} , y_h )
+
\Delta \ipheat (\dtau \heatvarn{n+1}, \varphi_h),
\end{align*}
and taking into account the manipulation
\begin{align*}
    \colstaboph \projReffull \colvarn{n} 
+
     ( \coloph \projReffull -  \projfull \colop ) \tauavg{\colvarn{n+1}}
&=
   - \frac{\tau}{2} \colstaboph  \projReffull \dtau \colvarn{n+1} 
+
    ( (  \coloph + \colstaboph  )\projReffull -  \projfull \colop ) \tauavg{\colvarn{n+1}} ,
\end{align*}
we recover by a straightforward calculation
\begin{align*}
    \ipcol{ ( (  \coloph + \colstaboph  )\projReffull -  \projfull \colop ) \tauavg{\colvarn{n+1}}}{\coltestvarh}
    &=
    \Delta_{s,b} ( \tauavg{\wavevar}^{n+1}, \tauavg{\heatvar}^{n+1}, y_h)
    +
    \widetilde{\Delta}_{s,b} (   \tauavg{\wavevar}^{n+1}, \tauavg{\heatvar}^{n+1},\varphi_h) .
\end{align*}
Similarly, the right-hand side terms are treated as
\begin{align*}
  \ipcol{   \tauavg{\colrhshn{n+1}} - \colrhshnavg{n+1} }{\coltestvarh}
=
    \Delta \tauavg{f^{n+1}} (y_h) 
    +
    \Delta \tauavg{g^{n+1}}  (\varphi_h) .
\end{align*}
It remains to consider the stabilization term. 
By Assumption~\ref{ass:discr_forms_skew_stab} and \eqref{eq:CFL_for_stab},
\[
 \normcol{\tauhalf\colstaboph\projReffull\dtau\colvarn{n+1}}
 \leq\frac{\sqrt{\eta\tau}}2\sigma_{J,h}(\dtau\wavevarn{n+1}).
\]
Together with \(\widetilde D_{\mathrm{CN}}^{n+1}=\mathcal R^{n+1}+L^{n+1}\)
and the trapezoidal Peano-kernel bound \(1/12\), this proves the norm estimates.
For the pairing bound, the Cauchy--Schwarz inequality for the stabilization
form and Young's inequality give
\begin{align*}
 2\tau|\ipcol{L^{n+1}}{v_h}|
 &=\tau^2|\sbilwavehpos(\projwaveRef\dtau\wavevarn{n+1},y_h)|\\
 &\leq\tau^2\sigma_{J,h}(\dtau\wavevarn{n+1})\stabnormwave{y_h}\\
 &\leq\varepsilon\tau\stabnormwave{y_h}^2
 +\frac{\tau^3}{4\varepsilon}\sigma_{J,h}(\dtau\wavevarn{n+1})^2,
\end{align*}
again employing Assumption~\ref{ass:discr_forms_skew_stab}.
\end{proof}

Turning now to the full defect, we define
\begin{align*}
  &  \dtau \projReffull \colvarn{n+1} 
        + 
        \coloph \projReffull \tauavg{\colvarn{n+1}}
        +
         \colstaboph \projReffull \colvarn{n} 
        \\
        &=   \tauavg{\colrhshn{n+1}}
            + 
            \tausqf \colpertoph \dtau \projReffull \colvarn{n+1} 
            - \tausqf \colpertrhsoph \dtau f_{h,1}^{n+1}
          +
           D_{\text{CN}}^{n+1} 
           +
            D_{\text{LFCN}}^{n+1}   ,
\end{align*}
and thus obtain the representation
\begin{align*}
    D_{\text{LFCN}}^{n+1}   
     =
    -
      \tausqf \colpertoph \dtau \projReffull \colvarn{n+1} 
      +
       \tausqf \colpertrhsoph    \dtau f_{h,1}^{n+1} ,
\end{align*}
which we study in more detail in the next lemma.

\begin{lemma} \label{lem:CNLF_defect}
Define the reconstructed first-component derivative by
\begin{equation}\label{eq:def_reconstructed_derivative}
 \wavevarhOneDotStar{n}
 =\waveophTwo\projwaveRef^2\wavevarnTwo{n}
 -\coupoph\projheatRef\heatvarn{n}+f_{h,1}^{n}.
\end{equation}
Then the LFCN defect factorizes as
\begin{equation*}
 D_{\text{LFCN}}^{n+1}
 =-\tausqf\coloph
 \begin{pmatrix}
  \dtau\wavevarhOneDotStar{n+1}\\0\\0
 \end{pmatrix}.
\end{equation*}
Moreover, for \(k=1,2\),
\begin{align*}
 \tau^2\normwaveOne{\ptau^k\wavevarhOneDotStar{n+1}}
 &\leq\tau^2\bigl(
 \|\Delta^s_{s,b}(\ptau^k\wavevarn{n+1},
                    \ptau^k\heatvarn{n+1})\|_{\wavespaceh^*}\\
 &\qquad
 +\|\Delta\ptau^k f^{n+1}\|_{\wavespaceh^*}
 +\|\ptau^k\pt\wavevarnOne{n+1}\|_{\wavespacehOne^*}
 \bigr),
\end{align*}
and
\begin{align*}
 \tau^2\stabnormwave{(\dtau\wavevarhOneDotStar{n+1},0)}
 &\leq\sqrt\eta\,\tau^{3/2}\bigl(
 \|\Delta^s_{s,b}(\dtau\wavevarn{n+1},\dtau\heatvarn{n+1})\|_{\wavespaceh^*}
 +\|\Delta\dtau f^{n+1}\|_{\wavespaceh^*}\bigr)\\
 &\quad+\tau^2\sigma_{P,h}^{(1)}(\dtau\pt\wavevarnOne{n+1}).
\end{align*}
The constants are independent of \(h\) and \(\tau\), and
\(\Delta^s_{s,b}\) is defined as \(\Delta_{s,b}\) in
\eqref{eq:def_space_errors_Delta_s_b} without the stabilization terms.
\end{lemma}

\begin{proof}
For clarity, we suppress the time argument of the exact variables in the
following calculation. The definitions of
\(\colpertoph\) and \(\colpertrhsoph\) give
\begin{align*}
 \colpertoph\projReffull\colvar
 &=-\colpertrhsoph
 \bigl(\waveophTwo\projwaveRef^2\wavevarTwo
       -\coupoph\projheatRef\heatvar\bigr),\\
 \colpertrhsoph z
 &=-\coloph\begin{pmatrix}z\\0\\0\end{pmatrix}.
\end{align*}
Together with \eqref{eq:def_reconstructed_derivative}, these identities prove
the factorization.

We next estimate the reconstruction error. For \(y_h=(y_{h,1},0)\),
\eqref{eq:first_eq_tested_in_X} gives
\[
 \ipwaveOne(\pt\wavevarOne,y_{h,1})
 =-\sbilwaveext(\wavevar,y_h)
 -\coupopfullformext(\heatvar,y_h)
 +\ipwaveOne(f_1,y_{h,1}).
\]
Thus, by the orthogonality of \(\projwave^1\) and the definition in \eqref{eq:def_reconstructed_derivative},
\begin{align*}
 \ipwaveOne(\wavevarhOneDotStar{n}-\projwave^1\pt\wavevarOne,y_{h,1})
 &=\ipwaveOne(\wavevarhOneDotStar{n}-\pt\wavevarOne,y_{h,1})\\
 &=\sbilwavehTwo(\projwaveRef^2\wavevarTwo,y_{h,1})
 -\coupopformh(\projheatRef\heatvar,y_{h,1})
 +\ipwaveOne(f_{h,1}-\pt\wavevarOne,y_{h,1})\\
 &=-\sbilwavehskew(\projwaveRef\wavevar,y_h)
 -\coupopfullformh(\projheatRef\heatvar,y_h)\\
 &\quad+\sbilwaveext(\wavevar,y_h)
 +\coupopfullformext(\heatvar,y_h)
 +\ipwaveOne(f_{h,1}-f_1,y_{h,1})\\
 &=-\Delta^s_{s,b}(\wavevar,\heatvar,(y_{h,1},0))
 +\Delta f((y_{h,1},0)).
\end{align*}
Applying the first two discrete time differences and using duality yields,
for \(k=1,2\),
\begin{equation*}
 \normwaveOne{\ptau^k\wavevarhOneDotStar{n+1}
 -\projwave^1\ptau^k\pt\wavevarnOne{n+1}}
 \leq\|\Delta^s_{s,b}(\ptau^k\wavevarn{n+1},
                        \ptau^k\heatvarn{n+1})\|_{\wavespaceh^*}
 +\|\Delta\ptau^k f^{n+1}\|_{\wavespaceh^*}.
\end{equation*}
The triangle inequality and the stability of \(\projwave^1\) prove the first
estimate. For the stabilization estimate, we split
\begin{equation*}
 \dtau\wavevarhOneDotStar{n+1}
 =\bigl(\dtau\wavevarhOneDotStar{n+1}
 -\projwave^1\dtau\pt\wavevarnOne{n+1}\bigr)
 +\projwave^1\dtau\pt\wavevarnOne{n+1}.
\end{equation*}
Using \eqref{eq:CFL_for_stab}, stabilization consistency in Assumption~\ref{ass:discr_forms_skew_stab}, and
\eqref{eq:def_stabilization_errors}, we obtain
\begin{align*}
 \tau^2\stabnormwave{(\dtau\wavevarhOneDotStar{n+1},0)}
 &\leq\sqrt\eta\,\tau^{3/2}
 \normwaveOne{\dtau\wavevarhOneDotStar{n+1}
 -\projwave^1\dtau\pt\wavevarnOne{n+1}}
 +\tau^2\sigma_{P,h}^{(1)}(\dtau\pt\wavevarnOne{n+1})\\
 &\leq\sqrt\eta\,\tau^{3/2}\bigl(
 \|\Delta^s_{s,b}(\dtau\wavevarn{n+1},\dtau\heatvarn{n+1})\|_{\wavespaceh^*}
 +\|\Delta\dtau f^{n+1}\|_{\wavespaceh^*}\bigr)\\
 &\quad+\tau^2\sigma_{P,h}^{(1)}(\dtau\pt\wavevarnOne{n+1}),
\end{align*}
and thus the second estimate.
\end{proof}

Finally, we can define the discrete error 
$\colerrhn{n} = \bigl(\waveerrhn{n},\heaterrhn{n}\bigr) = \projReffull \colvarn{n} - \colvarhn{n} $
and obtain by subtraction of the above representation from the numerical method
\begin{equation}\label{eq:full_error_equation}
 \bigl(I-\tauhalf\colstaboph-\tausqf\colpertoph\bigr)\dtau\colerrhn{n+1}
 +(\coloph+\colstaboph)\colerrhnavg{n+1}
 =D_{\mathrm{CN}}^{n+1}-\tausqf\coloph
 \begin{pmatrix}
  \dtau\wavevarhOneDotStar{n+1}\\0\\0
 \end{pmatrix},
\end{equation}
where we used the same identity as in the proof of Lemma~\ref{lem:stab_1}. These preparations now allow us to deduce our second main result. We note that the limit $\tau \to 0$ formally recovers the statement of Theorem~\ref{thm:space}.

\begin{theorem} \label{thm:full}
Let $\wavevar \in C^1([0,T],\wavespaceV)$
and
$\heatvar \in C^1([0,T],\heatspaceV)$ be the solution of
 \eqref{eq:full_system},
 which  satisfies
 $(\wavevar, \heatvar) \in C^1([0,T],\fullspace)$ and $\colvar \in C^4([0,T],\wavespace \times\heatspace)$, and
$(\wavevarhn{n},\heatvarhn{n})$ the solution of \eqref{eq:LFCN}.
Further, let 
Assumptions~\ref{ass:cont_forms}, \ref{ass:discr_forms_skew_stab},
\ref{ass:consistency},
and
\ref{ass:CFL}
hold true.
Then, for \(N\tau\leq T\), the fully discrete errors satisfy
\begin{equation*}
    \begin{aligned}
    &\max_{0\leq n\leq N}
    \bigl(
    \| \waveerrhn{n} \|_{\wavespaceh}^2
    +
    \| \heaterrhn{n} \|_{\heatspaceh}^2
    \bigr)
    +
    \tau
    \sum_{\ell = 0}^{N-1}
    \stabnormwave{ \waveerrhn{\ell} + \waveerrhn{\ell+1} }^2
    \\
    &\quad
    \lesssim
    \| \waveerrhn{0} \|_{\wavespaceh}^2
    +
    \| \heaterrhn{0} \|_{\heatspaceh}^2
    +
    \tau^4
    \max_{t\in[0,T]}
    \| \pt^3 \colvar(t) \|_{\wavespace \times \heatspace}^2
    +\mathcal D_N^{\mathrm{space}}
    +\mathcal D_N^{\mathrm{pair}}
    +\tau^4\mathcal D_N^{\mathrm{bounded}}.
    \end{aligned}
\end{equation*}
The constants depend on \(T\) and \(\eta\), but not on \(h\) and \(\tau\).
Here,
\begin{align}
&\begin{aligned}
\mathmakebox[\widthof{$\mathcal D_N^{\mathrm{bounded}}$}][r]{\mathcal D_N^{\mathrm{space}}}
    &=\tau\sum_{\ell=1}^N\Bigl(
    \|\Delta \ipwave (\dtau \wavevarn{\ell})\|^2_{\wavespaceh^*}
    +
    \|\Delta \ipheat (\dtau \heatvarn{\ell})\|^2_{\heatspaceh^*}
    +
    \|\widetilde{\Delta}_{s,b} ( \tauavg{\wavevar}^{\ell} , \tauavg{\heatvar}^{\ell}) \|^2_{\heatspaceh^*}
    \\
    &\qquad\qquad+
    \|\Delta \tauavg{f^{\ell}} \|_{\wavespaceh^*}^2
    +
    \|\Delta \tauavg{g^{\ell}} \|_{\heatspaceh^*}^2
    \Bigr)
    \\
    &\quad+\tau\sum_{\ell=1}^N\tau^3\Bigl(
    \|\Delta^s_{s,b}(\ptau\wavevarn{\ell},\ptau\heatvarn{\ell})\|^2_{\wavespaceh^*}
    +\|\Delta\ptau f^{\ell}\|^2_{\wavespaceh^*}\Bigr)
    \\
    &\quad+\tau\sum_{\ell=1}^N\tau^2
    \sigma_{J,h}(\ptau\wavevarn{\ell})^2,
\end{aligned}\label{eq:full_defect_space}\\
&\begin{aligned}
\mathmakebox[\widthof{$\mathcal D_N^{\mathrm{bounded}}$}][r]{\mathcal D_N^{\mathrm{pair}}}
    &=\max_{1\leq m\leq N}
    \bigl|
    \tau
    \sum_{\ell = 1}^m
    \Delta_{s,b} (
        \tauavg{\wavevar}^{\ell}, \tauavg{\heatvar}^{\ell},
        \tauavg{\waveerrhn{\ell}} )
    \bigr|
    \\
    &\quad+\tau\sum_{\ell=1}^N\tau^2
    \bigl|
    \Delta_{s,b} (
        \tauavg{\wavevar}^{\ell}, \tauavg{\heatvar}^{\ell},
        (\ptau\wavevarhOneDotStar{\ell},0) )
    \bigr|,
\end{aligned}\label{eq:full_defect_pair}\\
&\begin{aligned}
\mathcal D_N^{\mathrm{bounded}}
    &=\max_{1\leq \ell\leq N}
    \bigl(
    \| \Delta^s_{s,b} ( \ptau \wavevarn{\ell}, \ptau \heatvarn{\ell}) \|^2_{\wavespaceh^*}
    +
    \| \Delta \ptau f^{\ell} \|_{\wavespaceh^*}^2
    +
    \| \ptau \pt \wavevarnOne{\ell} \|_{\wavespacehOne^*}^2
    \bigr)
    \\
    &\quad+\tau\sum_{\ell=2}^N
    \bigl(
    \| \Delta^s_{s,b} ( \ptau^{2} \wavevarn{\ell}, \ptau^{2} \heatvarn{\ell}) \|^2_{\wavespaceh^*}
    +
    \| \Delta \ptau^{2} f^{\ell} \|_{\wavespaceh^*}^2
    +
    \| \ptau^2 \pt \wavevarnOne{\ell} \|_{\wavespacehOne^*}^2
    \bigr)
    \\
    &\quad+\tau\sum_{\ell=1}^N
    \sigma_{P,h}^{(1)}(\ptau\pt\wavevarnOne{\ell})^2.
\end{aligned}\label{eq:full_defect_bounded}
\end{align}
The spatial defects are defined in \eqref{eq:def_space_error},
the errors \(\sigma_{J,h}\), \(\sigma_{P,h}^{(1)}\) in \eqref{eq:def_stabilization_errors}, and 
the reconstructed derivative \(\wavevarhOneDotStar{\ell}\) in
\eqref{eq:def_reconstructed_derivative}.
\end{theorem}

\begin{proof}
As in the proof of Theorem~\ref{thm:StabilityGenericRecursion}, we apply
Lemma~\ref{lem:stab_1} to the error equation \eqref{eq:full_error_equation},
sum for \(j=1,\ldots,m\), \(m\leq N\), and use the CFL condition
\eqref{eq:CFL} to obtain
\begin{equation}\label{eq:EnergyInequalityForError}
 (1-\eta)\normcol{\colerrhn{m}}^2
 +2\tau\sum_{j=1}^m\stabnormwave{\tauavg{\waveerrhn{j}}}^2
 \leq\normcol{\colerrhn{0}}^2
 +2\tau\sum_{j=1}^m\ipcol{K^j}{\colerrhnavg{j}},
\end{equation}
with
\(K^j=D_{\mathrm{CN}}^j
-\tausqf\coloph(\dtau\wavevarhOneDotStar{j},0,0)\).
For brevity, set
\(\mathcal M_h=I-\tauhalf\colstaboph-\tausqf\colpertoph\).
Since \((\dtau\wavevarhOneDotStar{j},0,0)\) is zero except in its first
component, the
skew-adjointness \eqref{eq:cond_skew_adj} gives
\(\ipcol{(\dtau\wavevarhOneDotStar{j},0,0)}
{\coloph(\dtau\wavevarhOneDotStar{j},0,0)}=0\), and inserting
the error equation \eqref{eq:full_error_equation} yields
\begin{multline*}
 -\ipcol{\coloph(\dtau\wavevarhOneDotStar{j},0,0)}{\colerrhnavg{j}}\\
 =\ipcol{(\dtau\wavevarhOneDotStar{j},0,0)}{D_{\mathrm{CN}}^j}
 -\ipcol{(\dtau\wavevarhOneDotStar{j},0,0)}
 {\mathcal M_h\dtau\colerrhn{j}}%
 -\sbilwavehpos\bigl(
 (\dtau\wavevarhOneDotStar{j},0),\tauavg{\waveerrhn{j}}\bigr).
\end{multline*}
The first row and column of \(\colpertoph\) in
\eqref{eq:DiscreteOperators:Perturbations} vanish, so that
\(\colpertoph(\dtau\wavevarhOneDotStar{j},0,0)=0\). By the symmetry of
\(\colstaboph\) and
\(\colpertoph\), the perturbation drops out, i.e.,
\begin{equation*}
 \ipcol{(\dtau\wavevarhOneDotStar{j},0,0)}
 {\mathcal M_h\dtau\colerrhn{j}}
 =\ipcol{\bigl(I-\tauhalf\colstaboph\bigr)
 (\dtau\wavevarhOneDotStar{j},0,0)}
 {\dtau\colerrhn{j}}.
\end{equation*}
By Lemma~\ref{lem:CN_defect}, the CN defect splits for every discrete
\(\coltestvarh\) with wave component \(y_h\) into
\begin{equation*}
 \ipcol{D_{\mathrm{CN}}^j}{\coltestvarh}
 =\ipcol{\mathcal R^j}{\coltestvarh}
 +\ipcol{L^j}{\coltestvarh}
 +\Delta_{s,b}(\tauavg{\wavevar}^j,\tauavg{\heatvar}^j,y_h).
\end{equation*}
We use this for \(\coltestvarh=\colerrhnavg{j}\) and for
\(\coltestvarh=(\dtau\wavevarhOneDotStar{j},0,0)\). Altogether, the last sum in
the energy inequality \eqref{eq:EnergyInequalityForError} equals \(I_1+\dots+I_5\) with
\begin{align*}
 I_1&=2\tau\sum_{j=1}^m\bigl(\ipcol{\mathcal R^j}{\colerrhnavg{j}}
      +\tausqf\ipcol{\mathcal R^j}
      {(\dtau\wavevarhOneDotStar{j},0,0)}\bigr),\\
 I_2&=2\tau\sum_{j=1}^m\ipcol{L^j}
      {\colerrhnavg{j}+\tausqf(\dtau\wavevarhOneDotStar{j},0,0)},\\
 I_3&=2\tau\sum_{j=1}^m\bigl(
      \Delta_{s,b}(\tauavg{\wavevar}^j,\tauavg{\heatvar}^j,\tauavg{\waveerrhn{j}})
      +\tausqf\Delta_{s,b}(\tauavg{\wavevar}^j,\tauavg{\heatvar}^j,
      (\dtau\wavevarhOneDotStar{j},0))\bigr),\\
 I_4&=-\frac{\tau^3}{2}\sum_{j=1}^m
      \ipcol{\bigl(I-\tauhalf\colstaboph\bigr)
      (\dtau\wavevarhOneDotStar{j},0,0)}
      {\dtau\colerrhn{j}},\\
 I_5&=-\frac{\tau^3}{2}\sum_{j=1}^m
      \ipcol{\colstaboph^{1/2}(\dtau\wavevarhOneDotStar{j},0,0)}
      {\colstaboph^{1/2}\colerrhnavg{j}}.
\end{align*}
We estimate them individually.

For \(I_1\), we apply the Cauchy-Schwarz inequality on both terms under sum.
Additionally, we apply the Young inequality with \(\gamma = 4\tau^{-2}\) in the second term and obtain the estimate
\begin{equation*}
 |I_1|
 \leq2\tau\sum_{j=1}^m\normcol{\mathcal R^j}\normcol{\colerrhnavg{j}}
 +\tau\sum_{j=1}^m\bigl(\normcol{\mathcal R^j}^2
 +\frac{\tau^4}{16}\normwaveOne{\dtau\wavevarhOneDotStar{j}}^2\bigr),
\end{equation*}

For \(I_2\), we apply \eqref{eq:CN_stabilization_pairing} separately with
\(v_h=\colerrhnavg{j}\), \(\varepsilon=1/4\), and with
\(v_h=\tausqf(\dtau\wavevarhOneDotStar{j},0,0)\), \(\varepsilon=2\).
This gives
\begin{equation*}
    |I_2|
 \leq\frac\tau4\sum_{j=1}^m\stabnormwave{\tauavg{\waveerrhn{j}}}^2
 +\frac{9\tau}8\sum_{j=1}^m\tau^2\sigma_{J,h}(\dtau\wavevarn{j})^2
 +\frac\tau8\sum_{j=1}^m\tau^4
 \stabnormwave{(\dtau\wavevarhOneDotStar{j},0)}^2.
\end{equation*}

The terms of $I_3$ are already in the form of the theorem.

For \(I_4\), we apply summation by parts, which is possible since
\(I-\tauhalf\colstaboph\) does not depend on \(j\), and obtain
\begin{align*}
 I_4
 &=-\frac{\tau^2}{2}\ipcol{\bigl(I-\tauhalf\colstaboph\bigr)
 (\dtau\wavevarhOneDotStar{m},0,0)}{\colerrhn{m}}\\
 &\quad+\frac{\tau^2}{2}\ipcol{\bigl(I-\tauhalf\colstaboph\bigr)
 (\dtau\wavevarhOneDotStar{1},0,0)}{\colerrhn{0}}\\
 &\quad+\frac{\tau^3}{2}\sum_{j=2}^m
 \ipcol{\bigl(I-\tauhalf\colstaboph\bigr)
 (\dtau^2\wavevarhOneDotStar{j},0,0)}{\colerrhn{j-1}}.
\end{align*} 
We estimate \(\|I-\tauhalf\colstaboph\|\leq1\), which follows from \(0\leq\tauhalf\colstaboph\leq\frac\eta2 I\) by \eqref{eq:CFL_for_stab}.
We estimate every individual summand by Cauchy-Schwarz and additionally apply Youngs inequality for the boundary terms.
This gives
\begin{align*}
|I_4|
    &\leq 
        \frac{\tau^4}{16}\normwaveOne{\dtau\wavevarhOneDotStar{1}}^2
            +   \normcol{\colerrhn{0}}^2\\
            &\quad+ \frac{2}{1-\eta} \frac{\tau^4}{16}
            \normwaveOne{\dtau\wavevarhOneDotStar{m}}^2
            + \frac{1-\eta}{2} \normcol{\colerrhn{m}}^2\\  
        &\quad+2 \frac{\tau^3}{4}
        \sum_{j=2}^{m} \normwaveOne{\dtau^2\wavevarhOneDotStar{j}}
            \normcol{\colerrhn{j-1}}.
\end{align*}
Note that we choose the constant for the \(j= m\) boundary term such that it can be absorbed on the left-hand side.

For \(I_5\), we once again apply Cauchy-Schwarz and Young to obtain
\begin{equation*}
 |I_5|
 \leq\frac\tau2\sum_{j=1}^m\stabnormwave{\tauavg{\waveerrhn{j}}}^2
 +\frac\tau8\sum_{j=1}^m\tau^4
 \stabnormwave{(\dtau\wavevarhOneDotStar{j},0)}^2,
\end{equation*}

Let \(E_*=\max_{0\leq j\leq N}\normcol{\colerrhn{j}}\).
Absorbing the stabilization terms and the boundary term at \(j=m\), we
arrive at
\begin{equation*}
 \frac{1-\eta}2\normcol{\colerrhn{m}}^2
 +\tau\sum_{j=1}^m\stabnormwave{\tauavg{\waveerrhn{j}}}^2
 \leq2\normcol{\colerrhn{0}}^2+2AE_*+B,
\end{equation*}
where
\begin{align*}
 A&=\tau\sum_{j=1}^N\normcol{\mathcal R^j}
   +\frac{\tau^3}{4}\sum_{j=2}^N
   \normwaveOne{\dtau^2\wavevarhOneDotStar{j}},\\
 B&=\tau\sum_{j=1}^N\Bigl(\normcol{\mathcal R^j}^2
   +\frac98\tau^2\sigma_{J,h}(\dtau\wavevarn{j})^2
   +\frac14\tau^4
   \stabnormwave{(\dtau\wavevarhOneDotStar{j},0)}^2\Bigr)\\
  &\quad+\Bigl(T+1+\frac2{1-\eta}\Bigr)\frac{\tau^4}{16}
   \max_{1\leq j\leq N}\normwaveOne{\dtau\wavevarhOneDotStar{j}}^2
   +\frac{\tau^3}{2}\sum_{j=1}^N\bigl|\Delta_{s,b}
   (\tauavg{\wavevar}^j,\tauavg{\heatvar}^j,
   (\dtau\wavevarhOneDotStar{j},0))\bigr|\\
  &\quad+2\max_{1\leq m'\leq N}\Bigl|\tau\sum_{j=1}^{m'}
   \Delta_{s,b}(\tauavg{\wavevar}^j,\tauavg{\heatvar}^j,
   \tauavg{\waveerrhn{j}})\Bigr|.
\end{align*}

Let \(m_*\in\{0,\ldots,N\}\) satisfy
\(\normcol{\colerrhn{m_*}}=E_*\). If \(m_*=0\), then
\(E_*^2=\normcol{\colerrhn{0}}^2\). Otherwise, setting \(m=m_*\) in the
preceding inequality and dropping the nonnegative stabilization term gives
\begin{align*}
 \frac{1-\eta}{2}E_*^2
 &\leq2\normcol{\colerrhn{0}}^2+2AE_*+B\\
 &\leq2\normcol{\colerrhn{0}}^2
 +\frac{1-\eta}{4}E_*^2+\frac4{1-\eta}A^2+B.
\end{align*}
Hence, in both cases,
\begin{equation*}
 E_*^2\lesssim_\eta
 \normcol{\colerrhn{0}}^2+A^2+B.
\end{equation*}
We now set \(m=N\) in the preceding energy inequality and use this bound in
\(2AE_*\leq A^2+E_*^2\). This yields
\begin{equation*}
 E_*^2+\tau\sum_{j=1}^N
 \stabnormwave{\tauavg{\waveerrhn{j}}}^2
 \lesssim_\eta\normcol{\colerrhn{0}}^2+A^2+B.
\end{equation*}
Finally, Cauchy--Schwarz and \(N\tau\leq T\) give
\begin{align*}
 A^2
 &\leq2\left(\tau\sum_{j=1}^N\normcol{\mathcal R^j}\right)^2
 +\frac{\tau^6}{8}
 \left(\sum_{j=2}^N\normwaveOne{\dtau^2\wavevarhOneDotStar{j}}\right)^2\\
 &\leq2T\tau\sum_{j=1}^N\normcol{\mathcal R^j}^2
 +\frac{T\tau^5}{8}\sum_{j=2}^N
 \normwaveOne{\dtau^2\wavevarhOneDotStar{j}}^2.
\end{align*}
Consequently,
\begin{equation*}
 E_*^2+\tau\sum_{j=1}^N\stabnormwave{\tauavg{\waveerrhn{j}}}^2
 \lesssim\normcol{\colerrhn{0}}^2
 +\tau\sum_{j=2}^N\tau^4
 \normwaveOne{\dtau^2\wavevarhOneDotStar{j}}^2+B.
\end{equation*}

Finally, we insert the bounds of Lemma~\ref{lem:CN_defect} for
\(\mathcal R^j\) and of Lemma~\ref{lem:CNLF_defect} for
\(\dtau\wavevarhOneDotStar{j}\),
\(\dtau^2\wavevarhOneDotStar{j}\), and
\(\stabnormwave{(\dtau\wavevarhOneDotStar{j},0)}\).
This proves the claim.
\end{proof}

\section{Application to examples and numerical experiments}
\label{sec:application}
We now apply the error estimates from Sections~\ref{sec:space} and
\ref{sec:full} to the examples introduced above.
Throughout, discrete derivatives are bounded by the maximum norm of their
continuous counterparts without an additional constant.
See, for example, \cite[Lemma~4.2]{DoeN24*}.

To apply Theorem~\ref{thm:full} to the examples below, we 
decompose \(\mathcal D_N^{\mathrm{pair}}=\mathcal P_N+\mathcal Q_N\) with
\begin{align*}
 \mathcal P_N&=\max_{1\leq m\leq N}\left|
 \tau\sum_{j=1}^m
 \Delta_{s,b}(\tauavg{\wavevar}^j,\tauavg{\heatvar}^j,
 \tauavg{\waveerrhn{j}})\right|,\\
 \mathcal Q_N&=\tau\sum_{j=1}^N\tau^2
 |\Delta_{s,b}(\tauavg{\wavevar}^j,\tauavg{\heatvar}^j,
 (\dtau\wavevarhOneDotStar{j},0))|.
\end{align*}
We denote the left-hand side of the estimate in Theorem~\ref{thm:full}
by \(\mathcal E_N\). For a target spatial rate \(r_h\), the theorem reduces
the fully discrete estimate to two bounds:
\begin{equation}\label{eq:app_verification}
 \mathcal P_N\leq\varepsilon\mathcal E_N
 +C_\varepsilon(r_h^2+\tau^4),\qquad
 \mathcal D_N^{\mathrm{space}}+\mathcal Q_N
 +\tau^4\mathcal D_N^{\mathrm{bounded}}
 \lesssim r_h^2+\tau^4.
\end{equation}
Together with an initial error of order \(r_h\), these imply
\(\mathcal E_N^{1/2}\lesssim r_h+\tau^2\) after choosing
\(\varepsilon\) sufficiently small. The main choice is how to estimate
\(\mathcal P_N\): a consistency bound in the energy and stabilization norms
can be used directly, whereas a coupling term that differentiates an
uncontrolled error component calls for summation by parts in time.
For \(\mathcal Q_N\), we separate the projected exact derivative from the
reconstruction error using \eqref{eq:app_reconstructed_derivative_split}.
The factor \(\tau^4\) in the second bound means that the defects in
\(\mathcal D_N^{\mathrm{bounded}}\) only need uniform estimates.
Throughout this section, the meshes are shape-regular and quasi-uniform, the
polynomial degrees and material parameters are fixed, and the CFL condition in Assumption~\ref{ass:CFL} is satisfied.
We also use the decomposition of the reconstructed derivative
\begin{equation}\label{eq:app_reconstructed_derivative_split}
 \dtau\wavevarhOneDotStar{j}
 =\projwave^1\dtau\pt\wavevarnOne{j}
 +\bigl(\dtau\wavevarhOneDotStar{j}
 -\projwave^1\dtau\pt\wavevarnOne{j}\bigr).
\end{equation}
All bounds are uniform for \(N\tau\leq T\).

\subsection{Thermo-elastic string}
\label{sec:app_waveheat}

    For the example in Section~\ref{sec:ana_thermoelastic_string}, we employ continuous finite elements of order $\ell \geq 1$ to the heat and $k\geq 1$ to the wave part, and thus no stabilization is used. 
In addition, we apply a conforming method and thus use the continuous bilinear forms also in the discrete case, 
and allow for $k \leq \ell +1$.
In particular,
Assumptions~\ref{ass:discr_forms_skew_stab} and \ref{ass:consistency} are satisfied. 
Concerning the projections, we employ the Ritz projection $\projRitz$ and the $L^2$-projection $\projLtwo$ via
$\projwave = (\projLtwo,\projRitz)$ and $\projheat = \projLtwo$.
For the reference operators, we choose 
$\projwaveRef = (\projRitz , \projRitz)$
and
$\projheatRef  =  \projRitz$.
All the upcoming approximation results can be found in the book
\cite{BreS08}.
For the initial values and the right-hand side, we employ the interpolation $\interpolation$.
To shorten notation for linear elements, we set $k^* = \max \{k,2\}$,
and can thus conclude for the initial data and the sources 
\begin{align*}
     \| \waveerrh(0) \|_{\wavespaceh}
          +
     \| \heaterrh(0) \|_{\wavespaceh} 
     &\lesssim
     h^k
     \bigl( 
        |\pt u |_{H^{k}}
        +
        |\pt^2 u |_{H^{k^*}}
        \bigr) 
        +
        h^{\ell +1} |\pt \heatvar |_{H^{\ell+1}},
        \\
    \|\Delta f \|_{\wavespaceh^*}
+
    \|\Delta g\|_{\heatspaceh^*}
    &\lesssim
        h^{k+1}  |f |_{H^{k+1}}     
        +
        h^{\ell+1} 
        |g|_{H^{\ell+1}}
        .
    \end{align*}
By the choice of the reference operators, we also obtain
\begin{align*}
      \|\Delta \ipwave (\pt \wavevar)\|_{\wavespaceh^*}
      +
        \|\Delta \ipheat (\pt \heatvar)\|_{\heatspaceh^*}
       \lesssim
       h^{k+1}
        |\pt^2 u |_{H^{k+1}}
        +
        h^{\ell+1} 
        |\pt \heatvar |_{H^{\ell+1}}
        .
\end{align*}
For this problem, we use the fact that $\wavespaceh \subset \wavespaceV$ holds, and thus obtain
\begin{align*}
   | \int_0^t  \Delta_{s,b} ( \wavevar, \heatvar, \waveerrh)  \,ds |
     &=
     | \int_0^t   \sbilwaveh( \projwaveRef \wavevar, \waveerrh)
    +
    \coupopfullformh( \projheatRef \heatvar, \waveerrh)
    -
     \sbilwave(  \wavevar, \waveerrh)
     -
    \coupopfullform(  \heatvar, \waveerrh) 
    \,ds |
    \\
    &=
    | \int_0^t  \coupopfullformext( \projheatRef \heatvar - \heatvar, \waveerrh) \,ds |
    =
    | \int_0^t 
      \langle
   \projheatRef \heatvar - \heatvar , \partial_x \waveerrhOne
    \rangle_{L^2} \,ds |,
    \end{align*}    
    where we used the definition of $ \projwaveRef$ in the second last step, and  integration by parts in space in the last one.
    The error equation implies
    $
     \pt \waveerrhTwo = \waveophOne \waveerrhOne
    $,
    where in this example $ \waveophOne = I$.   
    Furthermore,
    integration by parts in time leads to
\begin{align*}
   | \int_0^t  \Delta_{s,b} ( \wavevar, \heatvar&, \waveerrh)  \,ds |\\
     &\lesssim
     h^{\ell+1}  |\pt \heatvar |_{C^1([0,t];H^{\ell+1})}
     \bigl(
    \int_0^t 
       \| \waveerrh \|_{\wavespaceh}
      \,ds 
         +
       \| \waveerrh (t) \|_{\wavespaceh} + \| \waveerrh (0) \|_{\wavespaceh}
     \bigr),
    \end{align*}
For the fully discrete estimate, we seek the rate
\(r_h=h^k+h^{\ell+1}\). The obstruction is \(\mathcal P_N\): the
coupling defect contains a spatial derivative of the first wave-error
component, which the energy does not control. As in the semidiscrete
argument, we move the time sum onto the heat projection error. Set
\(z^j=(\projheatRef-I)\tauavg{\heatvar}^j\) and
\(W^m=\tau\sum_{j=1}^m\tauavg{\waveerrhnOne{j}}\), with \(W^0=0\).
Then summation by parts gives, for \(1\leq m\leq N\),
\begin{align*}
 &\tau\sum_{j=1}^m\Delta_{s,b}
 (\tauavg{\wavevar}^j,\tauavg{\heatvar}^j,\tauavg{\waveerrhn{j}})\\
 &\quad=\sum_{j=1}^m\langle z^j,\partial_x(W^j-W^{j-1})\rangle_{L^2}\\
 &\quad=\langle z^m,\partial_xW^m\rangle_{L^2}
 -\tau\sum_{j=1}^{m-1}\langle\dtau z^{j+1},\partial_xW^j\rangle_{L^2}.
\end{align*}
It remains to bound the sum involving $\partial_x \tauavg{\waveerrhnOne{k+1}}$.
We use the error equation in \eqref{eq:full_error_equation}, project to the second component (denoted by $\pi_2$) and sum from 
$\ell = 0,\ldots , n$ 
to obtain
\begin{equation} \label{eq:error_eqfor_Sh1eh1}
\begin{aligned} 
   \tau \sum_{\ell = 0}^n \bigl(\waveophOne & \tauavg{\waveerrhnOne{\ell+1}}
   +
   \waveophposTwo \tauavg{\waveerrhnTwo{\ell+1}} \bigr)
    =
    (1 - \tauhalf \waveophposTwo  -  \tausqf \waveophOne \waveophTwo ) (  \waveerrhnTwo{0} - \waveerrhnTwo{n+1} ) 
    \\
    &+
     \tausqf \waveophOne \coupoph ( \heaterrhn{0} - \heaterrhn{n+1} ) 
    +
    \tau \sum_{\ell = 0}^n \pi_2 D_{\text{CN}}^{\ell+1}  
    -
    \tausqf \waveophOne
    ( \wavevarhOneDotStar{n+1}
    - \wavevarhOneDotStar{0}
    ), 
    \end{aligned}
\end{equation}
using the notation from Lemma~\ref{lem:CNLF_defect}. For the specific example here, we have no stabilization, and in particular $ \waveophOne = I$ 
and 
$\| \partial_x \pi_2 D_{\text{CN}}^{n+1}\|_{L^2} = \mathcal{O}(h^k + \tau^2)$. Using the CFL conditions \eqref{eq:CFL}, we obtain
\begin{align*} 
  \| \partial_x \tau \sum_{\ell = 0}^n \waveophOne  \tauavg{\waveerrhnOne{\ell+1}} \|_{L^2}
&\lesssim  %
  \|  \waveerrhnTwo{0}    \|_{H_0^1} + \|  \waveerrhnTwo{n+1}    \|_{H_0^1}
    +
    \|  \heaterrhn{0}    \|_{L^2} + \|  \heaterrhn{n+1}    \|_{L^2} + h^{k} + \tau^2   
    \\
    &+
    \tausqf  (  \| \wavevarhOneDotStar{n+1}
    -  \projwave^1 \pt \wavevarnOne{n+1}   \|_{H_0^1}
    + 
    \| \wavevarhOneDotStar{0}
    -  \projwave^1 \pt \wavevarnOne{0}  \|_{H_0^1}  )
    \\
    &+
    \tausqf   \|  \projwave^1 \pt \wavevarnOne{n+1}   \|_{H_0^1} 
    +
    \tausqf   \|  \projwave^1 \pt \wavevarnOne{0}   \|_{H_0^1}   
    \\
    \lesssim 
      \|  \waveerrhnTwo{0}  &  \|_{H_0^1} 
      +
       \|  \waveerrhnTwo{n+1}    \|_{H_0^1}
    +
    \|  \heaterrhn{0}    \|_{L^2} + \|  \heaterrhn{n+1}    \|_{L^2}
    + h^{k} + h^{\ell+1}    + \tau^2   .
\end{align*}
The CFL condition ensures that
\(\tau^2\waveophOne\coupoph\) is uniformly bounded from the heat
\(L^2\)-space into the second wave \(H_0^1\)-space. Consequently, the
temperature contribution in this estimate only involves the \(L^2\)-error.
The summation-by-parts formula above and the initial-error bound therefore give,
for every \(\varepsilon>0\),
\[
 \mathcal P_N\lesssim h^{\ell+1}
 \bigl(\mathcal E_N^{1/2}+h^k+h^{\ell+1}+\tau^2\bigr)
 \leq\varepsilon\mathcal E_N
 +C_\varepsilon\bigl(h^{2k}+h^{2\ell+2}+\tau^4\bigr).
\]
We next estimate the remaining terms in Theorem~\ref{thm:full}.
Using the definition of $\projheatRef$ we compute 
\begin{align*}
   \|
   \widetilde{\Delta}_{s,b} ( \wavevar, \heatvar)
    \|_{\heatspaceh^*}
    &\leq
       \|
   \sbilheat( \projheatRef \heatvar - \heatvar , \cdot ) 
    \|_{\heatspaceh^*}
    +
        \|
    \coupopfullform( \cdot , \projwaveRef \wavevar - \wavevar)
    \|_{\heatspaceh^*}
    \lesssim
    h^{k}
    | \pt u |_{H^{k+1}}.
    \end{align*}
Since the
method is conforming, both stabilization seminorms vanish. Moreover, Ritz orthogonality and
the source approximation give,
for \(a=1,2\) at the available time indices,
\[
 \|\Delta^s_{s,b}(\dtau^a\wavevarn{j},\dtau^a\heatvarn{j})\|_{\wavespaceh^*}
 +\|\Delta\dtau^a f^j\|_{\wavespaceh^*}
 \lesssim h^\ell+h^{k+1}.
\]
In particular,
\[
 \normwaveOne{\dtau\wavevarhOneDotStar{j}
 -\projwave^1\dtau\pt\wavevarnOne{j}}
 \lesssim h^\ell+h^{k+1}.
\]
The \(H^1\)-stability of the \(L^2\)-projection and the inverse inequality yield
\begin{align*}
 \|\partial_x\dtau\wavevarhOneDotStar{j}\|_{L^2}
 &\leq\|\partial_x\projwave^1\dtau\pt\wavevarnOne{j}\|_{L^2}\\
 &\quad+\|\partial_x(\dtau\wavevarhOneDotStar{j}
 -\projwave^1\dtau\pt\wavevarnOne{j})\|_{L^2}\\
 &\lesssim\|\dtau\pt\wavevarnOne{j}\|_{H^1}
 +h^{-1}(h^\ell+h^{k+1})\lesssim1,\\
 \mathcal Q_N
 &\leq\tau\sum_{j=1}^N\tau^2
 \|(\projheatRef-I)\tauavg{\heatvar}^j\|_{L^2}
 \|\partial_x\dtau\wavevarhOneDotStar{j}\|_{L^2}
 \lesssim\tau^2h^{\ell+1}
 \lesssim\tau^4+h^{2\ell+2}.
\end{align*}
The mass, heat, and source defects have the spatial rates established above.
Applying the same estimates to time differences gives
\[
 \mathcal D_N^{\mathrm{space}}
 \lesssim h^{2k}+h^{2\ell+2}
 +\tau^3(h^{2\ell}+h^{2k+2})
 \lesssim r_h^2+\tau^4.
\]
The first and second time differences and exact derivatives are uniformly bounded,
so \(\mathcal D_N^{\mathrm{bounded}}\lesssim1\). Together with the
estimates for \(\mathcal P_N\) and \(\mathcal Q_N\), this verifies
\eqref{eq:app_verification}.

Combining these estimates, 
we obtain for a sufficiently smooth solution the spatial discretization error bound 
of optimal order 
    \begin{equation*}
        \| u(t) - u_h(t)\|_{H_0^1} 
        +
        \| \pt u(t) - \pt u_h(t)\|_{L^2} 
      +
    \| \heatvar(t) - \heatvarh(t)\|_{L^2} 
     \lesssim
      h^{k} + h^{\ell+1},
    \end{equation*}
with constants independent of $h$,
and the full discretization error bound
      \begin{equation*}
    \| u(t_n) -  u^n\|_{H_0^1} 
        +
        \| \pt u(t_n) -  v_h^n\|_{L^2} 
        +    
    \| \heatvar(t_n) -  \heatvarh^n\|_{L^2} 
\lesssim
 \tau^2 +       h^{k} + h^{\ell+1}
    \end{equation*}
with constants independent of $h$ and $\tau$,
where the last estimate holds under the CFL condition \eqref{eq:CFL}.
Note that these estimates suggest the optimal choice $\ell = k-1$ for $k\geq 2$.

\subsection{Non-local in time visco-elasticity and acoustics.}
\label{app:ana_ell_wave}

For the example in Section~\ref{sec:ana_ell_wave}, 
we only discuss the visco-elastic case, since the acoustic equation can be treated by minor modifications of the considerations below.
We apply discontinuous elements for the velocity part and continuous elements for the pressure part. 
Further, the coupling term is discretized in a conforming way, by leaving the derivative on the continuous elements. 
For the bilinear forms $\ipwaveh$, $\ipheath$, and $\sbilheath$
we simply use the restriction of the continuous forms to the (broken) polynomial spaces.
In particular, we choose broken polynomial spaces of order $k$ and $\div$-conforming elements of order $\ell \leq k+1$, for example Raviart--Thomas elements, such that the image under the divergence is contained in the broken polynomial space of degree $\ell-1$.
For the bilinear form $\sbilwaveh$, we follow the discretization of Section~4 in \cite{HocPSTW15},
using the discrete operator $A_h$ there to define
$\sbilwaveh(x_h,y_h) = ( A_h x_h,y_h)$, which we then decompose as 
$\sbilwaveh = \sbilwavehskew + \sbilwavehpos$,
see the definitions in \eqref{eq:SkewFormGlobal} and \eqref{eq:PosFormGlobal}
in  Appendix~\ref{sec:app_1}. 
Further, Lemma~\ref{lem:proof_ass_discr_forms_skew_stab} implies that Assumption~\ref{ass:discr_forms_skew_stab} is satisfied.
We choose all projections for the wave part as the component-wise broken
$L^2$-projections $\pi_h$. The corresponding initial and consistency errors
vanish, that is,
\[
\waveerrh(0) = \Delta \ipwave (\pt \wavevar) = \Delta f = 0.
\]
The projection $ \projheatRef $ for the heat part is chosen
as the $\div$-conforming interpolation which satisfies a commuting diagram property with the broken polynomial spaces.
The initial values and
the right-hand side are computed by this interpolation, and thus it holds
\begin{align*}
      \| \heaterrh(0) \|_{\wavespaceh}
     +
    \|\Delta g\|_{\heatspaceh^*}
     \lesssim h^{\ell} 
     \bigl(
         \| \heatvar  \|_{H^{\ell}} 
     + 
     \| g  \|_{H^{\ell}} 
     \bigr).
    \end{align*}
Further, we have
\begin{align*}
  |   \widetilde{\Delta}_{s,b} ( \wavevar, \heatvar,\varphi_h ) 
  |
    &\leq
    |   \sbilheat( \projheatRef \heatvar - \heatvar, \varphi_h)
    |
    +
    |
    \coupopfullform( \varphi_h ,  \wavevar - \projwaveRef \wavevar)
    |
    \\
    &\lesssim
\| \projheatRef \heatvar - \heatvar \|_{L^2}
\| \varphi_h \|_{L^2}
    \\
    &\lesssim
    h^{\ell}
    \| \heatvar \|_{H^{\ell}}
\| \varphi_h \|_{L^2} ,
\end{align*}
where we used the best-approximation property, see e.g., \cite[Theorem~5.25]{Mon03}, as well as
\begin{align*}
    \coupopfullform( \varphi_h ,  \wavevar - \projwaveRef \wavevar)
 &=
\sum\limits_{i=1}^G   \langle
        \div \sigmabold_{h,i}  , \pi_h \vbold - \vbold 
        \rangle_{L^2}
        =
        0 ,
\end{align*}
since $\div \sigmabold_{h,i}$ is an element of the broken polynomial space of degree $\ell-1 \leq k$, and hence orthogonal to the projection error.
For the differential parts, we use 
Assumption~\ref{ass:consistency}
and
obtain
\begin{align*}
     |\Delta_{s,b} ( \wavevar, \heatvar, \waveerrh) |
     &=
    | \sbilwaveh( \projwaveRef \wavevar, \waveerrh)
    +
    \coupopfullformh( \projheatRef \heatvar, \waveerrh)
    -
    \sbilwaveext( \wavevar, \waveerrh)
    - \coupopfullformext(\heatvar , \waveerrh) |
    \\
    &\leq
    | \sbilwaveh( \projwaveRef \wavevar -\wavevar, \waveerrh)
    |
    +
    |\coupopfullformh( \projheatRef \heatvar - \heatvar, \waveerrh) | .
\end{align*}
For the error in $\sbilwaveh$ we employ Lemma~\ref{lem:EllWaveApproxLemma},
and finally, we have by the commuting diagram property
\begin{align*}
 |   \coupopfullformh( \projheatRef \heatvar - \heatvar, \waveerrh)
  |
    &=
  |
\sum\limits_{i=1}^G   \langle
        \div ( \projheatRef - I )\sigmabold_{i}  , \waveerrhOne
        \rangle_{L^2}
   |
       \\
        &=
   |
\sum\limits_{i=1}^G   \langle
        ( \pi_h - I ) \div \sigmabold_{i}  , \waveerrhOne
        \rangle_{L^2}
        |
        \\
        &\lesssim
         h^{\ell}
        |\waveerrh |_{\wavespaceh}
    \sum\limits_{i=1}^G | \div \sigmabold_{i} |_{H^{k+1}}.
\end{align*}
For the fully discrete estimate, the target is
\(r_h=h^\ell+h^{k+1/2}\). Here \(\mathcal P_N\) follows directly from
the spatial consistency estimate: for any discrete \(y_h\),
\[
 |\Delta_{s,b}(\tauavg{\wavevar}^j,\tauavg{\heatvar}^j,y_h)|
 \lesssim h^{k+1/2}\stabnormwave{y_h}+h^\ell\normwave{y_h}.
\]
The two terms pair with the stabilization and energy parts of
\(\mathcal E_N\), respectively. Summing in time and applying Young's
inequality gives, for every \(\varepsilon>0\),
\[
 \mathcal P_N\leq\varepsilon\mathcal E_N
 +C_\varepsilon r_h^2.
\]
To control \(\mathcal Q_N\), we also need the reconstruction error and
the stabilization consistency estimates. The heat interpolation yields
\(\|\Delta\ipheat(\dtau\heatvarn{j})\|_{\heatspaceh^*}\lesssim h^\ell\).
The jump formula \eqref{eq:PosFormGlobal} and the face approximation estimate
\eqref{eq:LTwoApproxProp} give spatial decay for the jump term in
\(\mathcal D_N^{\mathrm{space}}\) and a uniform bound for the projection
stabilization term in \(\mathcal D_N^{\mathrm{bounded}}\):
\begin{align*}
 \sigma_{J,h}(v)&\lesssim h^{k+1/2}\|v\|_{H^{k+1}(\dGMesh)},\qquad\,\,
 \tau\sum_{j=1}^N\tau^2\sigma_{J,h}(\dtau\wavevarn{j})^2
    \lesssim\tau^2h^{2k+1},\\
 \sigma_{P,h}^{(1)}(w)&\lesssim h^{k+1/2}\|w\|_{H^{k+1}(\dGMesh)},\qquad
 \tau\sum_{j=1}^N\sigma_{P,h}^{(1)}(\dtau\pt\wavevarnOne{j})^2
    \lesssim h^{2k+1}.
\end{align*}
Set \(R_h=h^k+h^\ell\). The face inverse inequality and the preceding
consistency estimates imply, for \(a=1,2\),
\begin{align*}
 \|\Delta^s_{s,b}(\dtau^a\wavevarn{j},\dtau^a\heatvarn{j})\|_{\wavespaceh^*}
 &\lesssim R_h,\qquad \Delta\dtau^a f^j=0,\\
 \normwaveOne{\dtau\wavevarhOneDotStar{j}
 -\projwave^1\dtau\pt\wavevarnOne{j}}&\lesssim R_h,\qquad
 \normwaveOne{\dtau\wavevarhOneDotStar{j}}\lesssim1,\\
 \stabnormwave{(\dtau\wavevarhOneDotStar{j},0)}
 &\leq\stabnormwave{(\dtau\wavevarhOneDotStar{j}
 -\projwave^1\dtau\pt\wavevarnOne{j},0)}
 +\sigma_{P,h}^{(1)}(\dtau\pt\wavevarnOne{j})
 \\
 &\lesssim h^{-1/2}R_h+h^{k+1/2}.
\end{align*}
Using these bounds in the reconstructed-derivative split gives
\begin{align*}
 \mathcal Q_N
 &\lesssim\tau^2\bigl(h^kR_h+h^{2k+1}+h^\ell\bigr)
 \lesssim\tau^4+r_h^2.
\end{align*}
The mass and consistency estimates above, applied also to time differences,
control the remaining spatial defects:
\[
 \mathcal D_N^{\mathrm{space}}
 \lesssim r_h^2+\tau^3R_h^2+\tau^2h^{2k+1}
 \lesssim r_h^2+\tau^4.
\]
The first and second time differences, exact derivatives, and projection
stabilization term are uniformly bounded, so
\(\mathcal D_N^{\mathrm{bounded}}\lesssim1\). This verifies
\eqref{eq:app_verification}.

Combining the preceding bounds, we obtain for a sufficiently smooth solution
the spatial discretization error estimate
of optimal order
    \begin{equation*}
        \| \vbold(t) - \vbold_h(t)\|_{L^2} 
        +
        \| \sigmabold_0 (t) - \sigmabold_{0,h} (t)\|_{L^2} 
      +
    \sum\limits_{i = 1}^G
    \| \sigmabold_i (t) - \sigmabold_{i,h} (t)\|_{L^2} 
     \lesssim
     h^\ell + h^{k+1/2} ,
    \end{equation*}
with constants independent of $h$,
and the full discretization error bound
\begin{equation*}
\| \vbold(t_n) - \vbold_h^n\|_{L^2} 
+
\| \sigmabold_0 (t_n) - \sigmabold_{0,h}^n\|_{L^2} 
+
\sum\limits_{i = 1}^G
\| \sigmabold_i (t_n) - \sigmabold_{i,h}^n \|_{L^2} 
\lesssim
\tau^2 
+
h^\ell + h^{k+1/2},
\end{equation*}
with constants independent of $h$ and $\tau$,
where the last estimate holds under the CFL condition \eqref{eq:CFL}.
The typical choices for $\ell$ are $k$ or $k+1$.

\subsection{Thermo-elastic wave coupling}
\label{sec:app_thermo_elasto}

For the third example in Section~\ref{sec:ana_thermo_wave}, we use the wave
discretization from Section~\ref{app:ana_ell_wave}. The component-wise broken
$L^2$-projections again eliminate the corresponding initial and consistency
errors:
\[
\waveerrh(0) = \Delta \ipwave (\pt \wavevar) = \Delta f = 0.
\]
For the heat part we employ (continuous) Lagrange finite elements of 
order $\ell\leq k+1$
and choose for $\projheatRef$ the Ritz projection, and for the initial values and the right-hand side the nodal interpolation, which directly yields
\begin{align*}
      \| \heaterrh(0) \|_{\wavespaceh}
     +
    \|\Delta g\|_{\heatspaceh^*}
     \lesssim h^{\ell+1} 
     \bigl(
         \| \heatvar  \|_{H^{\ell+1}} 
     + 
     \| g  \|_{H^{\ell+1}} 
     \bigr).
    \end{align*}
Further, we have
\begin{align*}
  |   \widetilde{\Delta}_{s,b} ( \wavevar, \heatvar,\varphi_h ) 
  |
    &\leq
    |   \sbilheat( \projheatRef \heatvar - \heatvar, \varphi_h)
    |
    +
    |
    \coupopfullform( \varphi_h ,  \wavevar - \projwaveRef \wavevar)
    |
=0 ,
\end{align*}
using the definition of the Ritz projection $\projheatRef$ and again the fact that $\nabla  \varphi_h$ is in the space of broken polynomials of order $\ell-1\leq k$, such that the orthogonality 
\begin{align*}
    \coupopfullform( \varphi_h ,  \wavevar - \projwaveRef \wavevar)
 &=
 \alpha \langle
        \nabla \varphi_h  ,   \vbold - \pi_h \vbold 
        \rangle_{L^2}
        =
        0 
\end{align*}
holds. Using the estimate in Lemma~\ref{lem:EllWaveApproxLemma} for the upwind part, it remains to bound
\begin{align*}
     |\coupopfullformh( \projheatRef \heatvar - \heatvar, \waveerrh) |
     =
     |\langle
        \alpha 
        \nabla  (\projheatRef \heatvar - \heatvar) ,\waveerrhOne
        \rangle_{L^2}
        |
        \lesssim
h^\ell
  \| \heatvar  \|_{H^{\ell+1}} 
   \| \waveerrh  \|_{\wavespaceh}  .
\end{align*}
We can further improve this as in Section~\ref{sec:app_waveheat}
by first using the relation \(\nabla \theta = \div(I\cdot \theta)\) in
   \begin{align*}
    |\coupopfullformh( \projheatRef \heatvar - \heatvar, \waveerrh) |
    &=
    |\langle
    \alpha 
    \div \bigl( I \cdot (\projheatRef \heatvar - \heatvar) \bigr) ,\waveerrhOne
    \rangle_{L^2}
    |
    =
    \alpha 
    |\sbilwavehTwo(
    I \cdot (\projheatRef \heatvar - \heatvar)  ,\waveerrhOne
    )
    | .
\end{align*}
The latter identification holds due to the continuity of \(\projheatRef \heatvar - \heatvar\),
since all jump terms in the definition of $\sbilwaveh$ in \eqref{eq:SkewFormGlobal} vanish.
By skew-adjointness, we further simplify
\begin{align*}
    |\coupopfullformh( \projheatRef \heatvar - \heatvar, \waveerrh) |
    &=
    \alpha 
    | \sbilwavehOne(
    \waveerrhOne, I \cdot (\projheatRef \heatvar - \heatvar)
    )
    |
    =
    \alpha 
    | \ipwavehTwo  (
    \waveophOne \waveerrhOne, I \cdot (\projheatRef \heatvar - \heatvar)
    )
    |,
\end{align*}
and then obtain as before using integration by parts
   \begin{align*}
    & | \int_0^t \coupopfullformh( \projheatRef \heatvar(s) - \heatvar(s), \waveerrh(s)) \,ds |
    \\&=
    \alpha 
    | \int_0^t \ipwavehTwo  (
    \waveophOne \waveerrhOne (s), I \cdot (\projheatRef \heatvar (s)- \heatvar(s))
    )
    \,ds
    |
    \\
    &\leq
    \alpha 
    | \ipwavehTwo  (
    \int_0^{t}  \waveophOne \waveerrhOne (s_1) \,d{s_1} , I \cdot (\projheatRef \pt \heatvar (t)- \pt \heatvar(t))
    )
    \\
    &\quad
    +
    \alpha 
    | \int_0^t \ipwavehTwo  (
    \int_0^{s}  \waveophOne \waveerrhOne (s_1) \,d{s_1} , I \cdot (\projheatRef \pt \heatvar (s)- \pt \heatvar(s))
    )
    \,ds
    | .
\end{align*}
The error equation
$ 
\pt \waveerrhTwo = \waveophOne \waveerrhOne + \waveophposTwo \waveerrhTwo 
+
\Delta_{s,b} ( \wavevar, \heatvar, \cdot  )
$
and the estimate in Lemma~\ref{lem:EllWaveApproxLemmaTwo} imply
\begin{align*}
	& | \Delta_{s,b} (
	 \wavevar, \heatvar,
	( 0,  I \cdot (\projheatRef \pt \heatvar - \pt \heatvar ) )
	)  |
	\\
	&\leq 
	| \sbilwaveh( \projwaveRef \wavevar -\wavevar,
	( 0,  I \cdot (\projheatRef \pt \heatvar - \pt \heatvar ) )
	   )
	|
	+
	|\coupopfullformh( \projheatRef \heatvar - \heatvar,
	 ( 0,  I \cdot (\projheatRef \pt \heatvar - \pt \heatvar ) ) 
	  ) |
	\\
	&\lesssim 
	h^{k + \ell + 1} + h^{2\ell +1}.
\end{align*}
This leads to the bound
  \begin{align*}
    & | \int_0^t \coupopfullformh( \projheatRef \heatvar(s) - \heatvar(s), \waveerrh(s)) \,ds |
    \lesssim
    h^{\ell +1}
    \bigl(
    \| \waveerrhTwo(t)  \|_{\wavespacehTwo}
    +
    \| \waveerrhTwo(0)  \|_{\wavespacehTwo}
    +
    h^{k}
    +
    h^{\ell}
     \bigr) .
\end{align*}
For the fully discrete estimate, the target is
\(r_h=h^{k+1/2}+h^{\ell+1/2}\). The upwind part of
\(\mathcal P_N\) pairs directly with the stabilization seminorm, as in
the previous example. The coupling part again differentiates the wave
error, so we use the summation-by-parts argument from
Section~\ref{sec:app_waveheat}. Its remaining quantity is
$
 h^{\ell+1}
   \| \tau \sum_{\ell = 0}^n \waveophOne \tauavg{\waveerrhnOne{\ell+1}} \|_{\wavespaceVTwo^*}
$.
We use the representation in \eqref{eq:error_eqfor_Sh1eh1} to estimate
\begin{align*}
    h^{\ell+1}
   \| \tau \sum_{\ell = 0}^n \waveophOne \tauavg{\waveerrhnOne{\ell+1}} \|_{\wavespaceVTwo^*}
    &\lesssim
    h^{\ell +1}
    \bigl(
    \| \colerrhn{n+1} \|_{\ipwave \times \ipheat}
    +
    \| \colerrhn{0} \|_{\ipwave \times \ipheat}
    + h^{k}
    + h^\ell
    +\tau^2
     \bigr) ,
\end{align*}
along the lines of Section~\ref{sec:app_waveheat}.
Together with the corresponding summation-by-parts identity and the upwind
consistency bound, this gives
\[
 \mathcal P_N\lesssim
 (h^{k+1/2}+h^{\ell+1})\mathcal E_N^{1/2}
 +h^{\ell+1}(h^k+h^\ell+\tau^2)
 \leq\varepsilon\mathcal E_N
 +C_\varepsilon\bigl(h^{2k+1}+h^{2\ell+1}+\tau^4\bigr)
\]
for every \(\varepsilon>0\).
For the fully discrete analysis, the jump estimates from
Section~\ref{app:ana_ell_wave} give
\[\tau\sum_{j=1}^N\tau^2\sigma_{J,h}(\dtau\wavevarn{j})^2
\lesssim\tau^2h^{2k+1},\qquad
\tau\sum_{j=1}^N\sigma_{P,h}^{(1)}(\dtau\pt\wavevarnOne{j})^2
\lesssim h^{2k+1},\]
while the heat
mass defect satisfies
\(\|\Delta\ipheat(\dtau\heatvarn{j})\|_{\heatspaceh^*}\lesssim h^{\ell+1}\).
Set \(R_h=h^k+h^\ell\). The same inverse estimates yield
\begin{align*}
 \|\Delta^s_{s,b}(\dtau^a\wavevarn{j},\dtau^a\heatvarn{j})\|_{\wavespaceh^*}
 &\lesssim R_h\quad(a=1,2),\\
 \normwaveOne{\dtau\wavevarhOneDotStar{j}
 -\projwave^1\dtau\pt\wavevarnOne{j}}&\lesssim R_h,\\
 \stabnormwave{(\dtau\wavevarhOneDotStar{j},0)}
 &\lesssim h^{-1/2}R_h+h^{k+1/2}.
\end{align*}
For \(\mathcal Q_N\), the projected exact derivative in
\eqref{eq:app_reconstructed_derivative_split} needs an extra spatial
integration by parts to gain the heat-projection order. Set
\(\zeta^j=(\projheatRef-I)\tauavg{\heatvar}^j\). Elementwise integration
by parts gives
\begin{align*}
 &\left|\left\langle\nabla\zeta^j,
 \projwave^1\dtau\pt\wavevarnOne{j}\right\rangle_{L^2(\Omega)}\right|\\
 &\quad=\Biggl|
 -\sum_{K\in\dGMesh}\left\langle\zeta^j,
 \div(\projwave^1\dtau\pt\wavevarnOne{j})\right\rangle_{L^2(K)}
 +\sum_{F\in\dGFaceInt}\left\langle\zeta^j,
 \dGJmp{\projwave^1\dtau\pt\wavevarnOne{j}}_F\cdot n_F
 \right\rangle_{L^2(F)}\Biggr|\\
 &\quad\lesssim h^{\ell+1}+h^{\ell+1/2}h^{k+1/2}
 \lesssim h^{\ell+1}.
\end{align*}
The reconstruction-error contribution satisfies
\begin{equation*}
 \left|\left\langle\nabla\zeta^j,
 \dtau\wavevarhOneDotStar{j}
 -\projwave^1\dtau\pt\wavevarnOne{j}
 \right\rangle_{L^2(\Omega)}\right|
  \lesssim h^\ell R_h.
\end{equation*}
The boundary term vanishes because \(\zeta^j=0\) on the boundary, and the
jumps of \(\projwave^1\dtau\pt\wavevarnOne{j}\) are projection errors because
\(\dtau\pt\wavevarnOne{j}\) is continuous. Combining these observations with
the upwind consistency bound
yields
\begin{align*}
 \mathcal Q_N
 &\lesssim\tau^2\bigl(h^kR_h+h^{2k+1}+h^{\ell+1}+h^\ell R_h\bigr)\\
 &\lesssim\tau^4+h^{2k+1}+h^{2\ell+1}.
\end{align*}
The mass and consistency estimates, including their time differences,
complete the spatial-defect bound:
\[
 \mathcal D_N^{\mathrm{space}}
 \lesssim r_h^2+\tau^3R_h^2+\tau^2h^{2k+1}
 \lesssim r_h^2+\tau^4.
\]
The first and second time differences, exact derivatives, and projection
stabilization term give
\(\mathcal D_N^{\mathrm{bounded}}\lesssim1\). The two pairing estimates
then verify \eqref{eq:app_verification}.

As in the previous example, a sufficiently smooth solution therefore satisfies
the optimal-order spatial error bound
    \begin{equation*}
        \| \vbold(t) - \vbold_h(t)\|_{L^2} 
        +
        \| \sigmabold (t) - \sigmabold_{h} (t)\|_{L^2} 
      +
\| \heatvar (t) -  \heatvarh (t) \|_{L^2} 
     \lesssim
     h^{\ell+1/2} + h^{k+1/2}
    \end{equation*}
with constants independent of $h$,
and the full discretization error bound
\begin{equation*}
\| \vbold(t_n) - \vbold_h^n\|_{L^2} 
+
\| \sigmabold (t_n) - \sigmabold_{h}^n\|_{L^2} 
+
\| \heatvar (t_n) -  \heatvarhn{n} \|_{L^2} 
\lesssim
\tau^2 
+
h^{\ell+1/2} + h^{k+1/2},
\end{equation*}
with constants independent of $h$ and $\tau$,
where the last estimate holds under the CFL condition \eqref{eq:CFL}.
The natural choice for $\ell$ is equal to $k$. %

\subsection{Maxwell equation with dissipative 2D material}
\label{sec:app_MaxwADE}
Finally, for the example in Section~\ref{sec:ana_Maxwell} we apply a dG method with central flux (upwind could be included as well) to the Maxwell part $\wavevar$ and use its restriction to the interface as the ansatz space for $\heatvar$.
To be precise, we use the space of broken polynomials of degree $k$ for $\wavespaceh$, which yields again a broken polynomial space of degree $k$ on the (lower dimensional) restriction to the interface.
This leads to the standard
 dG Maxwell operators in $\sbilwaveh$, which can for example be found in \cite[Section~4.2]{HocS19},
 and thus Assumption~\ref{ass:discr_forms_skew_stab} is satisfied.
For the coupling form $\coupopformh$, we set
\begin{align*}
    \coupopfullformh( \varphi_h , y_h)
    =
    \ipheat( \varphi_h , \avg{ (c_\ell \otimes I_2)  y_{h,1,\parallel} }),
\end{align*}
 for $y_h \in \wavespaceh$ and $\varphi_h \in \heatspaceh$.
We note that the exact solution 
satisfies the generalized consistency condition of Assumption~\ref{ass:consistency},
see \cite[Lem.~3.10~(2)]{DoeDH26} for an analogous statement of external surface currents.

We now study the error terms in \eqref{eq:def_space_errors} appearing in Theorem~\ref{thm:space}, and choose the standard $L^2$-projections in all cases. 
This immediately implies that 
$\waveerrh(0)  =  \heaterrh(0) = 0$,
and similarly, as we work in the correct inner products also
$\Delta \ipwave (\pt \wavevar) = \Delta \ipheat (\pt \heatvar)  = 0 $. 
 In the same manner, if we compute the sources $f_h$ and $g_h$ via projections, we also have 
 $\Delta f  = \Delta g  = 0$.
As in the elastic cases, Assumption~\ref{ass:consistency}
yields
\begin{align*}
     |\Delta_{s,b} ( \wavevar, \heatvar, \waveerrh) |
     &=
    | \sbilwaveh( \projwaveRef \wavevar, \waveerrh)
    +
    \coupopfullformh( \projheatRef \heatvar, \waveerrh)
    -
    \sbilwaveext( \wavevar, \waveerrh)
    - \coupopfullformext(\heatvar , \waveerrh) |
    \\
    &=
    | \sbilwaveh( \projwaveRef \wavevar -\wavevar, \waveerrh)
    +
    \coupopfullformh( \projheatRef \heatvar - \heatvar, \waveerrh) |.
\end{align*}
The remaining term,
can be estimated using 
\cite[eq. (5.5)]{Hoc.S.2016}, as
\begin{align*}
  |\Delta_{s,b} ( \wavevar, \heatvar, \waveerrh) |
\lesssim
h^{k}
\bigl( 
| \wavevar |_{H^{k+1}}
+
| \heatvar |_{H^{k+1}}
\bigr)
\|\waveerrh \|_{\wavespaceh}.
\end{align*}

For the fully discrete estimate, the target is \(r_h=h^k\). With the
central flux, \(\mathcal P_N\) has the direct estimate: its spatial
consistency defect is at most \(h^k\) times the energy norm of the wave
error. Thus, for every \(\varepsilon>0\),
\[
 \mathcal P_N\leq\varepsilon\mathcal E_N+C_\varepsilon r_h^2.
\]
The central flux also gives \(\colstaboph=0\), so both stabilization
seminorms vanish and \(\Delta^s_{s,b}=\Delta_{s,b}\). Applying the same
consistency estimate to the first two time differences gives
\[
 \|\Delta^s_{s,b}(\dtau^a\wavevarn{j},\dtau^a\heatvarn{j})\|_{\wavespaceh^*}
 \lesssim h^k\quad(a=1,2),\qquad
 \normwaveOne{\dtau\wavevarhOneDotStar{j}}\lesssim1.
\]
The trace approximation estimate
for the projected bulk field gives
\[
 \|\widetilde\Delta_{s,b}(\tauavg{\wavevar}^j,
 \tauavg{\heatvar}^j)\|_{\heatspaceh^*}\lesssim h^{k+1/2}.
\]
Testing the same consistency defect against the reconstructed derivative
gives \(\mathcal Q_N\lesssim\tau^2h^k\lesssim\tau^4+r_h^2\).
The mass and source defects vanish by projection, while the interface
trace estimate and first time difference yield
\[
 \mathcal D_N^{\mathrm{space}}
 \lesssim h^{2k+1}+\tau^3h^{2k}
 \lesssim r_h^2+\tau^4.
\]
The first and second time differences and bounded exact derivative give
\(\mathcal D_N^{\mathrm{bounded}}\lesssim1\), so
\eqref{eq:app_verification} applies.

Finally, we assume that the solution is sufficiently regular, 
and note the best approximation results in
\cite[Section~18.4]{Ern.G.2021.FEAI}.
Then, we obtain for
the spatial discretization the estimate
\begin{equation*}
    \| E(t) - E_h(t)\|_{L^2} 
    +
    \| H(t) - H_h(t)\|_{L^2} 
    +
\| \heatvar(t) - \heatvarh(t)\|_{L^2} 
    \lesssim
    h^{k}
\end{equation*}
with constants independent of $h$,
and for the fully discrete scheme 
\begin{equation}\label{eq:MaxwellConvergenceResult}
\| E(t_n) -  E_h^n\|_{L^2} 
+
\| H(t_n) -  H_h^n\|_{L^2} 
+    
\| \heatvar(t_n) -  \heatvarh^n\|_{L^2} 
\lesssim
\tau^2 +       h^{k} ,
\end{equation}
with constants independent of $h$ and $\tau$,
where the last estimate holds under the CFL condition \eqref{eq:CFL}.

\section{Numerical experiments}
\label{sec:num_exp}

In this section, we present numerical experiments for three of our model problems. We begin with the thermo-elastic string from 
Section~\ref{sec:ana_thermoelastic_string}.
For it, we study convergence in space and time and examine the behavior of the CFL condition stated in Assumption~\ref{ass:CFL}.
    Our second example is concerned with the Thermo-elastic wave coupling described in Section~\ref{sec:ana_thermo_wave}, where we demonstrate the use of an upwind stabilized discontinuous Galerkin scheme in space for the elastic wave equation.
    The third example presents Maxwell's equation coupled to an ordinary differential equation, as described in Section~\ref{sec:ana_Maxwell}.
    For the latter two examples, we verify the space and time convergence, as described in Section~\ref{sec:app_thermo_elasto} and \ref{sec:app_MaxwADE} respectively.

Following the Guidelines for Safeguarding Good Research Practice of the Deutsche Forschungsgemeinschaft,\footnote{Deutsche Forschungsgemeinschaft, \textit{Guidelines for Safeguarding Good Research Practice. Code of Conduct}, \url{https://doi.org/10.5281/zenodo.3923601}.} we make the software underlying this section publicly available.
The experiments are implemented in deal.II \cite{2026:arndt.bangerth.ea:deal} and FEniCSx \cite{BarattaEtal2023}.
The source codes, version-pinned environments, and the parameter sets that reproduce the individual results reported below are available at 
\begin{center}
    \url{https://github.com/juliandoerner/lfcn_paper_code}.
\end{center}

\subsection{Thermo-elastic string}
\label{sec:app_waveheat:numerical_experiments}

We consider the example from Section~\ref{sec:ana_thermoelastic_string}.
We first verify second‑order convergence in time and the CFL condition in  Assumption~\ref{ass:CFL}. 
Therefore, let
\begin{equation} \label{eq:strong_ex_sol_time}
    \wavevarTwo(t,x) = e^{t} \cos{( \omega t)} x (1-x),
    \qquad \theta(t,x) = e^{t} \sin{(\eta t)} x (1-x),
\end{equation}
and choose \(f\) and \(g\) such that the \(\colvar = (\wavevarOne,\wavevarTwo,\heatvar)\) solves \eqref{eq:acous_wave_heat_1d}. 
In the following examples we set \(\omega = 3.37\cdot\pi\), \(\eta = 2.3\cdot\pi\) and the final time \(T = 1.0\).
Note that the exact solution is a second‑order polynomial in space and is therefore represented exactly by second‑order ansatz functions on any mesh. 
Consequently, we only measure the temporal integration error in the experiments below. 
For a sequence of eight mesh sizes ranging from \(0.5\) to \(0.005\), we compute the \(L^\infty\)-error in time and the \(L^2 \times H^1_0\times L^2\) error in space against the exact solution, using 40 different time‑step sizes between \(0.1\) and \(0.001\). 
The results are displayed in Figure~\ref{fig:wave_1d:tau_plot}. 
We conclude that our method exhibits second‑order convergence in time. 

\begin{figure}
    \begin{center}
        \includeexternalstandalonewithpath[]{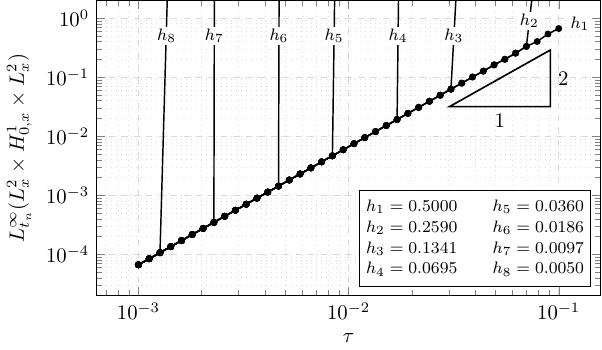}
    \end{center}
    \caption{
     \(L^\infty\)-error in time and the \(L^2 \times H^1_0\times L^2\)-error in space of the discrete solution against exact solution 
     \eqref{eq:strong_ex_sol_time}, plotted against the time-step size $\tau$.}
    \label{fig:wave_1d:tau_plot}
\end{figure}

In the next example, we determine the threshold time step \( \tau_{\text{CFL}} \) for which our method is \textit{stable}.
For a given target mesh size \( h_{\text{target}} \) we choose a sufficiently large interval \( \tau_{\text{CFL}} \in [\tau_{\text{low}}, \tau_{\text{high}}] \) such that the method is \textit{unstable} at \( \tau_{\text{high}} \) and \textit{stable} at \( \tau_{\text{low}} \) when compared with a reference calculation using a mesh size \( h_{\text{ref}} = 1.0 \).
Using the bisection method we halve the interval and repeat the comparison at each step until the interval width falls below \( 10^{-5} \).
The midpoint of the final interval is then taken as an approximation of \( \tau_{\text{CFL}} \).
Figure~\ref{fig:wave_1d:cfl_trace} shows the results of this procedure for 20 different target mesh sizes ranging from \( 0.1 \) to \( 0.001 \) and different polynomial degrees of ansatz functions.
From these data we conclude that the method is stable for all time-step sizes \( \tau < \tau_{\text{CFL}} \approx h\).

\begin{figure}
    \begin{center}
        \includeexternalstandalonewithpath[]{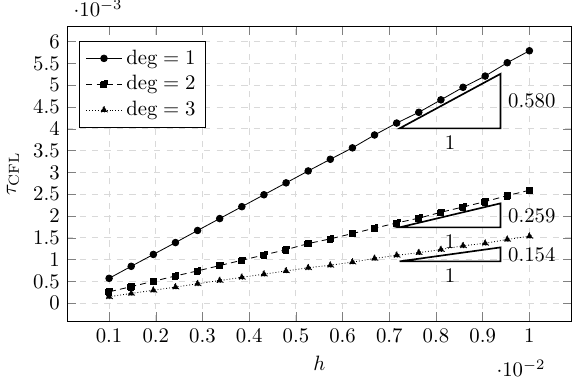}
    \end{center}
     \caption{
    Approximation of \( \tau_{\text{CFL}}\)
    plotted against the mesh size~$h$.
    }
    \label{fig:wave_1d:cfl_trace}
\end{figure}

The next example verifies the convergence order in space.
We define a different exact solution, i.e.
\begin{equation} \label{eq:strong_ex_sol_space}
    \wavevarTwo(t,x) = t \sin{( 3\pi  x)} e^{-20 (x-\frac{1}{2})^2} ,
    \qquad \theta(t,x) = t \sin{( 4\pi  x)} e^{-20 (x-\frac{1}{2})^2},
\end{equation}
and choose again \(\wavevarOne, f\) and \(g\) such that the collection is a solution of \eqref{eq:acous_wave_heat_1d}.
Note that the exact solution is polynomial in time and therefore integrated exactly by our scheme.
Hence, we only measure the space discretization error in the experiment below.
For a sequence of 40 different mesh sizes, we calculate the error against the exact solution with a time step size of \(0.001\).
    For the first sequence of meshes, we choose the polynomial degree of the heat and the wave ansatz space to be equal, i.e. \(k = \ell = p\) for \(p\in\monosetc{1,2,3,4}\).
    For the second sequence, we choose \(k = \ell+1 = p\) for \(p\in\monosetc{2,3,4}\).
Figure~\ref{fig:wave_1d:h_plot} shows the result of this experiment and confirms our findings from Section~\ref{sec:app_waveheat}.
    Especially, the calculations for degree \(k = \ell +1\) show that this choice is optimal.

\begin{figure}
    \begin{center}
        \includeexternalstandalonewithpath[]{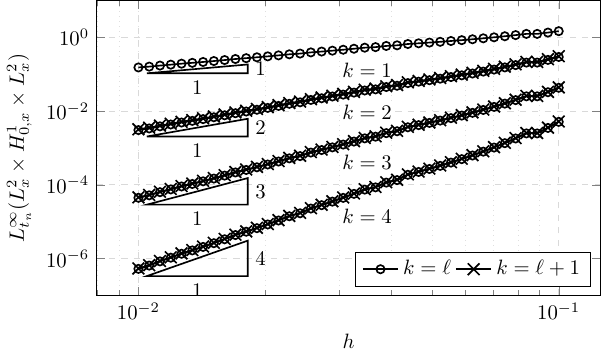}
    \end{center}
     \caption{
     \(L^\infty\)-error in time and the \(L^2 \times H^1_0\times L^2\)-error in space of the discrete solution against the exact solution 
     \eqref{eq:strong_ex_sol_space}, plotted against the mesh size $h$ for $\tau = 0.001$.}
    \label{fig:wave_1d:h_plot}
\end{figure}

\subsection{Thermo-elastic wave coupling}
\label{sec:app_thermo_elasto:numerical_experiment}

    We consider the thermo-elastic wave coupling \eqref{eq:thermo_elastic} from Section~\ref{sec:ana_thermo_wave}.
    On the two-dimensional domain \(\Omega = (0,1)^2\), we construct an exact solution, and show the spatial and temporal convergence rates as described in Section~\ref{sec:app_thermo_elasto}.
    
    Let \(\bm{V}\,:\,\Omega\to\mathbb{R}^2\) and \(\Theta\,:\, \Omega\to\mathbb{R}\) denote two spatial profiles that satisfy homogeneous boundary conditions on \(\partial\Omega\).
    For three temporal profiles \(a,A,c\,:\,\mathbb{R}_+\to\mathbb{R}\) with \(a(0) = 1\) and \( A(0) = c(0) = 0\), we make the ansatz
    \begin{equation*}
        \vbold(t,\bm{x})=a(t)\bm{V}(\bm{x}),
        \qquad 
        \sigmabold(t,\bm{x}) = A(t)C\epsilon(\bm{V}(\bm{x})),
        \qquad
        \theta(t,\bm{x}) = c(t)\Theta(\bm{x}),
    \end{equation*}
    where \eqref{eq:thermo_elastic:evo_sigma} demands the relation \(A'(t) = a(t)\).
    The remaining equations \eqref{eq:thermo_elastic:evo_v} and \eqref{eq:thermo_elastic:evo_theta} are satisfied by choosing the data as
    \begin{align}
       \bm{f}(t,\bm{x}) &= 
        a'(t)\bm{V}(\bm{x})
               -\rho^{-1}A(t)\div\big(C\epsilon(\bm{V}(\bm{x}))\big)
               -\alpha\rho^{-1}c(t)\nabla\Theta(\bm{x}),\\
       h(t,\bm{x}) &= c'(t)\Theta(\bm{x})-c(t)\kappa\Delta\Theta(\bm{x})-a(t)\alpha\div\bm{V}.
   \end{align}
   In all experiments we choose the material coefficients as
   \begin{equation*}
        \rho = 1.1,
        \qquad
        \mu = 2.32,
        \qquad
        \lambda = 1.42,
        \qquad 
        \kappa = 4.2,
        \qquad
        \alpha = 1,
   \end{equation*}
   and calculate until end time \(T = 1\).

   First, we verify second-order convergence in time.
   For that, we choose a polynomial profile \(P(\bm{x}) = x_1(1-x_1)x_2(1-x_2)\) in space and define
   \begin{equation}\label{eq:ExactSolutionThermoWavePoly}
        \bm{V}(\bm{x}) = (P(\bm{x}),P(\bm{x})), 
        \quad
        \Theta(\bm{x}) = P(\bm{x}),
        \quad
        \quad a(t) = \cos(\omega t),
        \quad c(t) = \sin(\eta t),
   \end{equation}
   with temporal frequencies \(\omega,\eta > 0\).
   Note that we use second-order ansatz functions to resolve the spatial profile exactly.
   Hence, only the temporal error is measured in the following.
   For a sequence of 10 different meshes with target mesh sizes ranging between \(0.2\) and  \(0.02\), we compute the \(L^\infty\)-error in time and the \(L^2(\Omega)^2\times L^2(\Omega)^{2\times 2}\times L^2(\Omega)\)-error in space of our scheme against the manufactured solution.
   Figure~\ref{fig:wave_thermo_elast:tau_plot} verifies second order in time if the CFL condition \(\tau\lesssim h\) is satisfied.
   \begin{figure}
       \begin{center}
           \includeexternalstandalonewithpath[]{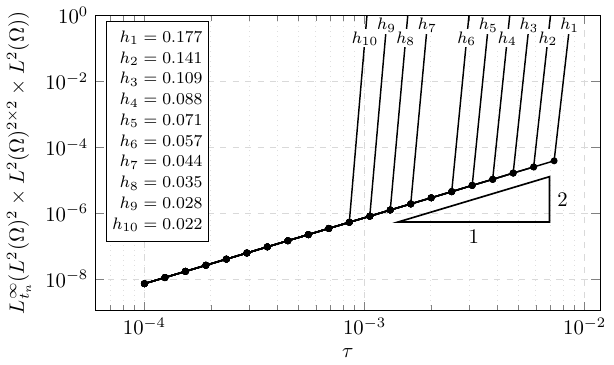}
       \end{center}
       \caption{
           \(L^\infty\)-error in time and the \(L^2(\Omega)^2\times L^2(\Omega)^{2\times 2}\times L^2(\Omega)\)-error of the discrete solution against the exact solution 
        \eqref{eq:ExactSolutionThermoWavePoly}, plotted against the time step-size \(\tau\).
       }
       \label{fig:wave_thermo_elast:tau_plot}
   \end{figure}
   
   Second, we verify the spatial convergence rate.
   Therefore, we choose a cavity-like profile \(P(\bm{x}) = \sin(2\pi x_1)\sin(2\pi x_2)\) and define
    \begin{equation}\label{eq:ExactSolutionThermoWaveCavity}
        \bm{V}(\bm{x}) = (P(\bm{x}),P(\bm{x})), 
        \quad
        \Theta(\bm{x}) = P(\bm{x}),
        \quad
        \quad a(t) = 1,
        \quad c(t) = t.
   \end{equation}
   Note that our second order in time scheme resolves the temporal differential equation exactly since all third time derivatives of the solution vanish.
   We perform two different calculations on a sequence of 15 meshes with mesh sizes ranging between \(0.2\) and \(0.02\) each.
   The same error as in the temporal case is calculated.
   For the first set of calculations, we choose the polynomial degree \(k\) of the wave and \(\ell\) of the  heat ansatz spaces equal, i.e. \(k= \ell\).
   The result is shown in Figure~\ref{fig:wave_thermo_elast:h_plot} with circle markers.
   For \(p\in \monosetc{1,2,3}\) we obtain in the case \(k=\ell= p\) spatial convergence  of order \(h^{p+1}\).
   This does not agree with our result in Section~\ref{sec:app_thermo_elasto}, where we show convergence of order \(h^{p+1/2}\).
   This, however, is a well known phenomena in the case of upwind-flux stabilized dG schemes.
   For that, compare \cite[Sec.~4.5]{Hes.War.2008} and reference therein.
   For the second set of calculations, we choose different polynomial degrees for the wave and heat part, namely \(k = \ell + 1 = p\) for \(p\in \monosetc{2,3}\).
   Our theory predicts order \(h^{p-1/2}\).
   Figure~\ref{fig:wave_thermo_elast:h_plot} depicts the result with cross markers.
   After the pre-asymptotic regime for coarser mesh sizes, ansatz functions with degree \(p\) give convergence of order \(h^{p}\).
    Up to the half-order gap due to upwind flux stabilization,
    the suboptimal choice of \(\ell = p-1\) leads to an order reduction, as suggested by our estimates in Section~\ref{sec:app_thermo_elasto}.

\begin{figure}
    \begin{center}
        \includeexternalstandalonewithpath[]{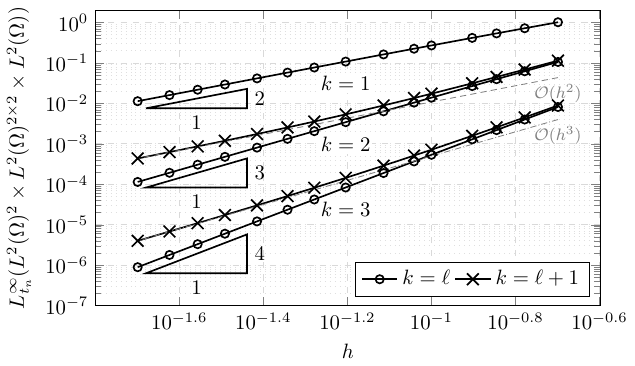}
    \end{center}
     \caption{
        \(L^\infty\)-error in time and the \(L^2(\Omega)^2\times L^2(\Omega)^{2\times 2}\times L^2(\Omega)\)-error of the discrete solution against the exact solution 
        \eqref{eq:ExactSolutionThermoWaveCavity}, plotted against the mesh-size \(h\).
     }
    \label{fig:wave_thermo_elast:h_plot}
\end{figure}

\subsection{Maxwell equation with dissipative 2D material.}
\label{sec:app_MaxwADE:numerical_experiment}

In order to reduce the degrees of freedom, we perform all calculations in transverse electric (TE) mode.
Hence, let \(\Omega = (-1,1)\times (0,1) \subset \mathbb{R}^2\), partitioned by an interface \(\Fint\) such that
\begin{equation*}
    \Omega_1 = (-1,0) \times (0,1),
    \quad
    \Omega_2 = (0,1) \times (0,1),
    \quad
    \Fint= \{0\}\times (0,1).
\end{equation*}
We set the components \(E_3 = H_1 = H_2 = 0\) and the evolution of the Maxwell equation in \eqref{eq:Maxwell_interface} reduce to
\begin{subequations}
    \begin{align}
        \partial_t E_{1,i} &=
        \partial_2 H_{3,i} - J_{1,i},
        &&[0,T] \times \Omega_{i},\label{eq:TEMaxwell:EOne}\\
        \partial_t E_{2,i} &=
        -\partial_1 H_{3,i} - J_{2,i},
        &&[0,T] \times \Omega_{i},\\
        \partial_t H_{3,i} &=
          -\partial_1 E_{2,i} + \partial_2 E_{1,i},
          &&[0,T] \times \Omega_{i},\label{eq:TEMaxwell:HThree}
    \end{align}
\end{subequations}
subject to initial conditions and perfectly conducting boundary conditions.
As an auxiliary differential equation we choose a Lorentz oscillator
\begin{equation}\label{eq:LorentzOsc}
    \partial_t^2 J_{\Fint}
    + \gamma \partial_t J_{\Fint}
    + \omega^2 J_{\Fint}
    = \partial_t E_{2,\parallel} + \partial_t f,
    \qquad  [0,T] \times \Fint,
\end{equation} 
with frequency \(\omega > 0\), damping \(\gamma >0\) and forcing function \(f\).
The coupling is described by
\begin{equation*}
    \jmp{H_3}_{\Fint} = -J_\Fint,\qquad [0,T]\times \Fint.
\end{equation*}
We rewrite the second-order differential equation \eqref{eq:LorentzOsc} by introducing a new variable
\begin{equation*}
    K_\Fint(t,\cdot) = \omega \int_0^t J_\Fint(s,\cdot)\mathrm{d}s,
    \quad \partial K_\Fint(t,\cdot) = \omega J_\Fint(t,\cdot),
    \quad K_\Fint(0,\cdot) = 0, 
\end{equation*}
and obtain for \(\theta = (K_\Fint,J_\Fint)\) the first-order system
\begin{equation*}
    \partial_t \theta + L\theta = c E_{2,\parallel} + g,
    \quad
    L = 
    \begin{pmatrix}
        0 && -\omega\\
        \omega && \gamma
    \end{pmatrix},
    \quad
    c = 
    \begin{pmatrix}
        0\\1
    \end{pmatrix},
    \quad
    g = 
    \begin{pmatrix}
        0\\ f
    \end{pmatrix}.
\end{equation*}
Note that the single components of the auxiliary differential equation in TE mode are scalar valued.
Therefore, the kronecker products are omitted.

For our first example, 
we construct a Cavity like solution alongside a discontinuous profile, defined as
\begin{equation*}
    h_1(s) = \cos\bigl( \beta(s-1) \bigr),
    \qquad h_2(s) = -\cos\bigl(\beta(s+1)\bigr),
\end{equation*}
with a parameter \(\beta > 0\) such that \(\cos(\beta) \neq 0\).
Our ansatz on \(\Omega_i\) now reads
\begin{subequations}\label{eq:ExactSolutionMaxwellCavity}
\begin{align}
    E_{1,i}(t,x_1,x_2) &= -k\nu^{-1}\sin(\nu t)\sin(k x_2)h_i(x_1),\\
    E_{2,i}(t,x_1,x_2) &= -\nu^{-1}\sin(\nu t)\cos(k x_2)h_i'(x_1),\\
    H_{3,i}(t,x_1,x_2)     &= \cos(\nu t)\cos(k x_2)h_i(x_1),
\end{align}
where \(\nu = \sqrt{\beta^2 + k^2}\) with wave number \(k =n\pi\), \(n\in \mathbb{N}\).
The ansatz satisfies \eqref{eq:TEMaxwell:EOne}-\eqref{eq:TEMaxwell:HThree}
 with homogeneous data \(J_1 = J_2  = 0\).
As a surface current, we obtain
\begin{equation}
    J_\Fint(t,x_1,x_2) = -\jmp{H_{3}(t,x_1,x_2)}_\Fint = - 2\cos(\beta) \cos(\nu t) \cos(kx_2),
\end{equation}
\end{subequations}
and choose the forcing function \(f\) such that \eqref{eq:LorentzOsc} is satisfied.
In order to verify the spatial convergence rate in \eqref{eq:MaxwellConvergenceResult}, we choose  \(\omega = 0.5\), \(\gamma = 1\), \(\beta = 3\) and \(n = 10\).
For a series of 20 target mesh sizes ranging between \(0.05\) and \(0.005\) we compute simulations over the temporal interval \([0,1]\) with constant time step size \(\tau = 0.0001\).
We compute the \(L^\infty\)-error in time and the \(L^2(\Omega)^2 \times L^2(\Omega) \times L^2(\Fint)^2\)-error in space against the exact solution.
The results are displayed in Figure~\ref{fig:maxwell:h_plot} for three different polynomial degrees.
We conclude that our method exhibits the correct spatial convergence rate.
\begin{figure}
    \begin{center}
        \includeexternalstandalonewithpath[]{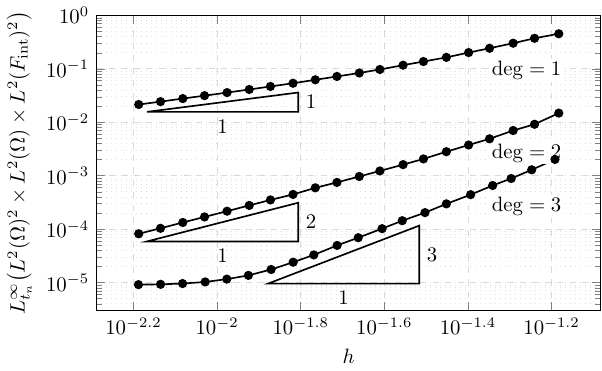}
    \end{center}
    \caption{
        \(L^\infty\)-error in time and the \(L^2(\Omega)^2 \times L^2(\Omega) \times L^2(\Fint)^2\)-error of the discrete solution against the exact solution 
        \eqref{eq:ExactSolutionMaxwellCavity}, plotted against the mesh-size \(h\).
     }
    \label{fig:maxwell:h_plot}
\end{figure}

For our second example, we construct an exact solution that is polynomial in space.
This allows to measure only the temporal error.
We choose the discontinuous profile
\(
    r_1(s) = 1-s
\) and 
\(
    r_2(s) = 1+s
\),
and make on \(\Omega_i\) the ansatz
\begin{subequations}\label{eq:ExactSolutionMaxwellPoly}
    \begin{align}
        E_{1,i}(t,x_1,x_2) &= -\nu\sin(\nu t)x_2(1-x_2)(1-x_1^2),\\
        E_{2,i}(t,x_1,x_2) &= -\nu\sin(\nu t)r_i(x_1),\\
        H_{3,i}(t,x_1,x_2)     &= \cos(\nu t)\bigl((-1)^{i} + (1-2x_2)(1-x_1^2)\bigr),
    \end{align}
\end{subequations}
with temporal frequency \(\nu > 0\).
The volume currents \(J_{1,i}\) and \(J_{2,i}\) are chosen such that \eqref{eq:TEMaxwell:EOne}-\eqref{eq:TEMaxwell:HThree}
 are satisfied.
Furthermore, the surface current is then given by \(J_{\Fint} = -2\cos(\nu t)\), and the forcing \(f\) is chosen such that \eqref{eq:LorentzOsc} is satisfied.
In our example, we choose polynomial degree two for the finite element space and fix the variables \(\gamma = 1\), \(\omega = 10\), and \(\nu = 24\). 
For a series of ten target mesh-sizes between $0.1$ and $0.01$, and 50 times step sizes between $0.1$ and $0.0001$, we calculate the same error as in the example before.
Figure~\ref{fig:maxwell:tau_plot} shows the result and confirms second-order convergence in time.
Furthermore, the plot shows that a finer mesh size also causes a more severe CFL condition, where the relation is linear.

\begin{figure}
    \begin{center}
        \includeexternalstandalonewithpath[]{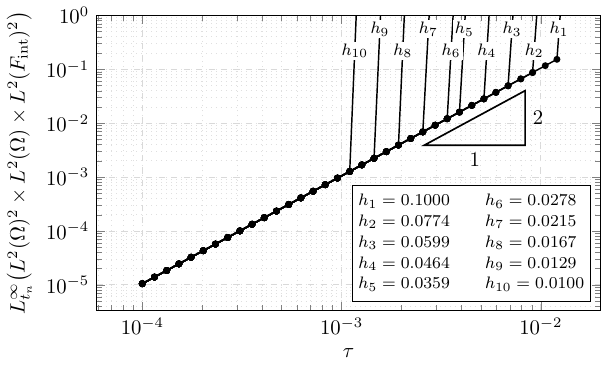}
    \end{center}
    \caption{
        \(L^\infty\)-error in time and the \(L^2(\Omega)^2 \times L^2(\Omega) \times L^2(\Fint)^2\)-error of the discrete solution against the exact solution 
        \eqref{eq:ExactSolutionMaxwellPoly}, plotted against the time step-size \(\tau\).
     }
    \label{fig:maxwell:tau_plot}
\end{figure}

\section*{Acknowledgements}
The authors acknowledge funding by the Deutsche Forschungsgemeinschaft (DFG, German Research Foundation) -- Project-ID 258734477 -- SFB 1173.
JD thanks the French Institute for Research in Computer Science and Automation (Inria) in Lille for the kind hospitality during his research stay. This research stay was financially supported by the Karlsruhe House of Young Scientists (KHYS).

\newpage
\bibliographystyle{amsalpha-doi} %
\bibliography{literature}

\appendix
\section{Computations for the examples}
\label{sec:app_1}

\subsection{Elastic example }\label{appendix:Elastic}

In Section~\ref{app:ana_ell_wave} and \ref{sec:app_thermo_elasto} we suggest a discontinuous Galerkin discretization of the forms
\begin{align*}
    &\ipwave\,:\,  \wavespace\times \wavespace \to \R,\qquad
    &&\ipwave(x,y) 
        = \langle \rho \vbold, \wbold \rangle_{L^2}
       + \langle C^{-1} \sigmabold, \lambdabold \rangle_{L^2},\\
    &\sbilwave\,:\, \wavespaceV\times \wavespaceV \to \R,\qquad 
    &&s(x,y) 
        = -\ip{\div \sigmabold, \wbold}_{L^2} - \ip{\epsilon(\vbold), \lambdabold}_{L^2},
\end{align*}
where \(\wavevar = (\vbold,\sigmabold)\) and \(\wavetestvar = (\wbold,\lambdabold)\) with  \(\wavevar, \wavetestvar\in Y = H^1_0(\Omega,\R^{d}) \times H(\div, \Omega,\R^{d\times d}_{\textup{sym}})\) and \(\epsilon(\vbold) = \Sym(D\vbold)\) denotes the symmetric gradient.
Our presentation follows \cite{Ziegler2020,HocPSTW15}.
Starting from their local discretization, we derive global dG forms and prove the necessary auxiliary result for the analysis of our examples.

We restrict ourselves in the following to dimension \(d=2\) and assume that the elasticity tensor describes an isotropic medium, i.e.
\begin{equation*}
    C w= C(\mu,\lambda) w = 2 \mu w + \lambda \operatorname{tr}(w)I,
\end{equation*}
with Lam\'e parameters \(\lambda \geq 0\), \(\mu > 0\).

Let \(\dGMesh\) be a matching triangulation of \(\Omega\) and denote with \(Y_h\) the space of broken polynomials of degree \(k > 0\) on this mesh.
All problem dependent constants are assumed to be constant on every \(K\in\dGMesh\).
In the following, we introduce a local and a global naming convention of mesh faces and trace-quantities of functions on them.

For any element \(K\in\dGMesh\), we define the set of faces \(F\) associated to \(K\) as \(\dGFaceK\).
For any face \(F\in \dGFaceK\), there is an associated outer unit normal vector \(\dGnKF\) and a unit tangential vector, chosen with respect to the right-hand rule, and denoted with \(\dGtKF\).
The element opposite of \(K\) on the face \(F\), if existent, is denoted with \(K_F\).
Local jumps and averages are defined as
\begin{equation*}
    \dGJmp{f}_{K,F} = 
        \bigrestr{(\restr{f}{K_F})}{F} - \bigrestr{(\restr{f}{K})}{F},
    \qquad
    \dGAvg{f}^\alpha_{K,F}
        = \alpha^{K_F} \bigrestr{(\restr{f}{K_F})}{F}
         + \alpha^{K} \bigrestr{(\restr{f}{K})}{F},
\end{equation*}
where \(\alpha^{K_F},\alpha^K> 0\) denote positive weights such that \(\alpha_{K,F} + \alpha_K = 1\).
The definition is understood element-wise for vector- and tensor-valued functions.

The set of all mesh faces is denoted as \(\dGFace\).
The latter set is further divided in the set \(\dGFaceInt\) of interior faces and \(\dGFaceBnd\) of boundary faces.
For any face \(F\in\dGFaceInt\), we fix a global numbering scheme, i.e., we choose an ordered tuple \((K_1,K_2) \in \dGMesh\times\dGMesh\) such that \(F\in \dGFaceKOne\cap\dGFaceKTwo\).
Based on this numbering, we define the global unit normal and tangential vectors as
\begin{equation}\label{eq:GloablNormalTangential}
    \dGnF = n_{K_1,F} = -n_{K_2,F},
    \qquad
    \dGtF = \tau_{K_1,F} = -\tau_{K_2,F}.
\end{equation}
Similarly, we define global jump and average as
\begin{equation}\label{eq:GlobalJmpAvg}
    \dGJmp{f}_F =
        \dGJmp{f}_{K_1,F} 
            = - \dGJmp{f}_{K_2,F},
    \qquad
    \dGAvg{f}_F^\alpha = 
        \dGAvg{f}_{K_1,F}^\alpha
            = \dGAvg{f}_{K_2,F}^\alpha.
\end{equation}
See Figure~\ref{fig:FaceNamingConvention} for our naming convention.

\begin{figure}
    \begin{tikzpicture}[x=0.9cm,y=0.8cm]
    \def\rightShift{0.1}
\draw[color=black] 
  (5,5) -- (0,0) -- (5,0) -- cycle;
\draw[color=black]
  (5 + \rightShift, 5) -- (5 + \rightShift,0) -- (10 + \rightShift,0) -- cycle;
\node[black, xshift=-7, yshift=-16] at (5,5) {$F$};
\node[black, xshift=28, yshift=12] at (0,0) {$K_1$};
\node[black, xshift=-28, yshift=12] at (10,0) {$K_2$};

\draw[line width=1, -{Latex[width=6pt, length=3pt]}]
    (5, 2) -- (6, 2);
\node[anchor=north west] at (5.35, 1.85) {$n_F = n_{K_1,F}$};
\draw[line width=1, -{Latex[width=6pt, length=3pt]}]
    (5, 2) -- (5, 3.1);
\node[anchor=east] at (4.9, 2.55) {$\tau_F = \tau_{K_1,F}$};
\draw[line width=1, -{Latex[width=6pt, length=3pt]}]
    (5 + \rightShift, 0.9) -- (4 + \rightShift, 0.9);
\node[anchor=south east] at (4.45, 0.95) {$-n_F = n_{K_2,F}$};
\draw[line width=1, -{Latex[width=6pt, length=3pt]}]
    (5 + \rightShift, 0.9) -- (5 + \rightShift, 0.2);
\node[anchor=west] at (5.35, 0.55) {$-\tau_F = \tau_{K_2,F}$};
\end{tikzpicture}
    \caption{Face naming convention with normal and tangential vectors.}
    \label{fig:FaceNamingConvention}
\end{figure}
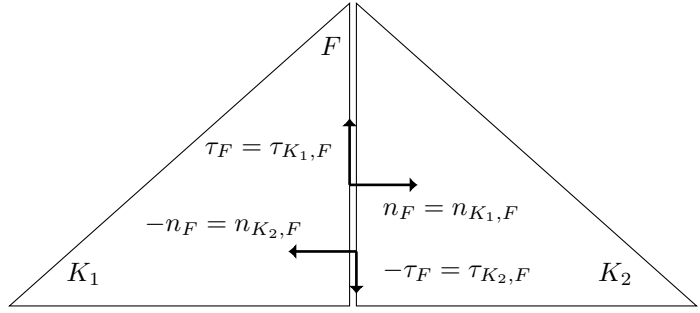

Following the work in \cite{HocPSTW15,Ziegler2020}, we obtain the local definition of the operator \(\waveoph = \waveophskew + \waveophpos\) on every \(K\in \dGMesh\), for \(\wavevarh = (\vboldh, \sigmaboldh)\) and \(\wavetestvarh = (\wboldh,\lambdaboldh)\) with \(\wavevarh,\wavetestvarh\in Y_h\)
\begin{equation}\label{eq:SkewFormLocal}
    \begin{aligned}
        \ip{\waveophskew \wavevarh, &\wavetestvarh}_{\LTwoK}\\
            &= -\ip{\div \sigmaboldh, \wboldh}_{\LTwoK}
                -\ip{ \epsilon(\vboldh), \lambdaboldh}_{\LTwoK}\\
            &\qquad-\sum_{F\in \dGFaceK}
                \begin{aligned}[t]
                    &\alpha_2^{K,F} \ip{\dGnKF \cdot (\dGJmp{\sigmaboldh}_{K,F}\dGnKF),
                        \dGnKF \cdot \wboldh}_{\LTwoF}\\
                    &\quad+\alpha_3^{K,F} \ip{\dGnKF \cdot \dGJmp{\vboldh}_{K,F},
                        \dGnKF \cdot (\lambdaboldh \dGnKF)}_{\LTwoF}
                \end{aligned}\\
            &\qquad-\sum_{F\in \dGFaceK}
                \begin{aligned}[t]
                    &\alpha_6^{K,F} \ip{\dGtKF \cdot (\dGJmp{\sigmaboldh}_{K,F}\dGnKF),
                        \dGtKF \cdot \wboldh}_{\LTwoF}\\
                    &\quad+\alpha_7^{K,F} \ip{\dGtKF \cdot \dGJmp{\vboldh}_{K,F},
                        \dGtKF \cdot (\lambdaboldh \dGnKF)}_{\LTwoF},
                \end{aligned}
    \end{aligned}
\end{equation}
\begin{equation}\label{eq:PosFormLocal}
    \begin{aligned}
        \ip{\waveophpos \wavevarh, &\wavetestvarh}_{\LTwoK}\\
            &=-\sum_{F\in \dGFaceK}
                \begin{aligned}[t]
                    &\alpha_1^{F} \ip{\dGnKF \cdot (\dGJmp{\sigmaboldh}_{K,F}\dGnKF),
                        \dGnKF \cdot (\lambdaboldh \dGnKF)}_{\LTwoF}\\
                    &\quad+\alpha_4^{F} \ip{\dGnKF \cdot \dGJmp{\vboldh}_{K,F},
                        \dGnKF \cdot \wboldh}_{\LTwoF}
                \end{aligned}\\
            &\qquad-\sum_{F\in \dGFaceK}
                \begin{aligned}[t]
                    &\alpha_5^{F} \ip{\dGtKF \cdot (\dGJmp{\sigmaboldh}_{K,F}\dGnKF),
                        \dGtKF \cdot (\lambdaboldh \dGnKF)}_{\LTwoF}\\
                    &\quad+\alpha_8^{F} \ip{\dGtKF \cdot \dGJmp{\vboldh}_{K,F},
                        \dGtKF \cdot \wboldh}_{\LTwoF},
                \end{aligned}
    \end{aligned}
\end{equation}
with constants
\begin{align*}
    \alpha_1^{F} &=
        \frac{1}{\rho_K c_{P,K} + \rho_{K_F} c_{P,K_F}},
    &&\alpha_2^{K,F} =
        \frac{\rho_{K}c_{P,K}}{\rho_{K}c_{P,K} + \rho_{K_F}c_{P,K_F}},\\
    \alpha_3^{K,F} &=
        \frac{\rho_{K_F} c_{P,K_F}}{\rho_K c_{P,K} + \rho_{K_F} c_{P,K_F}},
    &&\alpha_4^{F} =
        \frac{\rho_{K}c_{P,K}\rho_{K_F}c_{P,K_F}}{\rho_{K}c_{P,K} + \rho_{K_F}c_{P,K_F}},\\
    \alpha_5^{F} &=
        \frac{1}{\rho_K c_{S,K} + \rho_{K_F} c_{S,K_F}},
    &&\alpha_6^{K,F} =
        \frac{\rho_{K}c_{S,K}}{\rho_{K}c_{S,K} + \rho_{K_F}c_{S,K_F}},\\
    \alpha_7^{K,F} &=
        \frac{\rho_{K_F} c_{S,K_F}}{\rho_K c_{S,K} + \rho_{K_F} c_{S,K_F}},
    &&\alpha_8^{F} =
        \frac{\rho_{K}c_{S,K}\rho_{K_F}c_{S,K_F}}{\rho_{K}c_{S,K} + \rho_{K_F}c_{S,K_F}},
\end{align*}
where \(c_P\) and \(c_S\) denote the velocity of shear and pressure waves, defined as
\begin{equation*}
    c_P = \sqrt{\frac{2\mu / 3 +\lambda}{\rho}},
    \qquad
    c_S = \sqrt{\frac{\mu}{\rho}}.
\end{equation*} 
On a boundary face \(F\in \dGFaceBnd\) we define \(\dGJmp{\vboldh}_{K,F} = -2\bigrestr{\wavevarhnOne{}}{F}\), \(\dGJmp{\sigmaboldh}_{K,F} = 0\),
\(\rho_{K_F} = \rho_{K}\), \(c_{S,K_F} = c_{S,K}\) and \(c_{P,K_F} = c_{P,K}\).

Summing in \eqref{eq:SkewFormLocal} and \eqref{eq:PosFormLocal} over all elements \(K\in\dGMesh\) while using \eqref{eq:GloablNormalTangential}, \eqref{eq:GlobalJmpAvg} yields the global operators
\begin{equation}\label{eq:SkewFormGlobal} 
    \begin{aligned}
        \ip{\waveophskew \wavevarh, &\wavetestvarh}_{L^2(\Omega)}\\
            &= -\sum_{K\in\dGMesh} \ip{\div \sigmaboldh, \wboldh}_{\LTwoK}
                +\ip{ \epsilon(\vboldh), \lambdaboldh}_{\LTwoK}\\
            &\qquad-\sum_{F\in \dGFaceInt}
                \begin{aligned}[t]
                    &\ip{\dGnF \cdot (\dGJmp{\sigmaboldh}_{F}\dGnF),
                        \dGnF \cdot \dGAvg{\wboldh}_F^{\alpha_2}}_{\LTwoF}\\
                    &\quad+\ip{\dGnF \cdot \dGJmp{\vboldh}_{F},
                        \dGnF \cdot (\dGAvg{\lambdaboldh }_F^{\alpha_3}\dGnF)}_{\LTwoF}\\
                    &\quad+\ip{\dGtF \cdot (\dGJmp{\sigmaboldh}_{F}\dGnF),
                        \dGtF \cdot \dGAvg{\wboldh}_F^{\alpha_6}}_{\LTwoF}\\
                    &\quad+\ip{\dGtF \cdot \dGJmp{\vboldh}_{F},
                        \dGtF \cdot (\dGAvg{\lambdaboldh}_F^{\alpha_7} \dGnF)}_{\LTwoF},
                \end{aligned}\\
            &\qquad+\sum_{F\in \dGFaceBnd}
                \begin{aligned}[t]
                    &\ip{\dGnF \cdot \vboldh, \dGnF \cdot (\lambdaboldh\dGnF)}_{\LTwoF}\\
                    &\quad + \ip{\dGtF \cdot \vboldh, \dGtF \cdot (\lambdaboldh \dGnF)}_{\LTwoF},
                \end{aligned}
    \end{aligned}
\end{equation}
\begin{equation}\label{eq:PosFormGlobal}
    \begin{aligned}
        \ip{\waveophpos \wavevarh, &\wavetestvarh}_{L^2(\Omega)}\\
            &=\sum_{F\in \dGFaceInt}
                \begin{aligned}[t]
                    &\alpha_1^{F} \ip{\dGnF \cdot (\dGJmp{\sigmaboldh}_{F}\dGnF),
                        \dGnF \cdot (\dGJmp{\lambdaboldh}_F \dGnF)}_{\LTwoF}\\
                    &\quad+\alpha_4^{F} \ip{\dGnF \cdot \dGJmp{\vboldh}_{F},
                        \dGnF \cdot \dGJmp{\wboldh}_F}_{\LTwoF}\\
                    &\quad+\alpha_5^{F} \ip{\dGtF \cdot (\dGJmp{\sigmaboldh}_{F}\dGnF),
                        \dGtF \cdot (\dGJmp{\lambdaboldh}_F \dGnF)}_{\LTwoF}\\
                    &\quad+\alpha_8^{F} \ip{\dGtF \cdot \dGJmp{\vboldh}_{F},
                        \dGtF \cdot \dGJmp{\wboldh}_F}_{\LTwoF},
                \end{aligned}\\
            &\qquad+\sum_{F\in \dGFaceBnd}
                \begin{aligned}[t]
                    &\rho_F c_{P,F} \ip{\dGnF \cdot \vboldh, \dGnF \cdot \wboldh}_{\LTwoF}\\
                    &+\quad \rho_F c_{S,F} \ip{\dGtF \cdot \vboldh, \dGtF \cdot \wboldh}_{\LTwoF}.
                \end{aligned}
    \end{aligned}
\end{equation}

\begin{lemma} \label{lem:proof_ass_discr_forms_skew_stab}
    The operators \(\waveophskew\) and \(\waveophpos\), defined in \eqref{eq:SkewFormGlobal} and \eqref{eq:PosFormGlobal} respectively, satisfy Assumption~\ref{ass:discr_forms_skew_stab}.
\end{lemma}
\begin{proof}
    Assumption~\ref{ass:discr_forms_skew_stab}~(a) is shown in \cite[Lem.~3.2]{Ziegler2020}.
    Furthermore, from representation \eqref{eq:PosFormGlobal}, we directly conclude that for any \(\wavetestvar = (\wbold,\lambdabold) \in \wavespaceV\) it holds that
    \[
        \ip{\waveophpos \wavetestvar, \wavetestvarh}_{L^2(\Omega)} = 0, \qquad \text{for all  } \wavetestvarh \in \wavespaceVh,
    \]
    since \(\dGJmp{\wbold}_F = 0\) for \(\wbold\in H^1_0(\Omega)\) and \(\dGJmp{\lambdabold}_F \dGnF = 0\) for \(\lambdabold \in H(\div,\Omega,\R^{d\times d}_{\operatorname{sym}})\).
    Moreover, on any boundary face \(F\in \dGFaceBnd\) it holds that \(\restr{\wbold}{F} = 0\).
    Hence, Assumption~\ref{ass:discr_forms_skew_stab}~(b) is satisfied.
\end{proof}

It remains to proof two auxiliary results.
Let \(\projwaveRef\,:\, \wavespaceV \to \wavespaceVh\) denote the broken \(L^2\)-orthogonal projection on \(\dGMesh\), i.e., for \(\wavevar\in \wavespaceV\) it holds that
\begin{equation*}
    \ip{\projwaveRef \wavevar - \wavevar, \wavetestvarh}_{\LTwoK} = 0,
    \quad \text{ for all } K\in\dGMesh,\,\, \wavetestvarh \in \wavespaceVh.
\end{equation*}
For any \(\wavevar\in \wavespaceV\) with \(\restr{\wavevar}{K} \in H^{k+1}(K)\), \(K\in \dGMesh\), we have the approximation properties, see, e.g., \cite[Sec.~18.4]{Ern.G.2021.FEAI},
\begin{equation} \label{eq:LTwoApproxProp}
    \begin{aligned}
    \|\projwaveRef x - x\|_{\LTwoK} &\leq h_K^{k+1} |x|_{H^{k+1}(K)},\\
    \|\projwaveRef x - x\|_{\LTwoF} &\leq h_K^{k+1/2} |x|_{H^{k+1}(K)},\quad F\in \dGFaceK.
\end{aligned}
\end{equation}
\begin{lemma}\label{lem:EllWaveApproxLemma}
    Let \(\wavevar\in \wavespaceV\) and \(\wavetestvarh\in\wavespaceVh\) with
    \(\restr{\wavevar}{K} \in H^{k+1}(K)\), \(K\in \dGMesh\).
    Then
    \begin{align*}
        |\sbilwaveh(\projwaveRef x - x, \wavetestvarh)|
            &\lesssim 
                h^{k+1/2} \Bigl(\sum_{K\in\dGMesh}|x|_{H^{k+1}(K)}^2\Bigr)^{1/2}
                \stabnormwave{\wavetestvarh}.
    \end{align*}
\end{lemma}

\begin{proof}
    Let  \(\wavevar = (\vbold,\sigmabold)\) and \(\wavetestvarh = (\wboldh,\lambdaboldh)\).
    Since \(\waveophskew\) is skew-adjoint it holds that \(\ip{\waveophskew(\projwaveRef x - x), \wavetestvarh}_{L^2} = -\ip{\projwaveRef x - x, \waveophskew\wavetestvarh}_{L^2}\).
    Hence, for any \(K\in\dGMesh\) the bulk terms in \eqref{eq:SkewFormGlobal} vanish, i.e.
    \begin{equation*}
        \ip{\div \lambdaboldh, J_{h,2} \vbold -\vbold}_{\LTwoK}
            = \ip{\epsilon(\wboldh), J_{h,1}\sigmabold - \sigmabold }_{\LTwoK}
            = 0,
    \end{equation*}
    by the definition of the broken \(L^2\)-projection.
    Thus, \eqref{eq:SkewFormGlobal} only depends on faces, and we obtain with the Cauchy-Schwarz inequality that
    \begin{align*}
        |\ip{\waveophskew(\projwaveRef x - x), &\wavetestvarh}_{L^2}|\\
            &\lesssim \Bigl(\sum_{F\in\dGFace} 
                \|J_{h,1}\vbold - \vbold\|_{\LTwoF}^2 
                + \|J_{h,2}\sigmabold \dGnF - \sigmabold\dGnF\|_{\LTwoF}^2\Bigr)^{1/2}
                \stabnormwave{\wavetestvarh}\\
            &\lesssim h^{k+1/2}\Bigl(\sum_{K\in\dGMesh}|x|_{H^{k+1}(K)}^2\Bigr)^{1/2} \stabnormwave{\wavetestvarh},
    \end{align*}
    where we used the approximation properties \eqref{eq:LTwoApproxProp} in the last estimate.
    A similar estimate shows that
    \begin{align*}
        |\ip{\waveophpos(\projwaveRef x - x), \wavetestvarh}_{L^2}|
            &\lesssim h^{k+1/2}\Bigl(\sum_{K\in\dGMesh}|x|_{H^{k+1}(K)}^2\Bigr)^{1/2} \stabnormwave{\wavetestvarh},
    \end{align*}
    which proves the claim since \(\waveoph = \waveophskew + \waveophpos\).
\end{proof}

\begin{lemma}\label{lem:EllWaveApproxLemmaTwo}

    Let \(\wavevar\in \wavespaceV\), \(\heatvar\in \heatspaceV\) with \(\restr{\wavevar}{K} \in H^{k+1}(K)\), \(K\in \dGMesh\), and \(\heatvar\in H^{\ell+1}(\Omega)\).
    Then
    \begin{align*}
        |\sbilwaveh( \projwaveRef \wavevar -\wavevar, (0,I \cdot (\projheatRef\heatvar - \heatvar)) |
            &\lesssim
                 h^{k + \ell + 1}
                \Bigl(\sum_{K\in\dGMesh}|x|_{H^{k+1}(K)}^2\Bigr)^{1/2} | \heatvar|_{H^{\ell+1}(\Omega)}.
    \end{align*}
\end{lemma}
\begin{proof}
        Since \(\heatvar \in \heatspaceV = H^1_0(\Omega)\) and 
        $\heatspaceVh \subset  H^1_0(\Omega)$,
         we conclude that \((0,I \cdot (\projheatRef\heatvar - \heatvar)) \in \wavespaceV\).
        Furthermore, since 
        \(\heatvar,\projheatRef\heatvar\) are continuous, any jump terms in the stabilization form \eqref{eq:PosFormGlobal} vanish, and we only need to estimate \eqref{eq:SkewFormGlobal}.
        Let \(K\in\dGMesh\).
        Since the first volume term in \eqref{eq:SkewFormGlobal} vanishes, we only estimate the second one and obtain with the inverse inequality and the first interpolation estimate in \eqref{eq:LTwoApproxProp} that
        \begin{align*}
            \sum_{K\in\dGMesh}&|\ip{\epsilon(\projwaveRef^1\vbold - \vbold), I \cdot (\projheatRef\heatvar - \heatvar)}_{L^2(K)}|\\
            &\qquad\qquad\qquad\qquad\lesssim 
                h^k \bigl(\sum_{K\in\dGMesh}|\vbold|_{H^{k+1}(K)^2}^2\bigr)^{1/2}
                \|\projheatRef\heatvar -\heatvar\|_{L^2(\Omega)}\\
            &\qquad\qquad\qquad\qquad\lesssim
                h^{k+\ell+1} \bigl(\sum_{K\in\dGMesh}|\vbold|_{H^{k+1}(K)^2}^2\bigr)^{1/2}|\heatvar|_{H^{k+1}(\Omega)},
        \end{align*}
        where we used the convergence properties of the Ritz-projection \(\projheatRef\) in the last estimate.
        Let \(F\in\dGFaceK\).
        By the multiplicative trace inequality, see, for example, \cite[Lem.~12.15]{Ern.G.2021.FEAI}, it holds for \(\eta\in H^1(K)\) that
        \begin{equation*}
            \|\eta\|_{L^2(F)}
                \lesssim h_K^{-1/2} \|\eta\|_{L^2(K)}
                    + h_K^{1/2} |\eta|_{H^1(K)}
        \end{equation*}
        Hence, \(\eta = \projheatRef \heatvar - \heatvar\) shows together with the convergence properties of the Ritz projection that
        \begin{equation*}
            \bigl(\sum_{F\in\dGFaceInt}\|\projheatRef \heatvar - \heatvar\|_{L^2(F)}^2\bigr)^{1/2}
                \lesssim h^{\ell+1/2} |\heatvar|_{H^{\ell+1}(\Omega)}.
        \end{equation*}
        We apply this equality to the remaining face terms in \eqref{eq:SkewFormGlobal}.
        It holds for \(F\in \dGFaceInt\) that
        \begin{align*}
            \dGnF \cdot (\dGAvg{I \cdot \eta }_F^{\alpha_3}\dGnF)
                &= \dGAvg{I \cdot \eta }_F^{\alpha_3} \dGnF \cdot I \dGnF = \eta,\qquad &&\text{ on } F,\\
            \dGtF \cdot (\dGAvg{I \cdot \eta }_F^{\alpha_3}\dGnF)
                &= \dGAvg{I \cdot \eta }_F^{\alpha_7} \dGtF \cdot I \dGnF = 0,\qquad &&\text{ on } F.
        \end{align*}
        Thus, with the second estimate in \eqref{eq:LTwoApproxProp} we conclude
        \begin{align*}
            \sum_{F\in\dGFaceInt}|\ip{\dGnF \cdot &\dGJmp{\projwaveRef^1\vbold - \vbold}_{F},
                        \dGnF \cdot (\dGAvg{I \cdot \eta }_F^{\alpha_3}\dGnF)}_{\LTwoF}|\\
                &\leq h^{k+\ell+1}\bigl(\sum_{K\in\dGMesh} |\vbold|_{H^{k+1}(K)^2}\bigr)^{1/2}|\theta|_{H^{\ell+1}(\Omega)}.
        \end{align*}
        This proves the claim.
\end{proof}

\subsection{Maxwell's equations with interface coupling}\label{appendix:Maxwell}

In this section we verify Assumption~\ref{ass:cont_forms} for the Maxwell problem of Section~\ref{sec:ana_Maxwell} and prove that the corresponding evolution equation is well-posed.
Compared to Section~\ref{sec:ana_Maxwell} we admit piecewise constant material parameters
\begin{equation}\label{eq:app_mw_materials}
    \restr{\eps}{\Omega_i} \equiv \eps^{(i)} > 0,
    \quad
    \restr{\mu}{\Omega_i} \equiv \mu^{(i)} > 0,
    \qquad i \in \monosetc{1,2},
\end{equation}
i.e., the situation of Section~\ref{sec:ana_Maxwell} is recovered for \(\eps = \mu = 1\).
The governing equations then read
\begin{subequations}\label{eq:app_mw_pde}
\begin{align}
    \eps^{(i)} \pt E^{(i)} - \curl H^{(i)} &= J^{(i)},
        &&\text{in } [0,T]\times \Omega_i,
        \label{eq:app_mw_pde_E}\\
    \mu^{(i)} \pt H^{(i)} + \curl E^{(i)} &= 0,
        &&\text{in } [0,T]\times \Omega_i,
        \label{eq:app_mw_pde_H}\\
    \pt \heatvar + (\heatop\otimes I_2)\heatvar
        &= (c_\ell \otimes I_2) \wavevarOneparE + g,
        &&\text{on } [0,T]\times \Fint,
        \label{eq:app_mw_pde_ode}
\end{align}
supplemented by the perfectly conducting boundary condition and the interface condition
\begin{align}
    E \times \nu &= 0, &&\text{on } [0,T]\times\partial\Omega,\\
    \jmp{H\times \nint}_\Fint &= -(c_\ell^T \otimes I_2)\heatvar,
        &&\text{on } [0,T]\times \Fint.
        \label{eq:app_mw_pde_jump}
\end{align}
\end{subequations}
Throughout, \(\heatop \in \R^{\ell\times\ell}\) satisfies \(\varphi^T\heatop\varphi \geq 0\) for all \(\varphi\in\R^\ell\), i.e., its symmetric part is positive semi-definite, and \(c_\ell \in \monosetc{0,1}^\ell\), \(c_\ell\neq 0\).

We write \(\jmp{H_\tang} = \jmp{H\times\nint}_\Fint\) for the tangential jump and \(\wavevarOneparE\) for the tangential component of \(E\) on \(\Fint\).

The spaces are the ones of Section~\ref{sec:ana_Maxwell}.
For convenience, we repeat them here.
With \(\wavevar = (\wavevarOne,\wavevarTwo) = (E,H)\) we set
\begin{equation}\label{eq:app_mw_spaces}
    \wavespace = L^2(\Omega)^6,
    \quad
    \wavespaceV = H^{\imp}_0(\curl) \times H^\imp(\pcurl),
    \qquad
    \heatspace = \heatspaceV = L^2(\Fint)^{2\ell},
\end{equation}
with the curl-spaces of improved tangential regularity
\begin{align*}
    H^\imp_0(\curl) &= \setc{E\in L^2(\Omega)^3}{\curl E\in L^2(\Omega)^3,\, \wavevarOneparE \in L^2(\Fint)^2,\, E\times\nu = 0},\\
    H^\imp(\pcurl) &= \setc{H\in L^2(\Omega)^3}{\pcurl H \in L^2(\Omega)^3,\, \jmp{H_\tang}_\Fint \in L^2(\Fint)^2}.
\end{align*}
Here, \(\pcurl\) denotes the piecewise curl, i.e., \(\pcurl H \in L^2(\Omega)^3\) if and only if \(\curl H^{(i)} \in L^2(\Omega_i)^3\) for \(i \in\monosetc{1,2}\).
Both spaces are equipped with their graph norm.
The domain \(\fullspace\) of Assumption~\ref{ass:cont_forms}~(c) is specified in Corollary~\ref{cor:app_mw_ass_cont_forms} below.

The forms are as in Section~\ref{sec:ana_Maxwell} as well, except that the inner products carry the material weights,
\begin{equation}\label{eq:app_mw_forms}
    \ipwave(x,y) = \ip{\eps x_1, y_1} + \ip{\mu x_2, y_2},
    \qquad
    \ipheat(\heatvar,\varphi) = \ip{\heatvar,\varphi}_{L^2(\Fint)^{2\ell}} ,
\end{equation}
and \(\ipwave(f,y) = \ip{J,y_1}\), i.e., \(f = (\eps^{-1}J,0)\).
Since \(\eps\) and \(\mu\) are bounded from above and below by positive constants, \(\ipwave\) induces a norm which is equivalent to the one of \(L^2(\Omega)^6\).
Hence, part~(a) of Assumption~\ref{ass:cont_forms} holds.
The unweighted inner product on \(\Omega\) is still denoted with \(\ip{\cdot,\cdot}\).
Recall further the bilinear forms
\begin{equation}\label{eq:app_mw_bilforms}
\begin{aligned}
    \sbilwave(x,y) &= \ip{\curl x_1, y_2} - \ip{x_2,\curl y_1},
    &&
    \sbilheat(\heatvar,\varphi) = \ipheat\bigl((\heatop\otimes I_2)\heatvar,\varphi\bigr),
    \\
    \coupopfullform(\varphi,x) &= \ipheat\bigl(\varphi, (c_\ell\otimes I_2)\wavevarOnepar\bigr),
\end{aligned}
\end{equation}
for \(x,y \in \wavespaceV\) and \(\heatvar,\varphi\in\heatspaceV\).
Part~(b) of Assumption~\ref{ass:cont_forms} follows from
\begin{equation*}
    \sbilwave(y,y) = \ip{\curl y_1, y_2} - \ip{y_2, \curl y_1} = 0,
    \qquad
    \sbilheat(\varphi,\varphi) = \ip{(\heatop\otimes I_2)\varphi,\varphi}_{L^2(\Fint)^{2\ell}} \geq 0 .
\end{equation*}
Part~(c) of Assumption~\ref{ass:cont_forms} is an integration-by-parts formula.
Such a formula is available in \cite[Thm.~2.2]{DoeNS26}, but only on the divergence-free energy space.
Assumption~\ref{ass:cont_forms}~(c), however, has to hold for all test functions \(y \in \wavespaceV\).
The gap between the two statements is bridged by the Hodge projections of Section~\ref{app:mw:hodge}.
The same projectors allow us to treat sources \(J \in L^2(\Omega)^3\), which do not respect Gauss's law and therefore leave the divergence-free setting of \cite{DoeNS26}.

\subsubsection{Hodge projections}
\label{app:mw:hodge}

Following \cite{DoeNS26}, we use the divergence-free spaces
\begin{align*}
    \divspaceE &= \setc{V\in L^2(\Omega)^3}{\pdiv (\eps V) = 0},\\
    \divspaceH &= \setc{V\in L^2(\Omega)^3}{\div(\mu V) = 0,\, \mu V \cdot \nu = 0 \text{ on } \partial\Omega},
\end{align*}
where \(\pdiv\) denotes the piecewise divergence, i.e., \(\pdiv W = 0\) if and only if \(\div W^{(i)} = 0\) for \(i \in \monosetc{1,2}\), and the associated energy space
\begin{equation*}
    \energyspace
    =
    \bigl( H^\imp_0(\curl) \cap \divspaceE \bigr)
    \times
    \bigl( H^\imp(\pcurl) \cap \divspaceH \bigr) .
\end{equation*}
Note that no condition is imposed on the normal jump in \(\divspaceE\).
The interface condition \eqref{eq:app_mw_pde_jump} already constrains this degree of freedom via a conservation law, cf. \cite[Thm.~2.6]{DoeNS26}.

Since both \(\divspaceE\) and \(\divspaceH\) are closed subspaces of \(L^2(\Omega)^3\), there exist orthogonal projections
\begin{equation*}
    \projE\colon L^2(\Omega)^3 \to \divspaceE,
    \qquad
    \projH\colon L^2(\Omega)^3 \to \divspaceH,
\end{equation*}
with respect to the weighted inner products \(\ip{\eps\cdot,\cdot}\) and \(\ip{\mu\cdot,\cdot}\), respectively.
The following lemma collects the properties we need.
See Lemma~3.1 of \cite{DoeNS26} for the proof.

\begin{lemma}\label{lem:app_mw_hodge}
    \noindent{\normalfont(a)}
    For \(y\in H^\imp_0(\curl)\) it holds
    \begin{equation*}
        \projE y \in H_0^\imp(\curl) \cap \divspaceE,
        \qquad
        \curl \projE y = \curl y,
        \qquad
        (\projE y)_\parallel = y_\parallel.
    \end{equation*}
    \noindent{\normalfont(b)}
    For \(v\in H^\imp(\pcurl)\) it holds
    \begin{equation*}
        \projH v \in H^\imp(\pcurl) \cap \divspaceH,\qquad
        \pcurl \projH v = \pcurl v,
        \qquad
        \jmp{(\projH v)_\tang} = \jmp{v_\tang}.
    \end{equation*}
\end{lemma}

We can now extend the integration-by-parts formula \cite[Thm.~2.2]{DoeNS26} to the full spaces \(\wavespaceV = H_0^\imp(\curl) \times H^\imp(\pcurl)\), i.e., without any constraint on the divergence.
\begin{lemma}\label{lem:app_mw_ibp}
    For all \(y \in H^\imp_0(\curl)\) and \(v \in H^\imp(\pcurl)\) we have
    \begin{equation}\label{eq:app_mw_ibp}
        \ip{y, \pcurl v} - \ip{\curl y, v}
        = \ip{y_\parallel, \jmp{v_\tang}}_{L^2(\Fint)^2} .
    \end{equation}
\end{lemma}

\begin{proof}
    With Lemma~\ref{lem:app_mw_hodge}~(b) we obtain
    \begin{equation*}
        \ip{y,\pcurl v} 
        = \ip{y, \pcurl \projH v}
        = \ip{\eps y, \eps^{-1}\pcurl \projH v},
    \end{equation*}
    where we used \(\eps^{-1}\eps = 1\) in the latter step.
    Since \(\eps^{-1}\pcurl v \in \divspaceE\), we can introduce the projection in the first argument, and obtain
    \begin{equation*}
        \ip{y,\pcurl v} 
        = \ip{\eps \projE y, \eps^{-1}\pcurl v}
        = \ip{\projE y, \pcurl \projH v}.
    \end{equation*}
    Similarly, with Lemma~\ref{lem:app_mw_hodge}~(a) we conclude that
    \begin{equation*}
        \ip{\curl y, v} = \ip{\curl \projE y, \projH v}.
    \end{equation*}
    
    On the other hand, with Lemma~\ref{lem:app_mw_hodge}~(a) and~(b), we see that
    \begin{equation*}
        \ip{y_\parallel, \jmp{v_\tang}}_{L^2(\Fint)^2}
        =
        \ip{(\projE y)_\parallel, \jmp{(\projH v)_\tang}}_{L^2(\Fint)^2} .
    \end{equation*}
    As \((\projE y, \projH v) \in \energyspace\), the claim now follows from \cite[Thm.~2.2]{DoeNS26}.
\end{proof}

\begin{corollary}\label{cor:app_mw_ass_cont_forms}
    Assumption~\ref{ass:cont_forms} is satisfied with
    \begin{equation*}
        \fullspace = \setc{(x,\heatvar) \in \wavespaceV\times \heatspaceV}{\jmp{x_2\times\nint}_\Fint = -(c_\ell^T\otimes I_2)\heatvar},
        \qquad
        \coupopfullformext \equiv 0,
    \end{equation*}
    and \(\sbilwaveext(x,y) = \ip{\curl x_1, y_2} - \ip{\pcurl x_2, y_1}\).
\end{corollary}

\begin{proof}
    Parts~(a) and~(b) have been shown above.
    For \((x,\heatvar) \in \fullspace\) and \(y \in \wavespaceV\), Lemma~\ref{lem:app_mw_ibp} applied to \(y_1 \in H^\imp_0(\curl)\) and \(x_2 \in H^\imp(\pcurl)\) yields
    \begin{align*}
        \sbilwave(x,y) - \sbilwaveext(x,y)
        &= \ip{\pcurl x_2, y_1} - \ip{x_2,\curl y_1}\\
        &= \ip{ \jmp{x_2 \times \nint}_\Fint , y_{1,\parallel}}_{L^2(\Fint)^2}
        = - \coupopfullform(\heatvar, y),
    \end{align*}
    where we used the interface condition in \(\fullspace\) and \(\ip{(c_\ell^T \otimes I_2)\heatvar, y_{1,\parallel}} = \ip{\heatvar, (c_\ell\otimes I_2)y_{1,\parallel}}\) in the last step.
\end{proof}

\subsubsection{Well-posedness}
\label{app:mw:wellposedness}

We now write \eqref{eq:app_mw_pde} as an evolution equation on \(\wavespace\times\heatspace\) and define
\begin{equation}\label{eq:app_mw_operator}
    \maxop \colon \wavespace\times\heatspace \supset \fullspace \to \wavespace\times\heatspace,
    \qquad
    \maxop (x,\heatvar)
    =
    \begin{pmatrix}
        -\eps^{-1} \pcurl x_2 \\
        \mu^{-1}\curl x_1\\
        (\heatop \otimes I_2)\heatvar - (c_\ell\otimes I_2) x_{1,\parallel}
    \end{pmatrix},
\end{equation}
with \(\fullspace\) from Corollary~\ref{cor:app_mw_ass_cont_forms}, so that \eqref{eq:app_mw_pde} reads
\begin{equation}\label{eq:app_mw_evolution}
    \pt (x,\heatvar) + \maxop (x,\heatvar) = (\eps^{-1}J, 0, g),
    \qquad
    (x,\heatvar)(0) = (x_0,\heatvar_0).
\end{equation}
Note that \(\maxop\) is well-defined on \(\fullspace\).
The interface condition is meaningful since \(\jmp{x_2\times \nint}_\Fint \in L^2(\Fint)^2\) for \(x_2 \in H^\imp(\pcurl)\).

In contrast to \cite{DoeNS26}, no divergence constraints are imposed.
This allows for sources \(J\in L^2(\Omega)^3\).

We show well-posedness of the evolution equation \eqref{eq:app_mw_evolution} by means of the Lumer--Phillips theorem, see, e.g., \cite[Cor.~II.3.20]{Eng.Nag.2000}.
Thus, it remains to show that \(\maxop\) is maximal accretive.
\begin{lemma}\label{lem:app_mw_accretive}
    The operator \(\maxop\) is accretive with
    \begin{equation}\label{eq:app_mw_dissipation}
        \ipcol{\maxop(x,\heatvar)}{(x,\heatvar)}
        =
        \ip{(\heatop\otimes I_2) \heatvar,\heatvar}_{L^2(\Fint)^{2\ell}}
        \geq 0,
        \qquad (x,\heatvar) \in \fullspace.
    \end{equation}
\end{lemma}

\begin{proof}
    Let \((x,\heatvar) \in \fullspace\).
    By the definitions \eqref{eq:app_mw_forms} and \eqref{eq:app_mw_operator},
    \begin{align*}
        &\ipcol{\maxop(x,\heatvar)}{(x,\heatvar)}\\
        &\qquad\qquad=
        -\ip{\pcurl x_2, x_1} + \ip{\curl x_1, x_2}
        + \ip{(\heatop\otimes I_2)\heatvar,\heatvar}
        - \ip{(c_\ell \otimes I_2) x_{1,\parallel},\heatvar},
    \end{align*}
    where the material weights cancel.
    Lemma~\ref{lem:app_mw_ibp} and the interface condition in \(\fullspace\) give
    \begin{align*}
        -\ip{\pcurl x_2, x_1} + \ip{\curl x_1, x_2}
        &= -\ip{x_{1,\parallel},\jmp{x_{2,\tang}}}\\
        &= \ip{x_{1,\parallel}, (c_\ell^T\otimes I_2)\heatvar}
        = \ip{(c_\ell\otimes I_2)x_{1,\parallel}, \heatvar},
    \end{align*}
    so that the two coupling terms cancel and \eqref{eq:app_mw_dissipation} follows.
\end{proof}

It remains to prove that \(\maxop\) is maximal.
We first record the modification of \cite[Lem.~2.4]{DoeNS26} that is needed, namely an inhomogeneous impedance condition on the interface.
\begin{lemma}\label{lem:app_mw_resolvent_interface}
    Let \(\sigma > 0\), \(\kappa \in L^2(\Fint)^2\) and \((F,G) \in \divspaceE\times\divspaceH\).
    Then there exists a unique \((E,H) \in \energyspace\) with
    \begin{equation}\label{eq:app_mw_resolvent_interface}
        E - \eps^{-1}\pcurl H = F,
        \qquad
        H + \mu^{-1}\curl E = G,
        \qquad
        \jmp{H_\tang} = -\sigma \wavevarOneparE - \kappa .
    \end{equation}
\end{lemma}
The proof follows along the lines of \cite[Lem.~2.4]{DoeNS26}, and we omit the details here.
We can now prove that \(\maxop\) is maximal.
\begin{lemma}\label{lem:app_mw_maximal}
    The operator \(I + \maxop \colon \fullspace \to \wavespace\times\heatspace\) is bijective.
\end{lemma}

\begin{proof}
    Let \((F,G,\psi) \in \wavespace\times\heatspace\).

    Since \(\varphi^T(I_\ell + \heatop)\varphi \geq |\varphi|^2\) for all \(\varphi \in \R^\ell\), the matrix \(I_\ell + \heatop\) is invertible, and
    \(v = (I_\ell+\heatop)^{-1}c_\ell \neq 0\) satisfies
    \(c_\ell^T v = v^T(I_\ell + \heatop)v > 0\).
    We therefore set
    \begin{equation}\label{eq:app_mw_sigma_eff}
        \sigmaeff = c_\ell^T (I_\ell + \heatop)^{-1} c_\ell > 0,
        \qquad
        \kappa = \bigl( c_\ell^T (I_\ell+\heatop)^{-1}\otimes I_2\bigr)\psi \in L^2(\Fint)^2 .
    \end{equation}

    We decompose \(F = \projE F + F_{\mathrm{R}}\) and \(G = \projH G + G_{\mathrm{R}}\).
    Let \((E^0,H^0) \in \energyspace\) be the solution of \eqref{eq:app_mw_resolvent_interface} with data \((\projE F,\projH G)\), \(\sigma = \sigmaeff\) and \(\kappa\).
    We set
    \begin{equation*}
        x = (E^0 + F_{\mathrm{R}},\, H^0 + G_{\mathrm{R}}),
        \qquad
        \heatvar = \bigl((I_\ell + \heatop)^{-1}\otimes I_2\bigr)\bigl(\psi + (c_\ell\otimes I_2) x_{1,\parallel}\bigr) .
    \end{equation*}
    Lemma~\ref{lem:app_mw_hodge} shows together with an  approximation argument that
    \begin{equation*}
        \curl F_{\mathrm{R}} = 0,
        \quad (F_{\mathrm{R}})_\parallel = 0,
        \qquad
        \pcurl G_{\mathrm{R}} = 0,
        \quad
        \jmp{(G_{\mathrm{R}})_\tang} = 0.
    \end{equation*}
    Hence, \(x \in \wavespaceV\) since
    \begin{equation*}
        \curl x_1 = \curl E^0,
        \quad
        x_{1,\parallel} = E^0_\parallel,
        \qquad
        \pcurl x_2 = \pcurl H^0,
        \quad
        \jmp{x_{2,\tang}} = \jmp{H^0_\tang} .
    \end{equation*}
    Consequently, the third component of \eqref{eq:app_mw_resolvent_interface} and \eqref{eq:app_mw_sigma_eff} yield
    \begin{equation*}
        -(c_\ell^T\otimes I_2)\heatvar
        = -\kappa - \sigmaeff x_{1,\parallel}
        = \jmp{H^0_\tang}
        = \jmp{x_{2,\tang}},
    \end{equation*}
    i.e.\ \((x,\heatvar) \in \fullspace\).
    Moreover, the first two components of \eqref{eq:app_mw_resolvent_interface} give
    \begin{align*}
        x_1 - \eps^{-1}\pcurl x_2 
            &= E^0 - \eps^{-1}\pcurl H^0 + F_{\mathrm{R}}
            = \projE F + F_{\mathrm{R}} 
            = F,\\
        x_2 + \mu^{-1}\curl x_1 
            &= H^0 + \mu^{-1}\curl E^0 + G_{\mathrm{R}} 
            = \projH G + G_{\mathrm{R}} 
            = G,
    \end{align*}
    and the definition of \(\heatvar\) gives \(\heatvar + (\heatop\otimes I_2)\heatvar - (c_\ell\otimes I_2)x_{1,\parallel} = \psi\).
    Hence \((I+\maxop)(x,\heatvar) = (F,G,\psi)\).
    Injectivity follows from Lemma~\ref{lem:app_mw_accretive}, since \eqref{eq:app_mw_dissipation} implies
    \(\normcol{(I+\maxop)(x,\heatvar)} \geq \normcol{(x,\heatvar)}\).
\end{proof}

The Lumer--Phillips theorem, see, e.g., \cite[Cor.~II.3.20]{Eng.Nag.2000}, therefore yields the following result.
\begin{theorem}\label{thm:app_mw_wellposedness}
    Let \(\eps,\mu\) be as in \eqref{eq:app_mw_materials}, let \(\heatop\in\R^{\ell\times\ell}\) have a positive semi-definite symmetric part, and let \(c_\ell \in \monosetc{0,1}^\ell\setminus\monosetc{0}\).
    Then \(-\maxop\) generates a \(C_0\)-semigroup of contractions on \(\wavespace\times\heatspace\).
    In particular, for all
    \begin{equation*}
        (x_0,\heatvar_0) \in \fullspace,
        \qquad
        J \in C^1([0,T], L^2(\Omega)^3),
        \qquad
        g \in C^1([0,T],\heatspace),
    \end{equation*}
    problem \eqref{eq:app_mw_evolution} possesses a unique classical solution
    \begin{equation*}
        (x,\heatvar) \in C^1\bigl([0,T],\wavespace\times\heatspace\bigr)
        \cap C\bigl([0,T],\fullspace\bigr),
    \end{equation*}
    which satisfies \eqref{eq:full_system}.
\end{theorem}

\begin{proof}
    The first statement follows from Lemma~\ref{lem:app_mw_accretive} and Lemma~\ref{lem:app_mw_maximal}.
    A standard convolution argument then yields a classical solution for the inhomogeneous Cauchy problem, cf. \cite[Thm.~2.4]{Paz.1983.SLOA}.
\end{proof}

\end{document}